%% file: oplaxenriched.tex
\pdfoutput=1

\documentclass[11pt, a4paper]{article}

\usepackage[draft]{dl-en}
\usepackage{oplaxenriched-config}
\include{macros}

\begin{document}
\maketitle

\begin{abstract}
	We define a notion of category enriched
	over an oplax monoidal category \( \Category V \),
	extending the usual definition
	of category enriched over a monoidal
	category.
	Even though oplax monoidal structures
	involve
	infinitely many `tensor product' functors
	\( \Category V^n \to \Category V \),
	the definition of categories
	enriched over \( \Category V \)
	only requires the lower arity maps 
	\( (n \leq 3) \),
	similarly to the monoidal case.
	
	The focal point of the enrichment theory shifts,
	in the oplax case,
	from the notion of
	\( \Category V \)\=/category
	(classically given by collections of
	objects and hom-objects together with composition and unit
	maps) to the one
	of categories enriched over \( \Category V \)
	(genuine categories
	equipped with additional structures).

	One of the merits of the notion of categories 
	enriched over \( \Category V \)
	is that
	it becomes straightforward to define
	both enriched functors and enriched natural transformations.
	We show moreover that the resulting 2-category
	\( \Cat_{\Category V} \)
	can be put in correspondence (via the theory of distributors)
	with the 2-category of modules over \( \Category V \).

	We give an example of such an enriched category in the framework
	of operads: every cocomplete symmetric monoidal
	category \( \Category C \) is enriched over the
	category of sequences in \( \Category C \)
	endowed with an oplax monoidal structure
	stemming from the usual operadic composition
	product, whose monoids are still the (planar) operads.

	As an application of the study of the 2-functor
	\( \Category V \mapsto \Cat_{\Category V} \),
	we show that when
	\( \Category V \) is also endowed with a
	compatible lax monoidal structure---thus forming a lax-oplax duoidal
	category---the \( 2 \)\=/category
	\( \Cat_{\Category V} \) inherits 
	a lax \( 2 \)\=/monoidal structure, thereby generalising
	the corresponding result when the enrichment base is a braided monoidal
	category.
	We illustrate this result by discussing in details
	the lax-oplax structure on
	the category of \( (\Ring\Enveloping, \Ring\Enveloping) \)\=/bimodules,
	whose bimonoids are the bialgebroids.
	
	We conclude by commenting on the relations between
	the enrichment theory over oplax monoidal categories
	and other enrichment theories (monoidal, multicategories,
	skew and lax).
\end{abstract}

\tableofcontents

\footnotefirstpage

\section*{Introduction}

In the 1960s arose the first formalisations of the notion of enriched
categories
\cite{doi:10.4153/cjm-1965-076-0,
zbMATH03220364},
where the set of arrows between
two objects in a category could bear additional structure.
Given two objects \( x, y \), one would associate a mapping object
\( [x,y] \) of any other external `base category'
\( \Category V \),
provided that this category is endowed with a tensor structure,
necessary to define the composition morphisms
\[
	[y, z] \otimes [x, y] \longrightarrow [x, z]
\]
and composition units \( 1_\otimes \to [x, x] \), which
plays the rôle of
the usual product structure on the category of sets.
In addition, the regular set of `morphisms'
between two objects \( x \) and \( y \)
could be recovered
as the set of maps \( 1_\otimes \to [x,y] \) in \( \Category V \).

The notion of enriched category then got generalised through the years,
allowing enrichments over multicategories, double categories etc.
This culminated at the end of the century with the definition by
Leinster of `the most general structure' one could use to
enrich a category: fc-multicategories
\cite{arXiv:9901139},
which are
`very general kinds of two-dimensional structures,
encompassing bicategories, monoidal categories,
double categories and ordinary multicategories'.

Here we shall go in the opposite direction and develop an enrichment
theory which is only slightly more general than the classical one:
enrichment over an oplax monoidal category. Being only slightly more
general,
it has the benefit that the theory is mostly similar
to the traditional case.

In a tensor category, the symbol \( x \otimes y \otimes z \) is
formally undefined but tacitly assumed to mean \( (x \otimes y) \otimes z \)
(the other choice being possible as well), in which case one has
an isomorphism \( x \otimes y \otimes z \IsIsomorphicTo x \otimes (y
\otimes z) \) which is simply the associator.
In an oplax monoidal category, the symbol \( x \otimes y \otimes z \)
is supplied together with two decomposition maps
\[
	(x \otimes y) \otimes z
	 \longleftarrow x \otimes y \otimes z
	 \longrightarrow x \otimes (y \otimes z) 
\]
which are not assumed to be invertible.
These non-invertible maps are all that is required
to formulate the associativity axiom of enriched categories
as a symmetric hexagon
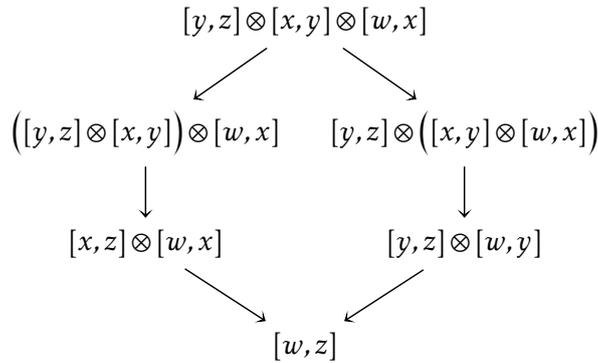
\[
	\begin{tikzcd}
		&[-15ex]
		\EnrichedHom y z \otimes
		\EnrichedHom x y \otimes
		\EnrichedHom w x
		\arrow[dl]
		\arrow[dr] &[-15ex]
		\\
		\big(\EnrichedHom y z \otimes \EnrichedHom x y\big)
		\otimes \EnrichedHom{w}{x}
		\arrow[d]
		&&
		\EnrichedHom{y}{z} \otimes
		\big(\EnrichedHom x y \otimes \EnrichedHom w x\big)
		\arrow[d]
		\\
		\EnrichedHom x z \otimes \EnrichedHom w x
		\arrow[dr]
		&& \EnrichedHom y z \otimes \EnrichedHom w y
		\arrow[dl]
		\\
		& \EnrichedHom{w}{z} &
	\end{tikzcd}
\]
thus replacing the usual---asymmetric---pentagon.

Even though the definition of an oplax monoidal structure involves
a 4-ary symbol \( w \otimes x \otimes y \otimes z \), a 5-ary symbol
et cætera up to infinity,
we shall explain why the oplax nature of the structure of
\( \Category V \) makes the use of the \( n \)-ary symbols
redundant for \( n > 3 \) when it comes to enrichment
\UnskipRef{prop:truncation}.
Thanks to this, the definition of a category enriched over
an oplax monoidal category remains very akin to the usual case.

The main departure
from the usual theory concerns the necessary distinction
between the concept
of \( \Category V \)-category and the more relevant concept of
`category enriched over \( \Category V \)'.
In the second case, one starts with a
genuine category
\( \Category C \) which one then endows with a bifunctor
\[
	\begin{tikzcd}
		\Category C\Op \times \Category C
		\rar["{[-,-]}"]
			& \Category V
	\end{tikzcd}
\]
thus yielding two distinct notions of morphisms:
the usual morphisms of
\( \Category C \) (the strong morphisms) versus the maps of the form
\( 1_\otimes \to [x,y] \) (the weak morphisms).
In this regard, a category enriched over \( \Category V \) can really be
thought of as `a category with an extra enrichment structure'
thus making the definition of the 2-category \( \Cat_{\Category V} \)
of categories enriched over \( \Category V \)
straightforward to articulate.

Contradistinctly, working with
\( \Category V \)-categories---created by assigning objects
\( [x, y] \in \Category V \) to each pair of objects \( x, y \) etc---one
 is quickly led to the problem that the assignment
\( (x, y) \mapsto [x, y] \) is not generically functorial
and that as a result, there is no obvious notion
of \( \Category V \)-natural transformation between
\( \Category V \)-functors that one could use to build a 2-category
\( \VCats \).

This chasm between \( \VCats \) and \( \Cat_{\Category V} \)
was already found by Campbell in the skew context
\cite{doi:10.1007/s10485-017-9504-0}.
It is not
of a mere technical nature but is also justified from a theoretical standpoint.
Whenever \( \Category V \) is endowed with a monoidal structure,
\( \VCats \) can be put in correspondence---via the use
of distributors---with the 2-category
of \( \Category V \)-modules
(also called `\( \Category V \)\=/actegories')
\cite{zbMATH01687308}, where
the denomination `module' here refers to the fact that monoidal categories
correspond to pseudo-monoids in the 2-category \( \Cat \).
If  \( \Category V \) is instead endowed
with an oplax monoidal structure,
it can then be seen as an oplax monoid in \( \Cat \)
\cite{zbMATH01963527} making it natural
to consider its 2-category of modules.
Making use of distributors again to build a correspondence
with \( \Category V \)-modules on one side, what one obtains
on the other side is precisely 
the 2-category \( \Cat_{\Category V} \) rather than \( \VCats \).

As an application of the general formalism, the present paper provides two examples of this enrichment theory which were
motivational to the authors.
The first example occurs within the theory of operads in a symmetric monoidal
category \( \Category C \)
\SectionRef{sec:operads}.
When the monoidal structure on \( \Category C \) is closed,
a theorem by Kelly
\cite{MR2177746},
central to the theory of operads, asserts that
(planar) operads in \( \Category C \) can be realised as the monoids
in the category of sequences in \( \Category C \)
endowed with an operadic-composition
monoidal structure.
However, when the monoidal structure on \( \Category C \) is not
closed, the operadic-composition no longer
necessarily defines a monoidal structure. 
Rather, the operadic-composition
induces an oplax monoidal structure on \( \Category C \), as
shown by Ching
\cite{doi:10.1007/s40062-012-0007-2},
whose monoids are still the (planar) operads.
The base category \( \Category C \) then becomes enriched over
the oplax monoidal category of sequences via a typical formula from
the theory of operads.
This setup is at the hearth of the theory since an algebra
over an operad in \( \Category C \) simply becomes a representation
of a monoid in the enriching oplax monoidal category.

The second example deals with the formalisation of the theory
of (non-commutative) bialgebroids
\SectionRef{sec:Re-bimod}.
Given a ring \( R \), let \( R\Enveloping \) denote the enveloping
ring \( R \otimes R\Op \).
Following Takeuchi
\cite{doi:10.2969/jmsj/02930459}, a bialgebroid can be defined as
a \( (R\Enveloping, R\Enveloping) \)\=/bimodule, denoted \( A \), together with
structure maps among which two structure maps
\[
	A \otimes_{R\Enveloping} A \longrightarrow A
	\qand
	A \longrightarrow A \times_R A \subset A \otimes_R A
\]
enjoying compatibility conditions resembling the ones of a bialgebra.
On the one hand the category
\( \BimodulesOn{R\Enveloping} \)
of
\( (R\Enveloping, R\Enveloping) \)-bimodules
is endowed with a genuine tensor structure given by
\( \otimes_{R\Enveloping} \), so that a bialgebroid
is in particular a monoid in this monoidal category.
On the other hand, the restricted tensor product
\( \times_R \subset \otimes_R \)
due to Sweedler
\cite{doi:10.1007/BF02685882}
and Takeuchi
\cite{doi:10.2969/jmsj/02930459}
is not associative.
We shall show that it endows
\( \BimodulesOn{R\Enveloping} \) with a lax monoidal
structure, of which only the first stages were described by
Takeuchi.
In addition, we shall show that the two structures
\( \otimes_{R\Enveloping} \) and \( \times_R \) possess natural
compatibility structures endowing
\( \BimodulesOn{R\Enveloping} \)
with a lax-oplax duoidal structure.
This general structure, involving both a lax monoidal structure and
an oplax monoidal structure, constitutes a natural generalisation of the notion of
duoidal structure in which the notion of bimonoid still makes sense \cite{doi:10.1090/conm/771}.

We shall then be able to say that:

\begin{center}
`Bialgebroids are the bimonoids
in the lax-oplax duoidal category of
\( (R\Enveloping, R\Enveloping) \)\=/bimodules.'
\end{center}

Lastly, the study of the 2-functor
\( \Category V \mapsto \Cat_{\Category V} \)
mapping each oplax monoidal category \( \Category V \)
to its corresponding 2-category of categories
enriched over \( \Category V \)
\SectionRef{sec:V to Cat_V} will allow us to claim that,
for every lax-oplax duoidal category \( \Category V \),
its associated 2-category \( \Cat_{\Category V} \) is
canonically endowed with a lax monoidal structure, thus extending
a well-known result about the 2-category of categories enriched
over a braided monoidal category.
Applied to \( \BimodulesOn{\Ring\Enveloping} \),
the 2-category
of categories enriched over \( \BimodulesOn{\Ring\Enveloping} \)
becomes endowed with a
lax monoidal structure stemming from the restricted
tensor product \( \times_R \).
A natural example of such a category
is the
category of \( (R, R) \)-bimodules.
More generally, the category of modules over a bialgebroid is
canonically enriched over the lax-oplax duoidal category of
\( (R\Enveloping, R\Enveloping) \)\=/bimodules,
which the authors expect should lead to new
developments in the theory of bialgebroids.

\section*{Notations for integral calculus}
\label{sec:notation}
\begin{description}
\item[Integrals]
	Following Yoneda's original notation
	\cite{On_Ext_and_exact_sequences},
	we shall denote
	\[
		\Integral{b} F(b,b)
	\]
	the integral of a functor
	\( F \From \Category B\Op \times \Category B \to \Category C \).
	This convention is the opposite of the Australian convention where
	subscripts usually denote cointegrals (`ends') 
	instead of integrals (`coends').
	We shall however denote the cointegral of \( F \) using the
	superscript
	\[
		\int^b F(b,b)
	\]
	which differs from both Yoneda's and the Australian school's
	conventions, but is instead standard in the literature
	about bialgebroids
	\cite{doi:10.1007/BF02685882,
	doi:10.2969/jmsj/02930459, doi:10.1023/A:1008608028634}
	which we shall discuss in a dedicated section
	\SectionRef{sec:Re-bimod}.
\item[Sequence of objects]
	We shall denote a sequence of \( n+1 \) objects
	in a category by
	\[
		\Collection x \coloneqq (x_0, \dots, x_n)
	\]
	and
	\[
		\Collection x \oplus \Collection y
		= (x_0, \dots, x_m, y_0, \dots, y_n)
	\]
	for the union of two sequences \( \Collection x \)
	and \( \Collection y \).
	We shall also write
	\[
		\Collection x = (\Collection{x}_a,\Collection{x}_b,\Collection{x}_c)
	\]
	whenever a sequence can be split into three subsequences
	\( \Collection{x}_a \), \( \Collection{x}_b \)
	and \( \Collection{x}_c \), and
	\[
		\Collection{x \otimes y} \coloneqq (x_0 \otimes y_0, \dots, x_n \otimes y_n)
	\]
	the sequence of term-by-term tensor products
	if the category is monoidal.

\item[Hom-sets]
	For any category \( \Category C \),
	we shall denote its hom-sets vertically as
	\[
		\Category C \inout{x}{y}
		\coloneqq \Hom{\Category C}{x}{y}
	\]
	and likewise, denote the values of a (set-valued coloured) operad
	\( \Category P \) by \( \Category P \inout {\Collection x} {y} \).
	The composition of two arrows then becomes a map
	written
	\( \Category C \inout y z \times \Category C \inout x y
	\to \Category C \inout x z \) for example.
\end{description}

\section{Oplax monoidal categories}
We shall start by reviewing the definition
and basic features of the \( 2 \)\=/category
\( \Oplax \) of oplax monoidal categories,
a notion that has yet to receive a full
treatment in the literature.
Along the way, we introduce several conventions
of notation that will be abundantly
used throughout the paper.

In order to maximise the readability of a subject
which necessarily involves an infinite amount
of structural maps, we have chosen to:
avoid writing indices and variables as much
as possible;
start all indexations from \( n = -1 \),
so as to obtain an easy
`\( p + q \) additivity rule' on the structural maps;
change the set of notations depending on the
case at hand.
Other presentation and notation choices
can be found in
Day \& Street
\cite{zbMATH01963527},
Leinster
\cite{doi:10.1017/cbo9780511525896},
Batanin \& Weber
\cite{doi:10.1007/s10485-008-9179-7}
or B\"ohm \& Vercruysse
\cite{doi:10.1090/conm/771}.

\subsection{Definition and notations}
\label{sec:def_oplax_cat}

\begin{definition}[(Oplax monoidal category)]
	\label{Definition: oplax monoidal category}
	An oplax monoidal category \( \Category V^\oplax \) is 
	a category \( \Category V \) equipped with the following data:
	\begin{itemize}
		\item
			functorial structure maps
			\[
				\begin{tikzcd}
					{\Category V}^{n+1}
					\rar["\oplax^n"]
						& \Category V
				\end{tikzcd}
			\]
			for every integer \( n \geq -1 \);
		\item
			natural transformations
			\[
				\begin{tikzcd}[column sep = large]
					\oplax^{p+q}
					\arrow[r, "{a^{p,q}_i}"]
						&\oplax^p
					\circ (\underbrace{
						\IdentityFunctor, \dots, \IdentityFunctor
					}_{i \text{ times}},
					\oplax^q,
					\underbrace{
						\IdentityFunctor, \dots, \IdentityFunctor
					}_{j \text{ times}})
				\end{tikzcd}
			\]
			for every \( p \geq 0 \), every \( q \geq -1 \)
			and every \( i \geq 0, j \geq 0 \) with
			\( i+j = p \), called associators, which we shall
			abbreviate as
			\[
				\begin{tikzcd}[column sep = large]
					 \oplax^{p+q}
					\rar
						&\oplax^p (-, \oplax^q, -)
				\end{tikzcd}
			\]
		\item
			a counit natural transformation
			\[
				\begin{tikzcd}[column sep = large]
					 \oplax^0
					\arrow[r, "i"]
						&\IdentityFunctor
				\end{tikzcd}
			\]
	\end{itemize}
	satisfying the following axioms
	\begin{description}
		\item[counitality]
			\[
				\begin{tikzcd}
					\oplax^p
					\ar[rd, equal]
					\rar
						& \oplax^0 \oplax^p
						\dar
						\\
						& \oplax^p
				\end{tikzcd}
				\qand
				\begin{tikzcd}
					\oplax^p
					\ar[rd, equal]
					\rar
						& \oplax^p(-, \oplax^0,-)
						\dar
						\\
						& \oplax^p
				\end{tikzcd}
			\]
			commute whenever they make sense;
		\item[Parallel decomposition]
			\[
				\begin{tikzcd}[column sep = 70]
					\oplax^{p+q+r}
					\arrow[r]
					\arrow[d]
						&\oplax^{p+q}(-, \oplax^r, -)
						\arrow[d] \\
					\oplax^{p+r}(-, \oplax^q, -)
					\arrow[r]
						& \oplax^p (-, \oplax^q, -, \oplax^r, -)
				\end{tikzcd}
			\]
			commutes whenever it makes sense;
		\item[Sequential decomposition]
			\[
				\begin{tikzcd}[column sep = 70]
					 \oplax^{p+q+r}
					\arrow[r]
					\arrow[d]
						& \oplax^{p+q}(-, \oplax^r, -)
						\arrow[d] \\
					\oplax^p(-, \oplax^{q+r}, -)
					\arrow[r]
						&\oplax^p (-, \oplax^q(-, \oplax^r, -), -)
				\end{tikzcd}
			\]
			commutes whenever it makes sense.
	\end{description}
	A \emph{lax} monoidal structure on a category \( \Category V \)
	is the data of an oplax monoidal structure on the opposite
	category \( \Category V\Op \).
	A lax (or oplax) monoidal structure is called
	\emph{strong} whenever the structural natural transformations
	are all invertible.
\end{definition}

Lax monoidal categories were originally introduced by
Day and Street as an example of lax monoids in \( \Cat \)
\cite{zbMATH01963527}. Notice however that our presentation
differs slightly from theirs, in that they require an associator 
of the form \( \oplax^n \to \oplax^p(\oplax^{n_0},\dots,\oplax^{n_p}) \)
with \( n_0 + \dots + n_p = n - p \) and \( 0 \leq p < n \),
\ie a natural transformation which relates the \( n \)th
structure map to the composition of \( \oplax^p \) with
any collection of \( \{\oplax^{n_k}\} \) corresponding to
a partition of \( n-p \) in \( p+1 \) parts. Consequently,
the parallel and sequential decomposition conditions are
replaced by the fact that two successive decompositions using a different
sequence of partitions are identical. These two presentations
can be seen as the counterpart of the presentations of operads in terms of
full composition maps and partial composition maps, and hence
turn out to be equivalent, as explained by Ching 
\cite{doi:10.1007/s40062-012-0007-2}. 

\begin{notation}
	For \( \Category V^\oplax \) an oplax monoidal category,
	we shall let
	\[
		1_\oplax \coloneqq \oplax^{-1}, \quad
		x \otimes_\oplax y \coloneqq \oplax^1(x,y)
		\qand
		x \otimes_\oplax y \otimes_\oplax z
		\coloneqq \oplax^2(x,y,z)
	\]
	for any objects \( x,y,z \) of \( \Category V \).
	The same notations will be used whenever \( \Category V^\lax \)
	is a lax monoidal category.
	
	To increase readability when the number of variables is
	unspecified,
	we shall write
	\( \oplax(\bx) \) to designate either \( 1_\oplax = \oplax^{-1} \)
	if \( \bx \) is the empty sequence
	or \( \oplax^n(x_0, \dots, x_n) \)
	if \( \bx = (x_0, \dots, x_n) \).
\end{notation}

\begin{remark}
	Any monoidal category \( \Category V^\otimes \) can be turned
	into a strong oplax category by setting
	\( \oplax^{-1} = 1_\otimes \),
	\( \oplax^0 = \IdentityFunctor_{\Category V} \)
	and
	\[
		\oplax^n(x_0, \dots, x_n)
		\coloneqq
		((((x_0 \otimes x_1) \otimes \dots)\dots) \otimes x_n)
	\]
	for \( n \geq 0 \).
	This is not a trivial result and we shall give more details
	about this in a dedicated section
	\SectionRef{Sec: comparison of enrichment theories}.
\end{remark}

\begin{example}
	\label{example comonad coproduct}
	By definition, \( \oplax^0 \) satisfies the axioms of a comonad.
	Conversely, if \( W \) is a comonad
	on a category with finite coproducts,
	one can define an oplax monoidal structure by setting
	\[
		\textstyle
		\oplax(\{x_i\}_{i \in I}) \coloneqq
		\coprod_{i \in I} W(x_i)
	\]
	for \( I \) finite.
	This construction is dual to the one given by Batanin and Weber
	where a lax monoidal structure was built from a monad
	\cite[2.3]{doi:10.1007/s10485-008-9179-7}.
\end{example}

\subsection{Lax functors}
\label{sec:lax_functors}

We give two definitions of lax functors, depending on
whether the source category is endowed with a lax or an oplax
structure.

\begin{definition}[(Lax functor with oplax source)]
	Let \( F \From \Category U^\psi \to \Category V^\oplax \)
	be a functor between two oplax monoidal categories.
	A lax monoidal structure on \( F \) consists in
	natural transformations
	\[
		\begin{tikzcd}
			\oplax^n \circ
			(\underbrace{F \times \cdots \times F}_{
				n+1 \text{ times}})
			\rar["l^n"]
				& F \circ \psi^n
		\end{tikzcd}
	\]
	which we shall abbreviate as
	\[
		\oplax^n F \longrightarrow F \psi^n
	\]
	for every \( n \geq -1 \).
	It is required that
	\begin{description}
			\item[counit]
				\[
					\begin{tikzcd}[row sep = large]
						\oplax^0 F
						\ar[dr]
						\dar["l^0" swap]
						& \\
						F \psi^0 \rar & F
					\end{tikzcd}
				\]
				commute;
			\item[decomposition]
				\[
					\begin{tikzcd}[row sep = large]
						\oplax^{p+q}F
						\rar
						\ar[dd, "l^{p+q}" swap]
							&
							\oplax^p(-,\oplax^q,-) F
							\dar["{\oplax^p(-, l^q, -)}"]
							\\
								& \oplax^p F (-, \psi^q, -)
								\dar["{l^p(-, \psi^q, -)}"]
							\\
							F \psi^{p+q}
							\rar
								& F \psi^p(-, \psi^q, -)
					\end{tikzcd}
				\]

				commute whenever it makes sense.
	\end{description}
\end{definition}

\begin{definition}[(Lax functor with lax source)]
	\label{def:lax_to_oplax_functor}
	Let \( F \From \Category U^\lax \to \Category V^\oplax \)
	be a functor between a lax monoidal category and an
	oplax monoidal category.
	A lax monoidal structure on \( F \) is the data of
	natural transformations
	\[
		\begin{tikzcd}
			\oplax^n \circ
			(\underbrace{F \times \cdots \times F}_{
				n+1 \text{ times}})
			\rar["l^n"]
				& F \circ \lax^n
		\end{tikzcd}
	\]
	which we shall abbreviate as
	\[
		\oplax^n F \longrightarrow F \lax^n
	\]
	for every \( n \geq -1 \).
	It is required that
	\begin{description}
			\item[Unitality]
				\[
					\begin{tikzcd}
						\oplax^0 F
						\ar[dd,"{l^0}",swap]
						\ar[dr]
						& \\
						& F \ar[dl]
						\\
						F \lax^0 &
					\end{tikzcd}
				\]
				commute;
			\item[Additivity]
				
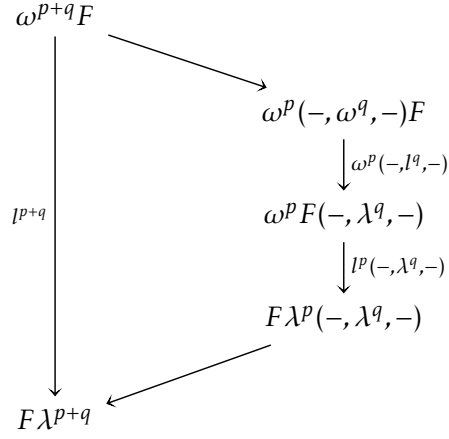
\[
					\begin{tikzcd}[column sep = huge]
						\oplax^{p+q} F \ar[dr]
						\ar[dddd,"l^{p+q}" swap]
						&
						\\
							       &
							       \oplax^p(-, \oplax^q, -)F
							       \dar["{\oplax^p(-, l^q, -)}"]
							       \\
							       &
							       \oplax^p F (-, \lax^q, -)
							       \dar["{l^p(-, \lambda^q, -)}"]
							       \\
							       &
							       F \lax^p(-, \lax^q, -)
							       \ar[ld]
							       \\
							       F \lax^{p+q}
					\end{tikzcd}
				\]
				commute whenever it makes sense.
	\end{description}
\end{definition}

The additivity rule just described makes it easy to
guess that the higher maps \( l^n \) for \( n > 1 \) as well as
\( l^0 \) can be reconstructed from \( l^1 \) and \( l^{-1} \).
The proof of this result proceeds via careful inductions.

\begin{proposition}[(Reconstruction of lax monoidal functors)]
	\label{prop:truncation_lax_functors}
	Let
	\( F \From \Category U^\lax \to \Category V^\oplax \)
	be a functor between a lax monoidal category and an
	oplax monoidal category.
	Let
	\[
		\begin{tikzcd}
			F(-) \otimes_\oplax F(-)
			\rar["l^1"]
				& F(- \otimes_\lax -)
		\end{tikzcd}
	\]
	be a natural transformation and let
	\[
		\begin{tikzcd}
			1_\oplax
			\rar["l^{-1}"]
				& F(1_\lax)
		\end{tikzcd}
	\]
	be a morphism.
	Moreover assume that
	\begin{description}
	\item[Associativity]
		the associativity diagram
	
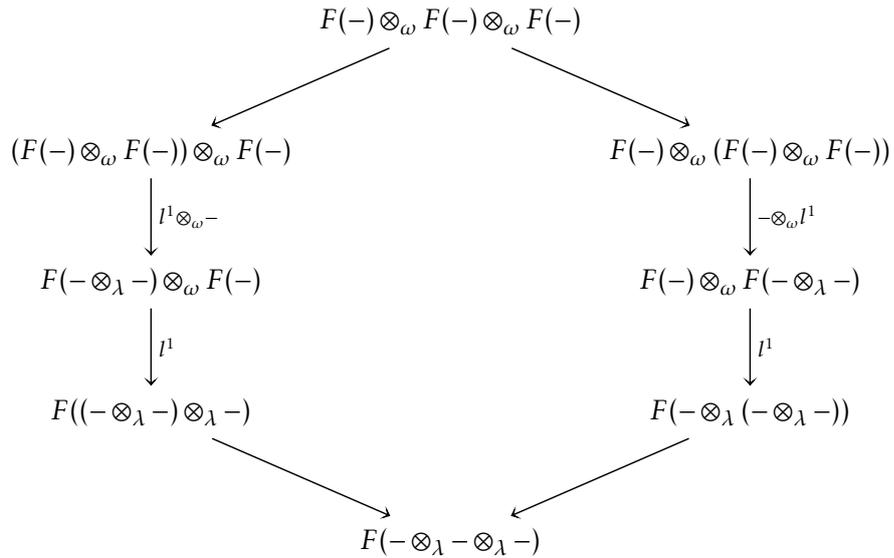
\[
		\begin{tikzcd}[column sep=0.05ex, row sep=large]
			& F(-) \otimes_\oplax F(-) \otimes_\oplax F(-)
			\ar[dl]
			\ar[dr] & \\
			(F(-) \otimes_\oplax F(-)) \otimes_\oplax F(-)
			\ar[d, "{l^1 \otimes_\oplax -}"]
			&& F(-) \otimes_\oplax (F(-) \otimes_\oplax F(-))
			\ar[d, "{- \otimes_\oplax l^1}"] \\
			F(- \otimes_\lax -) \otimes_\oplax F(-)
			\ar[d, "l^1"]
			&& F(-) \otimes_\oplax F(- \otimes_\lax -)
			\ar[d, "l^1"] \\
			F((- \otimes_\lax -) \otimes_\lax -)
			\ar[dr]
			&& F(- \otimes_\lax (- \otimes_\lax -))
			\ar[dl] \\
			& F(- \otimes_\lax - \otimes_\lax -)
		\end{tikzcd}
	\]
	commutes;
	\item[Unitality]
		the two unitality diagrams
	
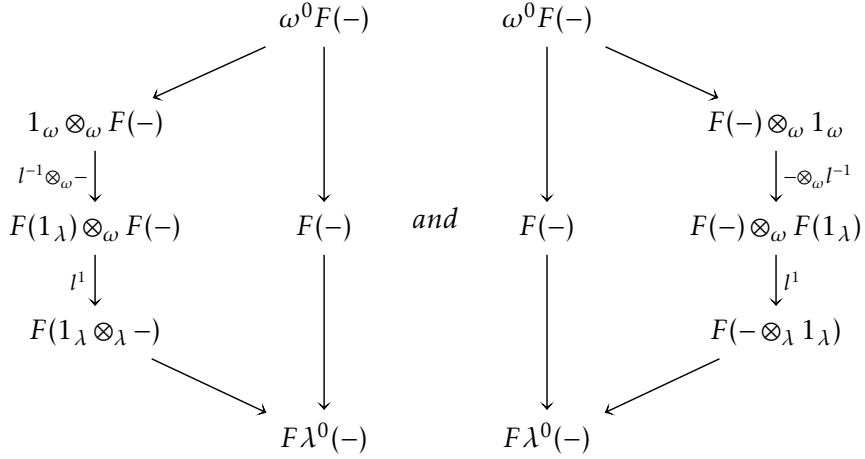
\[
		\begin{tikzcd}
			& \oplax^0 F(-)
			\ar[dl]
			\ar[dd]
			\\
			1_\oplax \otimes_\oplax F(-)
			\ar[d, "{l^{-1} \otimes_\oplax -}" swap]
			&
			\\
			F(1_\lax) \otimes_\oplax F(-)
			\ar[d, "l^1" swap] & F(-) \ar[dd]
			\\
			F(1_\lax \otimes_\lax -)
			\ar[dr]
			&
			\\
			& F \lax^0(-)
		\end{tikzcd}
		\qand
		\begin{tikzcd}
			\oplax^0 F(-)
			\ar[dd]
			\ar[dr] &
			\\
			& F(-) \otimes_\oplax 1_\oplax
			\ar[d, "{- \otimes_\oplax l^{-1}}"]
			\\
			F(-) \ar[dd] & F(-) \otimes_\oplax F(1_\lax)
			\ar[d, "l^1"]
			\\
			& F(- \otimes_\lax 1_\lax)
			\ar[dl]
			\\
			F \lax^0(-) &
		\end{tikzcd}
	\]
	commute,
	\end{description}
	then the pair \( (l^{-1}, l^1) \) can be extended,
	in a unique way,
	into a sequence of natural transformations
	\[
		\begin{tikzcd}
			\oplax^n F
			\rar["l^n"]
				& F \lax^n,
		\end{tikzcd}
		\quad n \geq -1
	\]
	endowing \( F \From \Category U^\lax \to \Category V^\oplax \)
	with a structure of lax monoidal functor.
\end{proposition}

\begin{proof}
	Let us index the additivity diagrams in the definition of a
	lax functor by \( A(p, q, i) \), where \(  0 \leq i \leq p \)
	is the position of \( \oplax^q \) inserted in \( \oplax^p \).

	Define \( l^0 \) as the composition
	\[
		\oplax^0 F \longrightarrow F \longrightarrow
		F \lax^0
	\]
	so that the unitality diagram for lax functors commute by
	definition, and
	define \( l^{q+1} \) by induction: for any \( q \geq 1 \),
	let \( l^{q+1} \) be such that
	\( A(1,q,0) \) commutes.

	Furthermore, let \( \Induction(p,q) \) denote the hypothesis
	that \( A(p,q,i) \) hold for every
	position \( 0 \leq i \leq p \).
	We shall show that \( \Induction(p, q) \) is true for every
	\( p \geq 0 \) and every \( q \geq -1 \).

	Let us start by the case of \( \Induction(0, q) \) with
	\( q \geq -1 \).
	The following two squares
	\[
		\begin{tikzcd}[ampersand replacement=\&]
			\oplax^q F
			\arrow[d, "l^q", swap]
			\& \oplax^0 \oplax^q F
			\lar
			\arrow[d, "\oplax^0(l^q)"] \\
			F \lax^q
			\& \oplax^0 F \lax^q
			\lar
		\end{tikzcd}
		\quad
		\begin{tikzcd}[ampersand replacement=\&]
			\oplax^q F
			\arrow[d, "l^q", swap]
			\rar
			\& \oplax^0 \oplax^q F
			\arrow[d, "\oplax^0(l^q)"] \\
			F \lax^q
			\& \oplax^0 F \lax^q
			\lar
		\end{tikzcd}
	\]
	commute.
	The first by functoriality of \( \oplax^0 \); the commutativity
	of the second one follows from the commutativity of the first
	one, in addition to the unitality axiom for \( \oplax \).
	As a consequence, the top part of the diagram
	\[
		\begin{tikzcd}
			\oplax^q F
			\ar[ddddd, "l^q" swap]
			\ar[dr]
			\\
			& \oplax^0 \oplax^q F
			\dar["\oplax^0 l^q"]
			\\
			& \oplax^0 F \lax^q
			\ar[dddl]
			\ar[d]
			\\
			& F \lax^q
			\dar
			\\
			& F \lax^0 \lax^q
			\ar[dl]
			\\
			F \lax^q
		\end{tikzcd}
	\]
	commutes and the bottom part commutes by unitality of
	\( \lax \).

	We shall now address \( \Induction(1, q) \)
	for every \( q \geq -1 \). The hypothesis
	\( \Induction(1,-1) \) is satisfied thanks to the two unitality
	assumptions.
	For \( \Induction(1,0) \), the two diagrams
	\[
		\begin{tikzcd}[column sep = small, row sep = large]
			F(-) \otimes_\oplax F(-)
			\dar["l^1" swap]
			\rar
				& F \lax^0(-) \otimes_\oplax F(-)
				\dar["{l^1\left(\lax^0(-),-\right)}"]
				\\
			F(- \otimes_\lax -) \rar
			     & F\left(\lax^0(-) \otimes_\lax -\right)
		\end{tikzcd}
		\quad
		\begin{tikzcd}[column sep = small, row sep = large]
			F(-) \otimes_\oplax F(-)
			\dar["l^1" swap]
			\rar
				& F \lax^0(-) \otimes_\oplax F(-)
				\dar["{l^1\left(\lax^0(-),-\right)}"]
			\\
			F(- \otimes_\lax -)
			     & F\left(\lax^0(-) \otimes_\lax -\right)
			     \lar
		\end{tikzcd}
	\]
	commute: the first one by functoriality of \( - \otimes_\lambda - \);
	the second one because of the commutativity of the first one
	coupled with the unitality of \( \lambda \).
	As a consequence, the bottom part of the diagram
	\[
		\begin{tikzcd}[column sep = huge]
			F(-) \otimes_\oplax F(-)
			\ar[ddddd, "l^1" swap]
			\ar[dr]
			\ar[dddr]
			\\
				& \oplax^0 F(-) \otimes_\oplax F(-)
				\dar
			\\
				& F(-) \otimes_\oplax F(-)
				\dar
			\\
				& F \lax^0(-) \otimes_\oplax F(-)
				\dar["{l^1\left(\lax^0(-),-\right)}"]
			\\
				& F(\lax^0(-) \otimes_\lax -)
				\ar[dl]
			\\
			F(- \otimes_\lax -)
		\end{tikzcd}
	\]
	commutes and the top part commutes by unitality of \( \oplax \).
	This shows that half of \( \Induction(1,0) \) holds, the other
	half can be obtained symmetrically.

	The hypothesis \( \Induction(1,1) \)
	is true by the associativity axiom, we can now show that
	\( \Induction(1,q) \) is true by induction on \( q \geq 1 \).
	The diagram
	
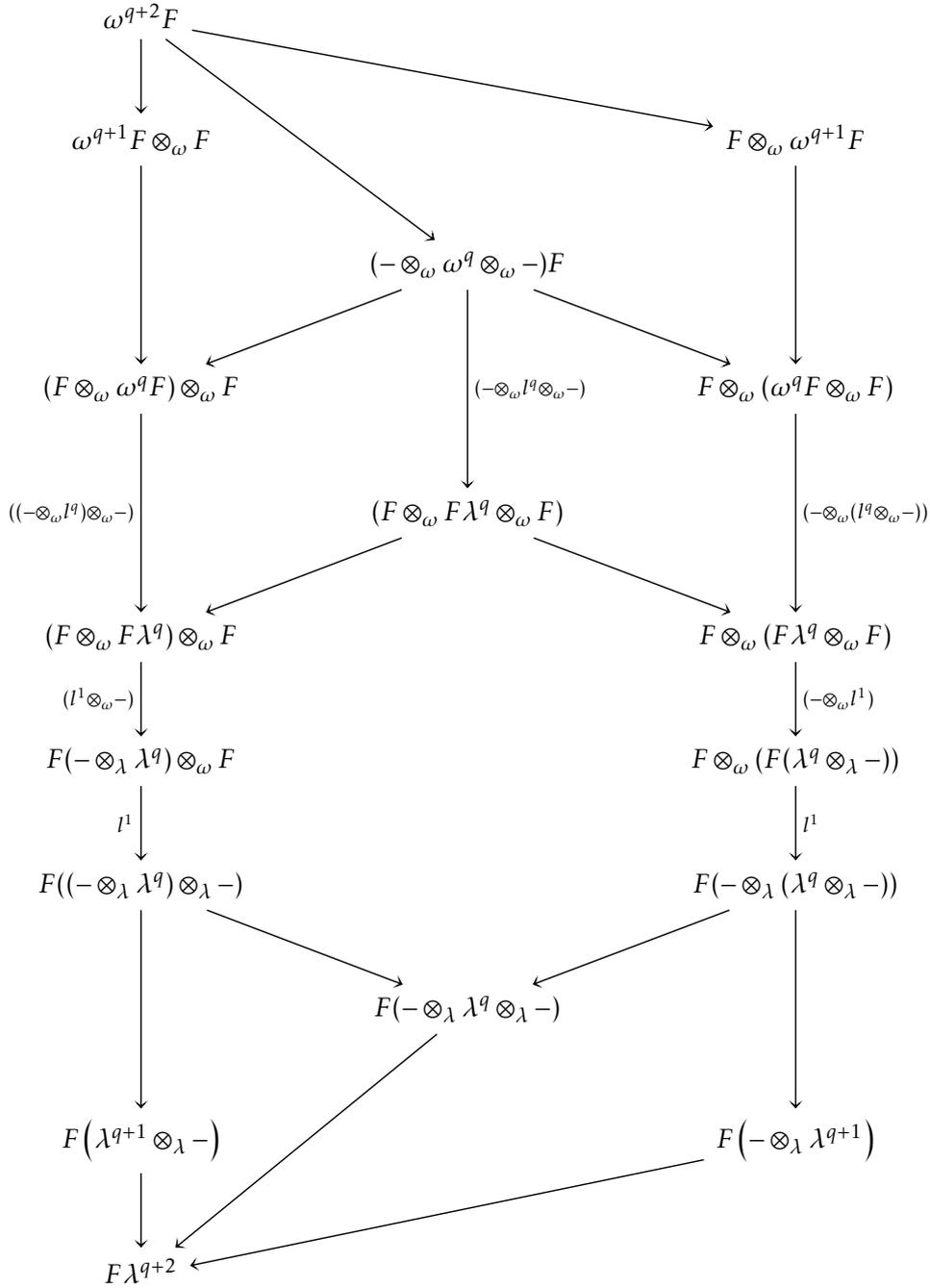
\[
		\begin{tikzcd}[column sep = large, row sep = large]
			\oplax^{q+2} F
			\dar
			\ar[ddr]
			\ar[rrd]
			\\
			\oplax^{q+1}F \otimes_\oplax F
			\ar[dd]
				&
					& F \otimes_\oplax
					\oplax^{q+1} F
					\ar[dd]
			\\
				& (- \otimes_\oplax \oplax^q
				\otimes_\oplax -) F
				\ar[dl] \ar[dr]
				\ar[dd, "{(- \otimes_\oplax l^q
				\otimes_\oplax -)}"]
			\\
			(F \otimes_\oplax \oplax^q F) \otimes_\oplax F
			\ar[dd, "{((- \otimes_\oplax
			l^q) \otimes_\oplax -)}" swap]
				& & F \otimes_\oplax
				(\oplax^q F \otimes_\oplax F)
				\ar[dd, "{(- \otimes_\oplax
				(l^q \otimes_\oplax -))}"]
			\\
				& (F \otimes_\oplax F \lax^q
				\otimes_\oplax F)
				\ar[dr] \ar[dl]
			\\
			(F \otimes_\oplax F \lax^q) \otimes_\oplax F
			\dar["{(l^1 \otimes_\oplax -)}" swap]
				& & F
				\otimes_\oplax
				(F \lax^q \otimes_\oplax F)
				\dar["(- \otimes_\oplax l^1)"]
			\\
			F(- \otimes_\lax \lax^q) \otimes_\oplax F
			\dar["l^1" swap]
				&& F \otimes_\oplax
				(F(\lax^q \otimes_\lax -))
				\dar["l^1"]
			\\
			F((- \otimes_\lax \lax^q) \otimes_\lax -)
			\ar[dd]
			\ar[dr]
				&& F(- \otimes_\lax (\lax^q \otimes_\lax
				-))
				\ar[dl]
				\ar[dd]
			\\
				& F(- \otimes_\lax\lax^q \otimes_\lax -)
				\ar[ddl]
			\\
			F\left(\lax^{q+1} \otimes_\lax -\right)
			\dar
				&& F\left(-
				\otimes_\lax \lax^{q+1}\right)
				\ar[dll]
			\\
			F \lax^{q+2}
		\end{tikzcd}
	\]
	commutes.
	Indeed, the top two squares commute by sequential decomposition;
	the middle two squares by functoriality;
	the octagon commutes by the associativity axiom and
	the bottom two squares commute by sequential composition.
	Using \( \Induction(1,q) \), one can see that the left full
	composite map equals \( l^{q+2} \) and thus, the commutativity
	of this diagram gives us the commutativity
	of \( A(1,q+1,1) \),
	since \( A(1,q+1, 0) \) commutes by definition, we have shown
	\( \Induction(1,q) \implies \Induction(1,q+1) \),
	and hence \( \Induction(1,q) \) holds for any integer
	\( q \geq -1 \).

	Finally, assume that \( \Induction(p,q) \) holds.
	Then the diagram
	\[
		\begin{tikzcd}[row sep=huge]
			\oplax^{p+q+1} F
			\dar
			\ar[drr]
			\\
			F \otimes_\oplax \oplax^{p+q} F
			\arrow[dddd, "{- \otimes_\oplax l^{p+q}}" swap]
			\ar[dr]
			&& \oplax^{p+1}(-,\oplax^q,-) F
			\ar[dl]
			\arrow[d, "{\oplax^{p+1}(-,l^q,-)}"]
			\\
			& F \otimes_\oplax (\oplax^p(-,\oplax^q,-) F)
			\ar[d, "{- \otimes_\oplax (\oplax^p(-,l^q,-))}" swap]
			& \oplax^{p+1} F (-,\lax^q,-)
			\arrow[dddd, "l^{p+1}"]
			\ar[dl]
			\\
			& F \otimes_\oplax (\oplax^p F (-, \lax^q, -))
			\arrow[d, "{- \otimes_\oplax l^p}" swap] &
			\\
				 & F \otimes_\oplax (F \lax^p(-,\lax^q,-))
			\ar[dl]
			\arrow[d, "l^1" swap]
			\\
			F \otimes_\oplax F \lax^{p+q}
			\arrow[d, "l^1" swap]
				 & F(- \otimes_\lax
					 (\lax^p(-,\lax^q,-)))
			\ar[dr]
			\ar[dl] &\\
			F \left(- \otimes_\lax \lax^{p+q}\right)
			\ar[d]
			&& F \lax^{p+1}(-,\lax^q,-)
			\ar[dll] \\
			F \lax^{p+q+1} &
		\end{tikzcd}
	\]
	commutes whenever it makes sense.
	The top square by sequential decomposition; the left pentagon
	by \( \Induction(p,q) \); the right pentagon by definition of
	\( l^{p+1} \); the left and the right squares by functoriality,
	and the bottom square by sequential composition.
	Using \( \Induction(1, p+q) \), one can see that the left
	total composite map equals \( l^{p+q+1} \).
	This shows that
	\( \Induction(p,q) \) implies the commutativity of
	\( A(p+1,q,i) \) for every \( 0 < i \leq p+1 \).
	The mirror diagram involving \( l^{p+q} \otimes_\oplax - \)
	instead of \( - \otimes_\oplax l^{p+q} \) shows that
	\( A(p+1, q, i) \) commutes for every \( 0 \leq i < p+1 \).
	Thus we have shown
	that \( \Induction(p, q) \implies \Induction(p+1, q) \)
	which concludes the proof of the proposition.
\end{proof}

\begin{remark}
	Note that if both categories in
	the previous proposition are strong monoidal,
	then we also recover the usual notion of 
	a lax monoidal functor between (strong) monoidal categories.

	Furthermore,
	out of the four possible combinations of lax functors between
	(op)lax categories, lax functors with lax source and oplax target
	are privileged as they are the only ones for which
	the additivity diagram can be used to define the higher maps
	\( l^{p+q} \) in terms of lower ones, as necessary for
	the reconstruction result to hold. 
\end{remark}

\subsection{Oplax functors}
\label{sec:oplax_functors}

The previous section provided two
(non-exhaustive) definitions of lax functors.
Using opposite categories,
we can readily define two corresponding types of oplax functors:
oplax functors with lax source
and lax target as well as oplax functors
with oplax source and lax target.
Since none of these involve an oplax
functor with oplax target,
we add to the above taxonomy the following definition.

\begin{definition}
	Let \( \Category U^\psi \) and \( \Category V^\oplax \)
	be two oplax monoidal categories.
	An oplax structure on a functor
	\( F \From \Category U \to \Category V \)
	consists in natural transformations
	\[
		\begin{tikzcd}
			F \circ \psi^n
			\rar["l^n"]
				& \oplax^n \circ
				(\underbrace{F \times \cdots \times F}_{
				n+1 \text{ times}})
		\end{tikzcd}
	\]
	which we shall abbreviate as
	\[
		F \psi^n \longrightarrow \oplax^n F
	\]
	for every \( n \geq -1 \).
	It is required that
	\begin{description}
			\item[counit]
				\[
					\begin{tikzcd}[row sep = large]
						F \psi^0
						\dar["l^0" swap]
						\ar[dr]
							\\
						\oplax^0 F
						\rar
							& F
					\end{tikzcd}
				\]
				commute;
			\item[decomposition]
				\[
					\begin{tikzcd}[row sep = large]
						F \psi^{p+q}
						\rar
						\ar[dd, "l^{p+q}" swap]
							& F \psi^p(-, \psi^q,-)
							\dar["{l^p(-, \psi^q,-)}"]
							\\
							& \oplax^p F (-, \psi^q,-)
							\dar["{\oplax^p(-, l^q,-)}"]
							\\
						\oplax^{p+q} F
						\rar
							& \oplax^p (-, \oplax^q, -) F
					\end{tikzcd}
				\]
				commute whenever it makes sense.
	\end{description}
\end{definition}

\begin{remark}
	Note that oplax monoidal structures on functors,
	from an oplax monoidal category to a lax monoidal
	one, enjoy a reconstruction theorem similar to 
	that of lax monoidal functors with lax monoidal
	source and oplax monoidal target
	\UnskipRef{prop:truncation_lax_functors}.
	This is simply a consequence of the fact that, for an oplax
	monoidal structure on a functor
	\( F \From \Category U^\oplax \to \Category V^\lax \),
	with \( \Category U^\oplax \) oplax monoidal and
	\( \Category V^\lax \) lax monoidal, the decomposition
	condition in the previous definition can be used 
	to define the natural transformations
	\( l^n \From F\,\oplax^n \to \lax^n\,F \)
	for \( n \geq 2 \) in terms of the pair \( (l^1, l^{-1}) \),
	while the counit condition define \( l^0 \) as
	the composition of the unit and counit of \( \lax \)
	and \( \oplax \).
\end{remark}

\subsection{Monoidal natural transformations}
\label{sec:monoidal_nat_transfo}

\begin{definition}[(Monoidal natural transformation)]
	Let \( \Category U^\psi \)
	and \( \Category V^\oplax \) be two oplax monoidal
	categories.
	Let \( F, G \From \Category U \to \Category V \) be two
	lax functors with respective lax monoidal structure
	\( \{k^n\}_{n \geq -1} \) and \( \{l^n\}_{n \geq -1} \).
	A natural transformation \( \alpha \From F \implies G \)
	is said to be monoidal if the diagrams
	\[
		\begin{tikzcd}[sep=large]
			\oplax^n\, F
			\arrow[r, "k^n"]
			\arrow[d, "{\oplax^n(\alpha,\dots,\alpha)}"{left}]
			& F \psi^n
			\arrow[d, "\alpha_{\psi^n}"] \\
			\oplax^n\, G
			\arrow[r, "l^n"]
			& G \psi^n
		\end{tikzcd}
	\]
	commute for any \( n \geq -1 \).
	One can define monoidal transformations between oplax
	monoidal functors in a similar way.
\end{definition}

For monoidal natural transformations between lax monoidal functors
that can be truncated, \ie whose source and target categories
are lax and oplax monoidal respectively, the infinite collection
of conditions in the previous definition is actually redundant,
as illustrated by the following lemma.

\begin{lemma}
	\label{lemma:truncation_monoidality}
	Let \( F,G \From \Category U^\lax \to \Category V^\oplax \)
	be two lax monoidal functors from a lax monoidal category
	to an oplax monoidal one, with respective lax monoidal
	structures \( \{k^n\} \) and \( \{l^n\}\). A natural transformation
	\( \alpha \From F \to G \) between such functors is monoidal
	if and only if the following diagrams
	\[
		\begin{tikzcd}[sep=large]
			(F \otimes_\oplax F)
			\ar[r, "\alpha \otimes_\oplax \alpha"]
			\ar[d, "k^1"] 
			& (G \otimes_\oplax G)
			\ar[d, "l^1"] \\
			F(- \otimes_\lax -)
			\ar[r, "\alpha_{- \otimes_\lax -}"]
			& G(- \otimes_\lax -)
		\end{tikzcd}
		\qand
		\begin{tikzcd}[sep=large]
			1_\oplax
			\ar[d, "k^{-1}"]
			\ar[dr, "l^{-1}"] & \\
			F(1_\lax)
			\ar[r, "\alpha_{1_\lax}"]
			& G(1_\lax)
		\end{tikzcd}
	\]
	commute.
\end{lemma}
\begin{proof}
	Let us denote by \( M (n) \), with \( n \geq -1 \),
	the diagrams encoding the conditions that \( \alpha \)
	is monoidal. The diagrams \( M(-1) \) and \( M(1) \)
	correspond to those required in the above lemma.
	The diagram \( M(0) \) can be written as
	\[
		\begin{tikzcd}[sep=large]
			\oplax^0\,F
			\ar[dd, "k^0"]
			\ar[dr]
			\ar[rrr, "\oplax^0(\alpha)"]
			&&& \oplax^0\,G
			\ar[dl]
			\ar[dd, "l^0"] \\
			& F
			\ar[dl]
			\ar[r, "\alpha"]
			& G 
			\ar[dr] & \\
			F\,\lax^0
			\ar[rrr, "\alpha_{\lax^0}"]
			&&& G\,\lax^0
		\end{tikzcd}
	\]
	and hence commutes as a consequence of the unitality
	conditions of the lax monoidal structures \( k \)
	and \( l \), and of the naturality of the co/unit
	of \( \oplax \) and \( \lax \).

	Now assume that \( M(k) \) commutes for all integers
	\( k < n \). Then the diagram 
	\[
		\begin{tikzcd}[sep=huge]
			\oplax^n\,F
			\ar[rrr, "{\oplax^n(\alpha,\dots,\alpha)}"]
			\ar[dddd, "k^n"]
			\ar[dr]
			&&& \oplax^n\,G 
			\ar[dl]
			\ar[dddd, "l^n"] \\
			& {(- \otimes_\oplax \oplax^{n-1})}\,F
			\ar[d, "{- \otimes_\oplax k^{n-1}}"]
			\ar[r, "{\alpha \otimes_\oplax \oplax^{n-1}(\alpha,\dots,\alpha)}"]
			& {(- \otimes_\oplax \oplax^{n-1})}\,G
			\ar[d, "{- \otimes_\oplax l^{n-1}}"] & \\
			& (F \otimes_\oplax F\,\lax^{n-1})
			\ar[d, "k^1"]
			\ar[r, "{\alpha \otimes_\oplax \alpha_{\lax^{n-1}}}"]
			& (G \otimes_\oplax G\,\lax^{n-1})
			\ar[d, "l^1"] & \\
			& F\,{(- \otimes_\lax \lax^{n-1})}
			\ar[dl]
			\ar[r, "{\alpha \otimes_\lax \lax^{n-1}(\alpha,\dots,\alpha)}"]
			& G\,{(- \otimes_\lax \lax^{n-1})}
			\ar[dr] & \\
			F\,\lax^n
			\ar[rrr, "\alpha_{\lax^n}"]
			&&& G\,\lax^n
		\end{tikzcd}
	\]
	commutes: the top and bottom (exterior) squares by
	naturality of the associators of \( \oplax \) and
	\( \lax \), the left and right (exterior) squares
	by additivity of the lax monoidal structures
	\( k \) and \( l \) and the interior two squares
	by assumption. This proves that \( M(n) \) also
	commutes and hence proves the lemma by recursion.
\end{proof}

\begin{remark}
	Note that the same result holds for monoidal
	natural transformations between two oplax monoidal
	functors, whose source and target categories are
	respectively oplax and lax monoidal. In other words,
	the monoidality condition on natural transformations
	between op/lax monoidal functors can be truncated
	whenever the op/lax monoidal structure of these functors
	can be.
\end{remark}
\begin{remark}
	When both \( \lax \) and \( \oplax \) are strong
	monoidal, the previous lemma reproduces the usual
	definition of a monoidal natural transformation
	between monoidal functors.
\end{remark}

\subsection{Monoids in oplax monoidal categories}
\label{sec:monoids}
\begin{definition}[(Category of monoids)]
	A monoid in an oplax monoidal category \( \Category V^\oplax \)
	is a lax monoidal functor from the punctual category
	\( * \to \Category V^\oplax \), and a morphism between
	two monoids is a monoidal transformations between 
	the corresponding functors. These data define a category,
	that will be denoted \( \mathsf{Mon}(\Category V^\oplax) \).

	The category of comonoids in \( \Category V^\oplax \)
	is the category of oplax functors
	\( \ast \to \Category V^\oplax \) and natural transformations
	between them.
\end{definition}

In more details, the functor \( * \to \Category V^\oplax \)
singles out an object \( A \) in the target category 
\( \Category V \), while the lax monoidal structure
on this functor consists in a collection of morphisms
\[
	\begin{tikzcd}[sep=large]
		\oplax^n(A,\dots,A)
		\ar[r, "m_A^n"]
		& A
	\end{tikzcd}
\]
of \( \Category V \), such that \( m_A^0 \) identifies with
the counit \( \oplax^0 \to \IdentityFunctor_{\Category V} \),
and which obey the additivity conditions
\[
	\begin{tikzcd}[column sep=normal, row sep=large]
		\oplax^{p+q}\,A
		\ar[dr]
		\ar[ddd, "m_A^{p+q}"] & \\
		& {\oplax^p(-,\oplax^q,-)}\,A
		\ar[d, "{\oplax^p(-,m_A^q,-)}"] \\
		& \oplax^p\,A
		\ar[dl, "m_A^p"] \\
		A &
	\end{tikzcd}
\]
for all integers \( p,q \geq -1 \).
A morphism between
two monoids \( (A, \{m_A^n\}) \) and \( (B, \{m_B^n\}) \)
is a morphism \( f \From A \to B \) such that
\[
	\begin{tikzcd}[column sep = huge]
		\oplax^n\,A
		\ar[r, "{\oplax^n(f,\dots,f)}"]
		\ar[d, "m_A^n"]
		& \oplax^n\,B
		\ar[d, "m_B^n"] \\
		A 
		\ar[r, "f"]
		& B
	\end{tikzcd}
\]
commutes for all integers \( n \geq -1 \).

The punctual category is (trivially) strong monoidal, 
and hence we can apply the reconstruction theorem for
lax monoidal functors \UnskipRef{prop:truncation_lax_functors}
to obtain the following result \cite[Th. 3.5]{doi:10.1090/conm/771}.

\begin{corollary}[(Truncation of monoids)]
	A structure of monoid on an object \( A \)
	in an oplax monoidal category \( \Category V^\oplax \)
	can be uniquely reconstructed from a pair of
	morphisms
	\[
		\mu_A \From A \otimes_\oplax A \to A\,,
		\qand 
		\eta_A \From 1_\oplax \to A\,,
	\]
	such that the associativity and unitality diagrams
	\[
		\hspace{-15pt}
		\begin{tikzcd}[column sep=tiny]
			&[-5ex] A \otimes_\oplax A \otimes_\oplax A
			\ar[dl]
			\ar[dr] &[-5ex] \\
			(A \otimes_\oplax A) \otimes_\oplax A
			\ar[d, "{\mu_A \otimes_\oplax -}" swap]
			&& A \otimes_\oplax (A \otimes_\oplax A)
			\ar[d, "{- \otimes_\oplax \mu_A}"] \\
			A \otimes_\oplax A
			\ar[dr, "\mu_A" swap]
			&& A \otimes_\oplax A
			\ar[dl, "\mu_A"] \\
			& A &
		\end{tikzcd}
		\qquad
		\begin{tikzcd}[column sep=small]
			&[-2ex] \oplax^0(A)
			\ar[dl]
			\ar[ddd]
			\ar[dr] &[-2ex] \\
			1_\oplax \otimes_\oplax A
			\ar[d, "{\eta_A \otimes_\oplax -}" swap]
			&& A \otimes_\oplax 1_\oplax
			\ar[d, "{- \otimes_\oplax \eta_A}"] \\
			A \otimes_\oplax A
			\ar[dr, swap, "\mu_A"]
			&& A \otimes_\oplax A
			\ar[dl, "\mu_A"] \\
			& A &
		\end{tikzcd}
	\]
	commute.
\end{corollary}

\begin{remark}
	When \( \Category V^\otimes \) is monoidal,
	one recovers the usual pentagonal axiom (on the left diagram)
	and the two triangular axioms (on the right diagram)
	of the usual definition of a monoid, independently of a choice
	of lift of the monoidal structure into a strong monoidal structure.

	When \( \Category V^\otimes \) is strictly normal,
	one recovers a theorem from Ching
	\cite[3.4]{doi:10.1007/s40062-012-0007-2}.
\end{remark}

\begin{remark}[(Comonoids in oplax monoidal categories)]
	Contrarily to the case of monoids, a structure
	of comonoid cannot be truncated to lower arity
	maps, as oplax monoidal functors with an oplax
	monoidal category as their target are not
	subject to the reconstruction theorem discussed
	previously.
	This can also be seen from the above
	diagram: it is not possible to use the latter
	as a definition of the higher arity maps in terms
	of the lower arity ones.
\end{remark}

\begin{remark}[(Comonoids in lax monoidal categories)]
	One can also define monoids and comonoids
	in a lax monoidal category \( \Category V^\lax \)
	as lax and oplax monoidal functors respectively,
	from the punctual category.
	In this case, the situation
	is reversed: comonoid structures can be truncated,
	while monoid structures cannot.
	This is simply due
	to the fact that lax monoidal structures on functors
	with lax monoidal target categories cannot be truncated.
\end{remark}

\begin{lemma}[(\( 1_\oplax\) is a canonical comonoid)]
	\label{lemma: 1 is a comonoid}
	Let \( \Category V^\oplax \) be an oplax monoidal category.
	Let us define by induction a map
	\[
		\begin{tikzcd}[column sep = huge]
			1_\oplax
			\rar["w^n"]
				& \oplax^n 1_\oplax \coloneqq
				\oplax^n(1_\oplax,\dots, 1_\oplax)
		\end{tikzcd}
	\]
	for every \( n \geq -1 \)
	by first letting \( w^{-1} \From 1_\oplax \to 1_\oplax \)
	be the identity of \( 1_\oplax \).
	Then,
	assuming \( w^n \) to be defined for \( n \geq -1 \),
	let \( w^{n+1} \) be defined by the composition
	\[
		\begin{tikzcd}[column sep = huge]
			1_\oplax \rar["{w^n}"]
				& \oplax^n 1_\oplax
				\rar["{a^{n+1,-1}_0} 1_\oplax"]
						& \oplax^{n+1} 1_\oplax
		\end{tikzcd}
	\]
	so that:
	\begin{itemize}
		\item
			for every \( n \geq 0 \) and every \( 0 \leq i \leq n \),
			one has \( w^{n+1} = (a^{n+1,-1}_i 1_\oplax) \circ w^n \);
		\item
			the sequence of maps \( \{w^n\}_{n \geq -1} \) makes
			\( 1_\oplax \) a comonoid.
	\end{itemize}
\end{lemma}

\begin{proof}
	The definition of the maps \( w^n \) can be shown to be canonical by induction
	using the parallel decomposition axiom for oplax monoidal categories.

	The counit axiom for \( 1_\oplax \) follows directly from the
	counit axiom for \( \oplax \).
	Let us index the decomposition diagrams for
	oplax functors by \( D(p,q,i) \), where \( 0 \leq i \leq p \)
	is the insertion index for \( \oplax^q \) inside \( \oplax^p \).
	Let \( \Induction(p,q) \) denote the hypothesis that
	\( D(p,q,i) \) commute for every \( 0 \leq i \leq p \).

	By construction \( \Induction(p, -1) \) holds for every
	\( p \geq 0 \).
	Furthermore, the diagram
	\[
		\begin{tikzcd}[row sep = large]
			1_\oplax
			\dar["w^{p+q}" swap]
			\rar["w^p"]
				& \oplax^p 1_\oplax
				\dar["{\oplax^p(-,w_q,-)}"]
				\\
			\oplax^{p+q} 1_\oplax
			\rar
			\dar
				& \oplax^p(-, \oplax^q,-) 1_\oplax
				\dar
				\\
			\oplax^{p+q+1} 1_\oplax
			\rar
				& \oplax^p(-, \oplax^{q+1}, -) 1_\oplax
		\end{tikzcd}
	\]
	commutes by virtue of \( \Induction(p,q)\) for the top square
	and by sequential decomposition for the  bottom square.
	In other terms, we showed
	\( \Induction(p,q) \implies \Induction(p, q+1) \) for
	every \( p \geq 0, q \geq -1 \)
	hence \( (1_\oplax, \{w^n\}_{n \geq -1}) \) is a comonoid.
\end{proof}

\begin{remark}
	Notice that the two parallel structure maps
	\( \oplax^0(1_\oplax)
	\rightrightarrows 1_\oplax \otimes_\oplax 1_\oplax \)
	may not be equal, but the two induced maps
	\( 1_\oplax \to 1_\oplax \otimes_\oplax 1_\oplax \)
	are equal.
\end{remark}

\subsection{Normal oplax monoidal categories}
\label{sec:normality}

\begin{definition}
	An oplax monoidal category \( \Category V^\oplax \) will be
	called \emph{normal} whenever the structure map
	\[
		\oplax^0(x) \longrightarrow x
	\]
	is invertible for every object \( x \in \Objects(\Category V) \).

	We shall say that \( \oplax \) is \emph{strictly normal}
	whenever \( \oplax^0(x) = x \),
	\[
		x = \oplax^0(x) \longrightarrow x
	\]
	is the identity of \( x \)
	for every \( x \in \Objects(\Category V) \) and the decomposition
	natural transformations
	\[
		\oplax^p \longrightarrow \oplax^0 \oplax^p
		\qand
		\oplax^p \longrightarrow \oplax^p(-, \oplax^0, -)
	\]
	are also identity natural transformations.
\end{definition}

\begin{remark}
	The terminology `normal' originates from
	Day and Street
	\cite{zbMATH01963527}.
	It is also used by Ching
	\cite{doi:10.1007/s40062-012-0007-2} to describe what
	we have called `strictly normal' structures.

	Notice in particular
	that the set of axioms defining strictly normal
	oplax monoidal categories is simpler: there is no need to
	specify \( \oplax^0 \) and the counitality axioms can be removed.
\end{remark}

\begin{remark}[(Strictification)]
	The distinction between `normal' and `strictly normal' is
	almost irrelevant since one can always strictify
	a normal structure.

	Every oplax monoidal structure \( \oplax \) on a
	category \( \Category V \) has an obvious underlying
	strictly normal oplax monoidal structure
	\( \Underlying \oplax \) obtained by replacing \( \oplax^0 \)
	with the identity functor and every decomposition natural
	transformation involving \( \oplax^0 \) with an identity.

	In addition, the identity functor of \( \Category V \)
	has an obvious structure of oplax monoidal functor
	\[
		\oplax \implies \Underlying \oplax
	\]
	which becomes an isomorphism of oplax monoidal structures on
	\( \Category V \) whenever \( \oplax \) is normal.
\end{remark}

\subsection{Comonad twist of an oplax monoidal category}
\label{sec:comonad_twist}

We shall expand the example of oplax monoidal category
given earlier using a comonad \( W \) and the coproduct
\UnskipRef{example comonad coproduct}.
That example is a special case of a general construction
where one can twist an oplax monoidal structure with
a lax monoidal comonad, which in turn provides a plethora
of examples of \emph{non-normal} oplax monoidal structures.

\begin{definition}
	Let \( \Category V^\oplax \) be an oplax monoidal category.
	A lax monoidal comonad on \( \Category V^\oplax \) is the
	data of a comonad
	\( (W \From \Category V \to \Category V, w \From W \implies W^2,
	t \From W \implies \IdentityFunctor) \) on \( \Category V \)
	together with a
	lax monoidal structure \( \{l^n\}_{n \geq -1} \)
	on \( W \From \Category V \to \Category V \) such that
	\( w \) and \( t \) are monoidal natural transformations.
\end{definition}

\begin{example}
	If \( \Category V \) admits finite coproducts, then any comonad
	admits a canonical structure of lax monoidal comonad for the
	coproduct
	strong monoidal structure \( \Category V^\amalg \).
\end{example}

\begin{proposition}
	Given a lax monoidal comonad \( (W, w, t, \{l^n\}_{n \geq -1}) \)
	on an oplax monoidal category \( \Category V^\oplax \), one
	can obtain a new oplax monoidal structure on \( \Category V \)
	by setting
	\[
		\psi^n \coloneqq \oplax^n W
	\]
	for every \( n \geq -1 \).
	The counit natural transformation
	\( \psi^0 \implies \IdentityFunctor \) is given
	as the composite
	\[
		\begin{tikzcd}
			\oplax^0 W
			\rar
				& W
				\rar["t"]
					& \IdentityFunctor
		\end{tikzcd}
	\]
	and the decompositions are given by
	\[
		\begin{tikzcd}[column sep = huge]
			\oplax^{p+q}W
			\ar[dr]
			\\
				&[-20ex] \oplax^p(-, \oplax^q, -) W
				\rar["{\oplax^p(-, \oplax^q w,-)}"]
				& \oplax^p (-, \oplax^q W, -) W
				\rar["{\oplax^p(-,l^q,-)W}"]
				& \oplax^p W (-, \oplax^q W, -).
		\end{tikzcd}
	\]
\end{proposition}

\begin{proof}
	The two counitality axioms follow straightforwardly
	from the counitality axioms of the
	comonad, the counit axioms of \( \oplax \),
	the fact that \( t \) is a monoidal natural
	transformation and the naturality of the counit natural 
	transformation associated with \( \oplax \).

	Parallel decomposition comes from parallel decomposition
	for \( \oplax \) and the fact that \( w \) is a monoidal
	natural transformation.

	Sequential decomposition
	follows from sequential decomposition for \( \oplax \), the coassociativity of
	\( W \) and the fact that \( w \)
	is a monoidal natural transformation.
\end{proof}

\subsection{Oplax monoidal categories as multicategories}
\label{sec:multi}

Recall that usually the words `multicategory' and
(set valued, coloured) `operads' are synonymous.
In the present context, it will be useful to introduce a slightly more general definition
of a multicategory so as to be able to compare it with
non-normal oplax monoidal categories.

\begin{definition}[(Multicategory)]
\label{definition:Multicategory}
	A multicategory is the data of a category \( \Category M \),
	a coloured operad \( \Category M_0 \) and a morphism
	of coloured operads
	\( \Category M \to \Category M_0 \) (where \( \Category M \)
	is considered as a coloured operad with only operations of arity
	\( 1 \)), which is the identity on objects.

	Given two objects \( a, b \in \Objects(\Category M) \),
	the set \( \Category M \inout a b \) will be called the set
	of \emph{strict morphisms} between \( a \) and \( b \) while
	the set \( \Category M_0 \inout a b \) will be referred to as the
	set of \emph{weak morphisms} between
	\( a \) and \( b \).

	A multicategory will be called \emph{normal} whenever
	\( \Category M \inout a b = \Category M_0 \inout a b \)
	for all objects \( a, b \in \Objects(\Category M) \),
	in which case \( \Category M \) is simply the underlying
	category of the operad \( \Category M_0 \).

	One can form a 2-category \( \Multi \) of multicategories in
	an obvious way where 1-morphisms are given by
	compatible pairs of operad morphisms \( (F \From \Category M \to
	\Category N, F_0 \From \Category M_0 \to \Category N_0) \)
	and 2-morphisms are given by
	compatible pairs
	of transformations
	\( (\alpha \From F \implies G, \alpha_0 \From F_0 \implies G_0) \).
\end{definition}

\begin{remark}
	The above definition of multicategories \UnskipRef{definition:Multicategory}
	corresponds to the
	definition of `lax promonoidal categories' of Day and
	Street \cite[6.3]{zbMATH01963527}.
\end{remark}

One of the first, and perhaps simplest, example
of multicategory derives from a strict
monoidal category. Similarly, any oplax monoidal category
induces a multicategory:
given an oplax monoidal category \( \Category V^\oplax \),
the operadic structure is given by
\[
		\Category V_0 \inout \ba b \coloneqq
		\Category V \inout {\oplax(\ba)} b
\]
for every object \( b \in \Objects(\Category V) \)
and every finite sequence \( \ba \) of elements of
\( \Category V \).
The composition maps are given by
\[
	\begin{tikzcd}[sep=large]
		\Category V\inout{\oplax^p(-,x,-)}{-}
		\times \Category V\inout{\oplax^q(-)}{x}
		\rar
		& \Category V\inout{\oplax^p(-,\oplax^q,-)}{-}
		\rar
		& \Category V\inout{\oplax^{p+q}(-)}{-}
	\end{tikzcd}
\]
where the first arrow is simply the composition of morphisms
in \( \Category V \) and the second one is defined using
the associator of \( \Category V^\oplax \).
Given two objects \( a, b \in \Objects(\Category V) \),
the maps
\[
	\Category V \inout a b \longrightarrow
	\Category V_0 \inout a b \coloneqq \Category V
	\inout {\oplax^0(a)} b
\]
are obtained using the counit \( \oplax^0(a) \to a \).

\begin{remark}
	It is straightforward to check that the assignment
	above
	extends to a fully faithful \( 2 \)\=/functor
	\[
		\begin{tikzcd}
			\Oplax
			\rar[hook]
			& \Multi
		\end{tikzcd}
	\]
	thereby generalising the normal case, described by Aguiar, Hiam and
	López Franco
	\cite[1.7]{doi:10.1007/s10485-017-9497-8}.

	Notice in particular that,
	an oplax monoidal category is normal if
	and only if the multicategory it generates is normal.
\end{remark}

\begin{remark}
	In the case where a multicategory is represented by
	an oplax monoidal category \( \Category V^\oplax \),
	the map sending strong morphisms with source
	\( \oplax(\ba) \) and target \( b \)
	to weak morphisms with source \( \oplax(\ba) \)
	and target \( b \)
	\[
		\begin{tikzcd}
		\Category V \inout{\oplax^p(\ba)}{b}
		\rar["\IsIsomorphicTo"]
			&
			\Category V \inout{\oplax^0\oplax^p(\ba)}{b}
		\end{tikzcd}
	\]
	is a bijection.
\end{remark}

\section{Lax-oplax duoidal categories}
\label{sec:lax-oplax_duoidal}

We review a generalisation of the notion of 
duoidal categories to the oplax setting, originally
introduced by B\"ohm and Vercruysse \cite{doi:10.1090/conm/771}. 

\subsection{Definition and notations}

\begin{notation}
\label{notation:natural transformation tau of Cat}
	We shall denote by \( \CanonicalTwist^{p,q} \)
	the permutation natural
	transformation stemming from the symmetric monoidal structure
	of \( \Cat \) and whose components are given by
	\begin{align*}
		\CanonicalTwist^{p,q}
		& \From (\Category C_{0,0} \times \dots \times
		\Category C_{0,q}) \times \dots \times (\Category C_{p,0}
		\times \dots \times \Category C_{p,q}) \\
		& \qquad \qquad \qquad
		\longrightarrow \quad (\Category C_{0,0} \times \dots
		\times \Category C_{p,0}) \times \dots \times (\Category C_{0,q}
		\times \dots \times \Category C_{p,q})
	\end{align*}
	where \( \Category C_{i,j} \) are categories
	for \( 0 \leq i \leq p \) and \( 0 \leq j \leq q \).
	In plain words, \( \CanonicalTwist^{p,q} \)
	is the functor that sends the
	product of \( p+1 \) products of \( q+1 \) categories,
	to the product of \( q+1 \) products
	of \( p+1 \) categories.
	One has
	\[
		\CanonicalTwist^{1,1}(x_0, x_1, x_2, x_3)
		= (x_0, x_2, x_1, x_3)
	\]
	for example.
\end{notation}

\begin{definition}[(Lax-oplax duoidal category)]
	\label{def:lax-oplax}
	A lax-oplax duoidal structure on a category \( \Category D \)
	is the data of a lax monoidal structure
	\( \Category D^\lax \) and an oplax monoidal structure
	\( \Category D^\oplax \) together with natural transformations
	\[
		\begin{tikzcd}[column sep = huge]
			\oplax^p \circ (\underbrace{\lax^q
			\times \dots \times \lax^q}_{p+1\, \text{times}})
			\rar["\chi^{p,q}"]
				&
				\lax^q \circ (\underbrace{\oplax^p \times
				\dots \times \oplax^p}_{q+1\, \text{times}}) \circ
				\CanonicalTwist^{p,q}
		\end{tikzcd}
	\]
	for every \( p, q \geq -1 \),
	such that
	\begin{itemize}
		\item
			the transformations \( \chi^{p, q} \) endow
			\[
				\begin{tikzcd}
					(\underbrace{\Category D \times \dots \times
					\Category D}_{q+1 \text{ times}})^\oplax
					\rar["\lax^q"]
						& \Category D^\oplax
				\end{tikzcd}
			\]
			with a lax monoidal structure, for every
			\( q \geq -1 \);
		\item
			the transformations
			\( \chi^{p,q} \CanonicalTwist^{p,q} \)
			endow
			\[
				\begin{tikzcd}
					(\underbrace{\Category D \times \dots \times
					\Category D}_{p+1 \text{ times}})^\lax
					\rar["\oplax^p"]
						& \Category D^\lax
				\end{tikzcd}
			\]
			with an oplax monoidal structure
			for every \( p \geq -1 \).
	\end{itemize}
\end{definition}

\begin{remark}
	For the case \( (p,q) = (-1,-1) \), we get a morphism
	\[
		1_\oplax \longrightarrow 1_\lax
	\]
	and for the case \( (p,q) = (1,1) \), we get morphisms
	\[
		(a \otimes_\lax b)
		\otimes_\oplax (c \otimes_\lax d)
		\longrightarrow
		(a \otimes_\oplax c)
		\otimes_\lax (b \otimes_\oplax d)
	\]
	for every tuple \( (a, b, c, d) \)
	of objects of \( \Category D \).
	When both \( \lax \) and \( \oplax \) are strong, these 
	correspond to the structure maps of a (usual) duoidal structure
	on \( \Category D \) \cite{doi:10.1090/crmm/029}.
\end{remark}

\begin{remark}
	For the case \( p=-1 \), we obtain structure maps
	\[
		1_\oplax
		\longrightarrow \lax^q(1_\oplax, \dots, 1_\oplax)
	\]
	turning \( 1_\oplax \) into a \( \lax \)-comonoid.

	Symmetrically, for \( q=-1 \), we get structure maps
	\[
		\oplax^p(1_\lax, \dots, 1_\lax)
		\longrightarrow 1_\lax
	\]
	turning \( 1_\lax \) into an \( \oplax \)\=/monoid.
\end{remark}

\subsection{Bimonoids in lax-oplax duoidal categories}


\begin{proposition}[{\cite[Th. 4.4]{doi:10.1090/conm/771}}]
	\label{prop:monoids=lax+comonoid=oplax}
	Let \( \Category D^{\lax, \oplax} \) be
	a lax-oplax duoidal category.
	Then \( \lax \) induces a lax monoidal structure on
	the category
	of monoids \( \mathsf{Mon}(\Category D^\oplax) \)
	with respect to the oplax monoidal structure \( \oplax \).
	Similarly, \( \oplax \) induces an oplax monoidal structure
	on the category of comonoids
	\( \mathsf{Comon}(\Category D^\lax) \) with respect to
	the lax monoidal structure \( \lax \).
\end{proposition}

We can now recall the definition of a bimonoid in a lax-oplax
duoidal category.

\begin{definition}[(Bimonoid)]
\label{definition: bimonoid}
  A bimonoid in a lax-oplax duoidal category
  \( \Category D^{\lax, \oplax} \)
  is a comonoid in the category of monoids of $\Category D^\oplax$, or
  equivalently, a monoid in the category of comonoids in $\Category D^\lax$.
\end{definition}

Very concretely, a bimonoid in
\( \Category D^{\lax, \oplax} \)
consists in an object $A$ in $\Category D$, together with morphisms
\[
    m_A^n \From \oplax^n(A,\dots,A) \to A
    \qand
    w^n_A \From A \to \lax^n(A,\dots,A)
\]
for any integer $n \geq -1$,
so that $(A, m_A)$ be a monoid in $\Category D^\oplax$ and $(A, w_A)$ be a comonoid
in $\Category D^\lax$.
On top of that, these morphisms have to be such that the diagram
\[
    \begin{tikzcd}[column sep=large]
	\oplax^p A
	\arrow[r,"\oplax^pw^q_A"] \arrow[dd,"m^p_A"] 
	& 
	\oplax^p\lax^q  A
	\arrow[d,"\chi^{p,q}"] 
	\\
        &
        \lax^q\oplax^p A\arrow[d,"\lax^q m^p_A"] 
        \\
	A \arrow[r,"w^q_A"]            
	& 
	\lax^qA          
     \end{tikzcd}
\]
commute for any integers $p,q \geq -1$. These diagrams encode
the requirement that the maps $w_A$ are monoid morphisms, or equivalently
that the maps $m_A$ are comonoid morphisms.

Thanks to the reconstruction theorem for monoids in an oplax monoidal
category, the data needed to describe a bimonoid in a lax-oplax
duoidal category is similar to the data needed to describe a bimonoid
in a usual duoidal category.

\begin{proposition}
	\label{Truncation of bimonoids}
	The data of a bimonoid \( A \)
	in a lax-oplax duoidal category \( \Category D^{\lax, \oplax} \)
	is equivalent to the data of:
	\begin{description}
	\item[Product and unit]
	a product map and a unit map
	\[
		\mu_A \From A \otimes_\oplax A \to A
		\qand
		\eta_A \From 1_\oplax \to A\,,
	\]
	such that the associativity and unitality diagrams
	\[
		\hspace{-15pt}
		\begin{tikzcd}[column sep=tiny]
			&[-5ex] A \otimes_\oplax A \otimes_\oplax A
			\ar[dl]
			\ar[dr] &[-5ex] \\
			(A \otimes_\oplax A) \otimes_\oplax A
			\ar[d, "{\mu_A \otimes_\oplax -}"]
			&& A \otimes_\oplax (A \otimes_\oplax A)
			\ar[d, "{- \otimes_\oplax \mu_A}"] \\
			A \otimes_\oplax A
			\ar[dr, swap, "\mu_A"]
			&& A \otimes_\oplax A
			\ar[dl, "\mu_A"] \\
			& A &
		\end{tikzcd}
		\qquad
		\begin{tikzcd}[column sep=small]
			&[-2ex] \oplax^0(A)
			\ar[dl]
			\ar[ddd]
			\ar[dr] &[-2ex] \\
			1_\oplax \otimes_\oplax A
			\ar[d, "{\eta_A \otimes_\oplax -}"]
			&& A \otimes_\oplax 1_\oplax
			\ar[d, "{- \otimes_\oplax \eta_A}"] \\
			A \otimes_\oplax A
			\ar[dr, swap, "\mu_A"]
			&& A \otimes_\oplax A
			\ar[dl, "\mu_A"] \\
			& A &
		\end{tikzcd}
	\]
	commute;
	\item[Coproduct and counit]
	A coproduct and a counit
	\[
		\delta_A \From A \to A \otimes_\lax A
		\qand
		\epsilon_A \From A \to 1_\lax\,,
	\]
	such that the coassociativity and counitality diagrams
	
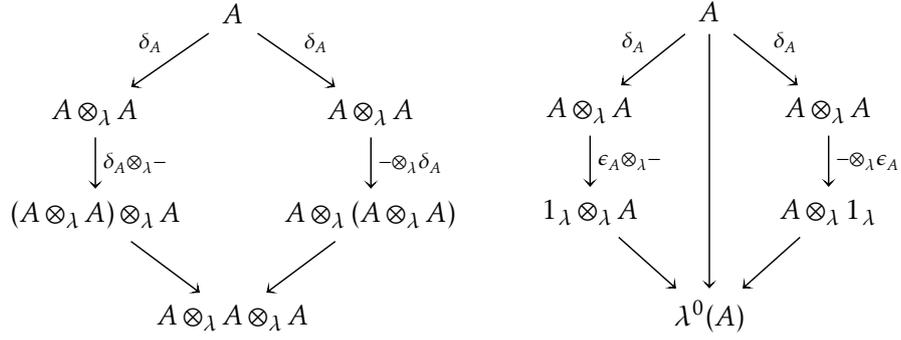
\[
		\hspace{-15pt}
		\begin{tikzcd}[column sep=tiny]
			&[-5ex] A
			\ar[dl, swap,, "\delta_A"]
			\ar[dr, "\delta_A"] &[-5ex] \\
			A \otimes_\lax A
			\ar[d, "{\delta_A \otimes_\lax -}"]
			&& A \otimes_\lax A
			\ar[d, "{- \otimes_\lax \delta_A}"] \\
			(A \otimes_\lax A) \otimes_\lax A
			\ar[dr]
			&& A \otimes_\lax (A \otimes_\lax A)
			\ar[dl] \\
			& A \otimes_\lax A \otimes_\lax A &
		\end{tikzcd}
		\qquad
		\begin{tikzcd}[column sep=small]
			&[-2ex] A
			\ar[dl, swap, "\delta_A"]
			\ar[ddd]
			\ar[dr, "\delta_A"] &[-2ex] \\
			A \otimes_\lax A
			\ar[d, "{\epsilon_A \otimes_\lax -}"]
			&& A \otimes_\lax A
			\ar[d, "{- \otimes_\lax \epsilon_A}"] \\
			1_\lax \otimes_\lax A
			\ar[dr]
			&& A \otimes_\lax 1_\lax
			\ar[dl] \\
			& \lax^0(A) &
		\end{tikzcd}
	\]
	commute;
	\end{description}
	which are also required to satisfy the compatibility
	conditions
	
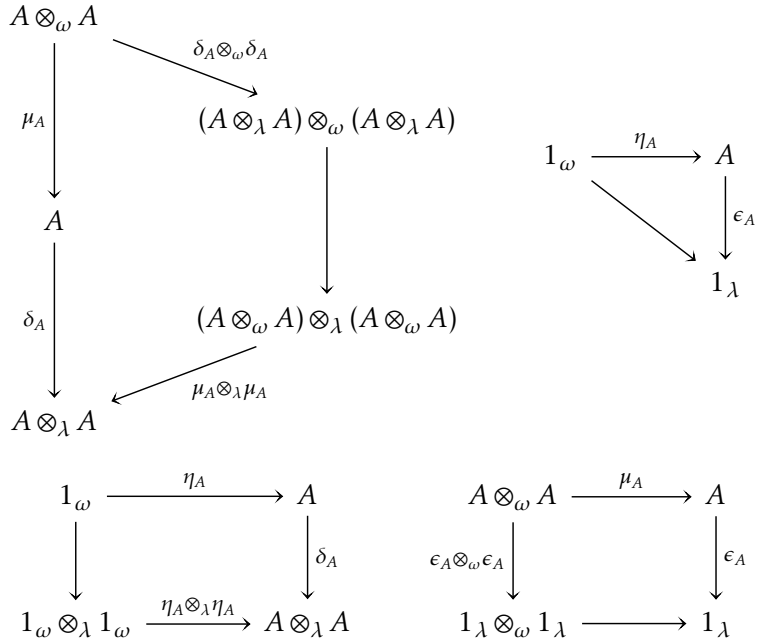
\[
		\begin{tikzcd}
			A \otimes_\oplax A
			\ar[dd, swap, "\mu_A"]
			\ar[dr, "\delta_A \otimes_\oplax \delta_A"] \\
			& (A \otimes_\lax A) \otimes_\oplax (A \otimes_\lax A)
			\ar[dd] \\
			A
			\ar[dd, swap, "\delta_A"] \\
			& (A \otimes_\oplax A) \otimes_\lax (A \otimes_\oplax A)
			\ar[dl, "\mu_A \otimes_\lax \mu_A"] \\
			A \otimes_\lax A
		\end{tikzcd}
		\qquad
		\begin{tikzcd}[sep=large]
			1_\oplax
			\arrow[dr]
			\arrow[r, "\eta_A"]
			& A 
			\arrow[d, "\epsilon_A"{right}] \\
			& 1_\lax
		\end{tikzcd}
	\]
	\[
		\begin{tikzcd}[sep=large]
			1_\oplax
			\arrow[d]
			\arrow[r, "\eta_A"]
			& A
			\arrow[d, "\delta_A"{right}]\\
			1_\oplax \otimes_\lax 1_\oplax
			\arrow[r, "\eta_A \otimes_\lax \eta_A"]
			& A \otimes_\lax A
		\end{tikzcd}
		\qquad
		\begin{tikzcd}[sep=large]
			A \otimes_\oplax A
			\arrow[d, "\epsilon_A
			\otimes_\oplax \epsilon_A"{left}]
			\arrow[r, "\mu_A"]
			& A
			\arrow[d, "\epsilon_A"{right}] \\
			1_\lax \otimes_\oplax 1_\lax
			\arrow[r]
			& 1_\lax
		\end{tikzcd}
	\]
	which impose that \( \mu_A \) and \( \eta_A \)
	are comonoid morphisms, or equivalently \( \delta_A \)
	and \( \epsilon_A \) are monoid morphisms.
\end{proposition}
\begin{proof}
	The truncations of a monoid structure in \( \Category D^\oplax \)
	and of a comonoid structure in \( \Category D^\lax \)
	follow from the reconstruction theorem 
	\UnskipRef{prop:truncation_lax_functors} on certain lax and
	oplax monoidal functors, whereas the truncations of the monoid
	and comonoid morphism condition follow from the truncation
	of the monoidality condition between such functors
	\UnskipRef{lemma:truncation_monoidality}.
\end{proof}

The notion of bimonoid in a lax-oplax duoidal category
extends the one of bimonoid in a 
duoidal category \cite{doi:10.1090/crmm/029}
which itself generalises the notion of bimonoid
in a braided monoidal category.
The princeps of such a bimonoid in a braided monoidal category
is the one of bialgebras.
A particularly interesting example of bimonoid in a genuine
(\ie not braided) duoidal category
is given by bialgebroids over a commutative ring
$R$ defined as bimonoids in the duoidal category of
$(R,R)$\=/bimodules \cite[6.44]{doi:10.1090/crmm/029}. 
However bialgebroids over a non-commutative ring
do not admit a characterisation as 
bimonoids in a duoidal category \cite[6.45]{doi:10.1090/crmm/029}.
The relevant framework in this case is given by
lax-strong duoidal categories, that is
lax-oplax duoidal categories for which the oplax
monoidal structure is strong monoidal.

\subsection{Lax-strong duoidal categories}

The example that we shall give of a lax-oplax duoidal category
\SectionRef{sec:Re-bimod} is in fact lax-strong.
In the case where the oplax structure is strong,
the description of the lax-oplax structure can be simplified.

\begin{proposition}[(Lax-strong duoidal category)]
	\label{prop: lax-strong}
	A lax-strong duoidal structure on a category 
	\( \Category D \) can be described as the data of a lax monoidal
	structure \( \Category D^\lax \), a strong monoidal 
	structure \( \Category D^\otimes \),
	together with natural transformations
	\[
		\begin{tikzcd}
			\lax^n \otimes \lax^n
			\rar["\chi^n"]
				& 
				\lax^n \circ (\otimes \times \dots \times \otimes)
				\circ \CanonicalTwist^{1,n}
		\end{tikzcd}
	\]
	and morphisms
	\[
		\begin{tikzcd}
			1_\otimes
			\rar["w^n"]
				& \lax^n 1_\otimes	
		\end{tikzcd}
	\]
	for all integers \( n \geq -1 \) such that
	\begin{itemize}
		\item
			each pair \( (\chi^n, w^n) \) endows
			\[
				\lax^n\From (\Category D \times \dots \times \Category D)^\otimes
				\longrightarrow \Category D^\otimes
			\]
			with a structure of lax monoidal functor
			(between strong monoidal categories);
		\item
			the natural transformations
			\( \chi^n \circ \CanonicalTwist^{1,n} \) endow
			\[
				\otimes \From
				(\Category D \times \Category D)^\lax
				\longrightarrow
				\Category D^\lax
			\]
			with a structure of oplax monoidal functor
			(between lax monoidal categories);
		\item
			the sequence \( \{w^n\}_{n \geq -1} \)
			endows the unit object \( 1_\otimes \)
			with a structure of comonoid
			in \( \Category D^\lax \).
	\end{itemize}
\end{proposition}
\begin{proof}
	According to the previously given definition
	of a lax-oplax duoidal category
	\( \Category D^{\lax,\oplax} \)
	\UnskipRef{def:lax-oplax}, the functors
	\( \lax^n \) are lax functors between oplax
	monoidal categories. However, in the special
	case where \( \oplax=\otimes \) is strong monoidal,
	the proposition on the reconstruction
	of the lax structure of a functor
	\UnskipRef{prop:truncation_lax_functors}
	applies and tells us that for every integer
	\( n \geq -1 \), the natural transformations
	\[
		\begin{tikzcd}[sep=large]
			\oplax^k\, \lax^n
			\ar[r, "\chi^{k,n}"]
			& \lax^n \oplax^k \CanonicalTwist^{k,n}
		\end{tikzcd}
	\]
	are completely determined by 
	\( w^n \coloneqq \chi^{-1,n} \) and
	\( \chi^n \coloneqq \chi^{1,n} \).
	The conditions listed in the above proposition
	follow from this identification.
\end{proof}

Let us spell out the conditions to be verified by
the data of a lax-strong duoidal category 
\( (\Category D, \lax, \otimes, \chi, w) \).
That each pair \( (\chi^n,w^n) \) defines a structure
of lax monoidal functor on \( \lax^n \) imposes that
the associativity diagram
\[
	\begin{tikzcd}[sep=large]
		\big(\lax^n(\Collection x) \otimes \lax^n(\Collection y)\big)
						\otimes \lax^n(\Collection z)
		\ar[d, "{\chi^n \otimes -}"]
		\ar[rr]
		&& \lax^n(\Collection x) \otimes \big(\lax^n(\Collection y)
						\otimes \lax^n(\Collection z)\big)
		\ar[d, "{- \otimes \chi^n}"] \\
		\lax^n(\Collection{x \otimes y}) \otimes \lax^n(\Collection z)
		\ar[d, "\chi^n"]
		&& \lax^n(\Collection x) \otimes \lax^n(\Collection{y \otimes z})
		\ar[d, "\chi^n"] \\
		\lax^n\big(\Collection{(x \otimes y) \otimes z}\big)
		\ar[rr]
		&& \lax^n\big(\Collection{x \otimes (y \otimes z)}\big)
	\end{tikzcd}
\]
as well as the two unitality diagrams
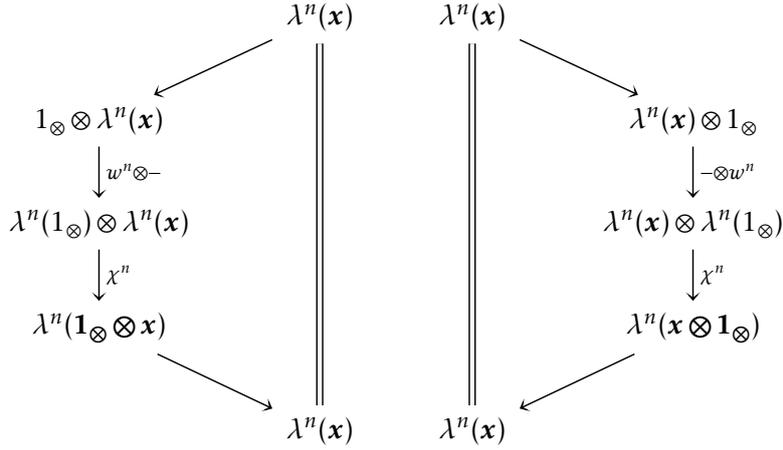
\[
	\begin{tikzcd}[sep=normal]
		& \lax^n(\Collection x)
		\ar[dl]
		\arrow[dddd, equal] \\
		1_\otimes \otimes \lax^n(\Collection x)
		\ar[d, "{w^n \otimes -}"] & \\
		\lax^n(1_\otimes) \otimes \lax^n(\Collection x)
		\ar[d, "\chi^n"] & \\
		\lax^n(\Collection{1_{\otimes} \otimes x})
		\ar[dr] & \\
		& \lax^n(\Collection x)
	\end{tikzcd}
	\qquad
	\begin{tikzcd}[sep=normal]
		\lax^n(\Collection x)
		\ar[dr]
		\arrow[dddd, equal] & \\
		& \lax^n(\Collection x) \otimes 1_\otimes
		\ar[d, "{- \otimes w^n}"] \\
		& \lax^n(\Collection x) \otimes \lax^n(1_\otimes)
		\ar[d, "\chi^n"] \\
		& \lax^n(\Collection{x \otimes 1_{\otimes}})
		\ar[dl] \\
		\lax^n(\Collection x) &
	\end{tikzcd}
\]
commute for any integer \( n \geq -1 \).
The requirement that \( \chi^n \CanonicalTwist^{1,n} \)
also endow \( \otimes \) with a structure of oplax
monoidal functor, means that the unitality diagram
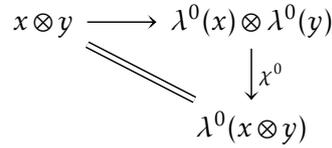
\[
	\begin{tikzcd}
		x \otimes y
		\ar[dr, equal]
		\rar
		& \lax^0(x) \otimes \lax^0(y)
		\arrow[d, "\chi^0"] \\
		& \lax^0(x \otimes y)
	\end{tikzcd}
\]
commutes for any pair of objects \( x, y \) of
\( \Category D \), the additivity diagrams
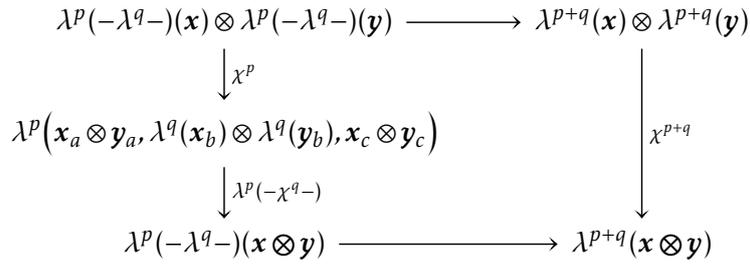
\[
	\begin{tikzcd}
		\lax^p(-\lax^q-)(\Collection x)
				\otimes \lax^p(-\lax^q-)(\Collection y)
		\ar[d, "\chi^p"]
		\rar
		& \lax^{p+q}(\Collection x) \otimes \lax^{p+q}(\Collection y)
		\ar[dd, "\chi^{p+q}"] \\
		\lax^p\big(\Collection{x}_a \otimes \Collection{y}_a,
			\lax^q(\Collection{x}_b) \otimes \lax^q(\Collection{y}_b),
				\Collection{x}_c \otimes \Collection{y}_c\big)
		\ar[d, "\lax^p(-\chi^q-)"] & \\
		\lax^p(-\lax^q-)(\Collection{x \otimes y})
		\rar
		& \lax^{p+q}(\Collection{x \otimes y})
	\end{tikzcd}
\]
as well as the comonoid condition diagrams
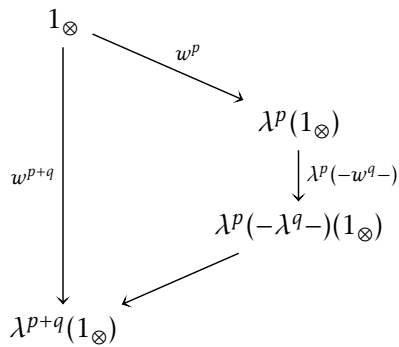
\[
	\begin{tikzcd}
		1_\otimes
		\ar[dr, "w^p"]
		\ar[ddd, swap, "w^{p+q}"] & \\
		& \lax^p(1_\otimes)
		\ar[d, "\lax^p(-w^q-)"] \\
		& \lax^p(-\lax^q-)(1_\otimes)
		\ar[dl] \\
		\lax^{p+q}(1_\otimes) &
	\end{tikzcd}
\]
commute for any \( p \geq 0 \) and \( q \geq -1 \).

\section{Categories enriched over oplax monoidal categories}
\label{Sec: oplax enriched}

Oplax monoidal categories can be seen as the oplax monoids
of \( \Cat^\times \).
Thus one is naturally led to study its modules
\SectionRef{sec:V-modules}.
As one could expect, modules over
an oplax monoidal category consist in a category together
with infinitely many functors (instead of a single bifunctor)
encoding the action of the (arbitrarily many copies of the)
oplax monoidal category on it.
We shall then repeat
this process of constructing a lax monoid and its lax modules
in a bigger bicategory
\( \Cat \subset \Cat_\deltaup \),
which is obtained from \( \Cat \) by replacing functors
with distributors.
We shall refer to these modules
as \( \deltaup \)modules, and show that a particular class
of the latter, namely those which are left representable
(defined below) are in fact equivalent to ordinary modules
over an oplax monoidal category.

One can also identify another class of \( \deltaup \)modules,
which we shall call right representable, and relate them to
categories endowed with additional structures defined in terms
of an oplax monoidal structure. We shall identify the latter as categories
enriched over an oplax monoidal category. Recall that
given a monoidal category \( \Category V^\otimes \),
a \( \Category V \)-category \( \Category C \) has a set of objects
\( \Objects(\Category C) \) and \( \Category V \)-hom objects
\( [x, y] \in \Objects(\Category V) \) for every pair
\( x, y \in \Objects(\Category C) \). There is also a composition map
\( \mu \From [y, z] \otimes [x, y] \to [x, z] \).
Now if \( \Category V^\oplax \) is an oplax monoidal category,
one expects to also have composition maps
\( [y, z] \otimes_\oplax [x, y] \to [x, z] \)
and also countably many higher composition maps \( \mu^n \) such as
\[
	\begin{tikzcd}
		{[y, z] \otimes_\oplax [x, y] \otimes_\oplax [w, x]}
		\rar["\mu^2"]
			& {[w, z]}
	\end{tikzcd}
\]
which are evident when \( \Category V^\oplax \) is strong.

There are a few new phenomena that appear
when enriching over an oplax monoidal category:
\begin{enumerate}
	\item
		By the \emph{oplax} nature of the tensor structure, one can
		always recover the higher composition maps \( \mu^n \)
		for \( n \geq 2 \) from \( \mu^1 \) and \( \mu^{-1} \).
		This leads to two equivalent definitions
		for \( \Category V \)-categories:
		an \emph{extended} one including
		all maps \( \mu^n \) and an \emph{abridged} one, including only
		\( \mu^1 \) and \( \mu^{-1} \).
	\item
		By the \emph{non-normal} nature of the oplax monoidal structure,
		even though a \( \Category V \)-category---by which we mean
		a collection of objects \( x, y, \dots \) and hom-objects
		\( [x,y], \dots \),
		together with composition maps as above---always
		possesses an underlying
		\( \Sets^\times \)-category with set of arrows \( x \to y \)
		given by the maps \( 1_\oplax \to [x, y] \), the object
		\( [x, y] \) is usually \emph{not functorial} in \( x \) and
		\( y \).
		Instead it is \( [x, y]_\oplax \coloneqq \oplax^0([x, y]) \)
		which is functorial. This leads us to define a category
		enriched over \( \Category V \) as a category \( \Category C \)
		endowed with additional structure, namely a bifunctor
		\( [-,-] \From \Category C \times \Category C\Op \to \Category V \)
		and natural transformations \( \mu \) between tensor products
		of \( [-,-] \). In other words, we \emph{require} the enrichment
		structure to be functorial instead of recovering it as a byproduct,
		as happens when enriching over a strong monoidal category.
\end{enumerate}
We shall spell out in more details the notion of categories enriched
over an oplax monoidal category \( \Category V^\oplax \)
\SectionRef{sec:Cat_V},
and show that they are equivalent to the aforementioned
right representable \( \Vdelta \)modules
\SectionRef{sec:right-rep}.
We shall then compare categories enriched over \( \Category V \) with
the notion of \( \Category V \)\=/categories, which is obtained
by trying to mimic as closely as possible the standard construction
of enrichment over strong monoidal categories, \ie without
requiring the enrichment bifunctor or the associated composition
maps to be functorial \SectionRef{sec:V-Cat&Cat_V}.
Finally, we shall finish by presenting an example of category
enriched over an oplax monoidal category, in the context of 
the theory of operads \SectionRef{sec:operads}.

To help jumping between definitions, we shall continue using the
notations \( 1_\oplax, x \otimes_\oplax y, x \otimes_\oplax y
\otimes_\oplax z \)
or \( \oplax^p(x_0, \dots, x_p) \) depending on when it is the most
appropriate.

\subsection{\( \Category V \)\=/modules}
\label{sec:V-modules}

In this subsection, we shall fix \(\Category V^\oplax\)
to be an oplax monoidal category. As the latter
is nothing but an oplax monoid in \( \Cat^\times \),
one can consider the 2-category
\( \VMods \)
of modules over it with lax maps, which 
will be defined hereafter.

\begin{definition}[(\( \Category V \)\=/module)]
    A \( \Category V \)\=/module \( \Category M^\rho \)
    is the data of a category \( \Category M \)
    equipped with
    \begin{itemize}
    \item structure maps
    \[
        \begin{tikzcd}
            \Category V^n \times \Category M
            \ar[r, "\rho^n "]
            & \Category M
        \end{tikzcd}
    \]
    for any \( n \geq 0 \);
    \item natural transformations
    \[
        \begin{tikzcd}
            \rho^{p+q}
            \ar[r]
            & \rho^p (-,
			\underset{i\text{-th}}{\underset{\uparrow}{\oplax}^q},-)
        \end{tikzcd}
    \]
    for every \( p \geq 1 \), \( q \geq -1 \), and
    \( 0 \leq i < p \) is the number of entries
    on the left of \( \oplax^q \);
    \item natural transformations
    \[
        \begin{tikzcd}
            \rho^{p+q}
            \ar[r]
            & \rho^p (-, \rho^q)
        \end{tikzcd}
    \]
    for every \( p, q \geq 0 \);
    \item a natural transformation
    \[
        \begin{tikzcd}
            \rho^0
            \ar[r]
            & \IdentityFunctor_{\Category M}
        \end{tikzcd}
    \]
    \end{itemize}
    such that
    \begin{description}
    \item[Counitalities] 
        \[
        \begin{tikzcd}
            \rho^p
            \rar
            \ar[dr, equal]
            & \rho^0 \rho^p
            \dar \\
            & \rho^p
        \end{tikzcd}
    \]
    and
   \[
        \begin{tikzcd}
            \rho^p
            \rar
            \ar[dr, equal]
            & \rho^p(-,\rho^0)
            \dar \\
            & \rho^p
        \end{tikzcd}
        \qquad
        \begin{tikzcd}
            \rho^p
            \rar
            \ar[dr, equal]
            & \rho^p(-,\oplax^0,-)
            \dar \\
            & \rho^p
        \end{tikzcd}
    \]
    commute whenever they make sense;
    \item[Parallel decompositions] 
    	\[
    	    \begin{tikzcd}[column sep = small]
    	        \rho^{p+q+r}
    	        \rar
    	        \dar 
    	        & 
    	        \dar  \rho^{p+q}(-,\rho^r)\\
    	       \rho^{p+r}(-,\oplax^q,-)
    	        \rar 
    	        & \rho^p(-,\oplax^q,-,\rho^r)
    	    \end{tikzcd}
    	    \quad
    	    \begin{tikzcd}[column sep = small]
    	        \rho^{p+q+r}
    	        \rar
    	        \dar
    	        &
    	        \dar \rho^{p+q}(-,\oplax^r,-)\\
    	         \rho^{p+r}(-,\oplax^q,-)
    	        \rar 
    	        & \rho^p(-,\oplax^q,-,\oplax^r,-)
    	    \end{tikzcd}
    	\]
    	commute whenever they make sense;
    \item[Sequential decompositions]
    \[
        \begin{tikzcd}[column sep=small]
            \rho^{p+q+r}
            \rar
            \dar
            & 
            \dar \rho^{p+q}(-,\oplax^r,-)\\
            \rho^p(-,\rho^{q+r})
            \rar
            & \rho^p(-,\rho^q(-,\oplax^r,-))
        \end{tikzcd}
        \quad
        \begin{tikzcd}[column sep=small]
            \rho^{p+q+r}
            \rar
            \dar
            & \rho^{p+q}(-,\rho^r)
            \dar \\
            \rho^p(-,\rho^{q+r})
            \rar
            & \rho^p(-,\rho^q(-,\rho^r))
        \end{tikzcd}
    \]
    and
    \[
        \begin{tikzcd}[column sep=small]
            \rho^{p+q+r}
            \rar
            \dar
            & \rho^{p+q}(-,\oplax^r,-)
            \dar \\
            \rho^p(-,\oplax^{q+r},-)
            \rar
            & \rho^p(-,\oplax^q(-,\oplax^r,-),-)
        \end{tikzcd}
    \]
    commute whenever they make sense.
    \end{description}
\end{definition}

\begin{remark}
	Note that the counitality conditions, as well as 
	the parallel and sequential decomposition conditions,
	of a \( \Category V \)\=/module are simply obtained
	from those of an oplax monoidal category by replacing
	the oplax monoidal structure with a module structure
	map whenever it is possible. In particular,
	any oplax monoidal category is a module over itself.
\end{remark}

\begin{definition}[(Lax \( \Category V \)\=/module morphism)]
	A lax morphism between two modules
	\( \Category M^\rho \)
	and \( \Category N^\sigma \)
	is the data of a functor \( F \From \Category M \to \Category N \)
	together with natural transformations
    \[
        \begin{tikzcd}
        	\sigma^n (-, F)
        	\ar[r, "f^n"]
        	& F \rho^n
        \end{tikzcd}
    \]
    for all \( n \geq 0 \), such that the diagrams
    \[
    	\begin{tikzcd}[sep=large]
    		\sigma^0 F
			\ar[dr]
    		\ar[d, "f^0" swap]
    		&
    		\\
    		F \rho^0
    		\rar
    		& F
    	\end{tikzcd}
    	\qquad
        \begin{tikzcd}[sep=large]
            \sigma^{p+q}(-,F)
            \ar[d, "f^{p+q}" swap]
            \rar
            & \sigma^p(-,\oplax^q,-,F)
            \ar[d, "f^p"] \\
            F \rho^{p+q}
            \rar
            & F \circ \rho^p(-,\oplax^q,-)
        \end{tikzcd}
    \]
    and
    \[
        \begin{tikzcd}[sep=large]
            \sigma^{p+q}(-,F)
            \ar[dd, "f^{p+q}" swap]
            \rar
            & \sigma^p\big(-,\sigma^q(-,F)\big)
            \ar[d, "{\sigma^p(-,f^q)}"] \\
            & \sigma^p(-,F \rho^q)
            \ar[d, "f^p"] \\
            F \rho^{p+q}
            \rar
            & F \rho^p(-,\rho^q)
        \end{tikzcd}
    \]
    commute whenever they make sense.
\end{definition}

\begin{definition}[(Natural transformation of \( \Category V \)\=/module morphisms)]
    A natural transformation of
    \( \Category V \)\=/module morphisms
    between two
    lax morphisms
    \( F,G \From \Category M^\rho \to \Category N^\sigma \),
    with structure \( f \) and \( g \) respectively,
    is a natural transformation \( \alpha \From F \implies G \)
    such that the diagram
    \[
        \begin{tikzcd}[sep=large]
            \sigma^n(-, F)
            \ar[r, "f^n"]
            \ar[d, "{\sigma^n(-,\alpha)}" swap]
            & F \rho^n
            \ar[d, "\alpha_{\rho^n}"] \\
            \sigma^n(-, G)
            \ar[r, "g^n"]
            & G \rho^n
        \end{tikzcd}
    \]
    commutes for any \( n \geq 0 \).
\end{definition}

\subsection{\( \Vdelta \)modules}

We shall now introduce a `distributor version'
of \( \Category V \)-modules, that we shall call
\( \Vdelta \)modules and we shall show that
the 2-category of
\( \Category V \)-modules is equivalent
to a certain subcategory of left representable
\( \Vdelta \)modules.

\paragraph{Distributors}
Recall that distributors---initially introduced
by Bénabou
\cite{Benabou} and Lawvere
\cite{doi:10.1007/bf02924844}
under the name `bimodule'---are a particular kind of
\( \Sets \)\=/valued functors.
More precisely, a distributor
between two categories \( \Category B \) and \( \Category C \) is a functor
\( T \From \Category B\Op \times \Category C \to \Sets \), 
denoted \( T \From \Category B \nrightarrow \Category C \).
We shall write
\(
	T \inout{b}{c}
\)
for the set obtained by evaluating the distributor
\( T \) on an object \( b \) of \( \Category B \)
and an object \( c \) of \( \Category C \).
Note that distributors are also called `profunctors' or `bimodules'
while the 2-category \( \Cat_\deltaup \) of categories, distributors and
natural transformations is also denoted \( \mathsf{Prof} \)
in the literature.

The composition of two distributors
\( T \From \Category A \nrightarrow \Category B \) and
\( U \From \Category B \nrightarrow \Category C \),
is again a distributor
\( U \circ T \From \Category A \nrightarrow \Category C \),
defined as the integral
\[
	(U \circ T) \inout {a}{c} \coloneqq
	\Integral{b} U \inout{b}{c} \times T \inout{a}{b}\,,
\]
for any objects \( a \in \Objects(\Category A) \)
and \( c \in \Objects(\Category C) \). Note that this composition
is associative and unital (the identity distributor is simply
the hom-functor) only up to an isomorphism, so that
\( \Cat_\deltaup \) is a bicategory
\cite[Chap. 5]{doi:10.1017/9781108778657}.

Every functor \( F \From \Category B \to \Category C \)
defines a distributor
\( \Category C \inout {F(-)} -
\From \Category B \nrightarrow \Category C \)
and this assignment induces a 2-functor
\( \Cat^{2\text{-}\mathrm{op}} \to \Cat_\deltaup \).
Its essential image,
that we shall denote by
\( \Cat_{\deltaup}^{\mathrm{1-rep}} \),
is comprised of distributors that are called \emph{left representable}.

\begin{definition}[(\( \Vdelta \)module)]
    A lax \( \Vdelta \)module \( \Category M^R \)
    is a category \( \Category M \) equipped with
	\begin{itemize}
		\item 
    		distributors
    		\[
    		    R^n \From \Category V^n \times \Category M
    		    \nrightarrow \Category M
    		\]
    		for all natural numbers \( n \geq 0 \);
		\item
    		natural transformations
    		\[
    		    \begin{tikzcd}
    		        R^p (-,
					\underset{i\text{-th}}{\underset{\uparrow}{\oplax}^q},-)
    		        \ar[r]
    		        & R^{p+q}
    		    \end{tikzcd}
    		\]
			for every \( p \geq 1, q \geq -1 \) and every
			\( 0 \leq i < p \) representing the number of
			entries on the left of \( \oplax^q \);
		\item
    		natural transformations
    		\[
    		    \begin{tikzcd}
    		        R^p (-,R^q) 
    		        \ar[r]
    		        & R^{p+q}
    		    \end{tikzcd}
    		\]
			for every \( p, q \geq 0 \);
		\item
			and a natural transformation
    		\[
    		        \IdentityFunctor_{\Category M}
					\longrightarrow R^0
    		\]
	\end{itemize}
	such that
    \begin{description}
    \item[Unitalities] 
        \[
        \begin{tikzcd}
            \IdentityFunctor_{\Category M} \circ R^p
            \rar
            \ar[dr,
			"\rotatebox{-25}{$\IsIsomorphicTo$}", outer sep = -1.5pt]
            & R^0 \circ R^p
            \dar \\
            & R^p
        \end{tikzcd}
    \]
    and
   \[
        \begin{tikzcd}
            R^p(-, \IdentityFunctor_{\Category M})
            \rar
            \ar[dr, "\rotatebox{-25}{$\IsIsomorphicTo$}",
				outer sep = -1.5pt]
            & R^p(-,R^0)
            \dar \\
            & R^p
        \end{tikzcd}
        \qquad
        \begin{tikzcd}
            R^p(-, \IdentityFunctor_{\Category V}, -)
            \rar
            \ar[dr, "\rotatebox{-25}{$\IsIsomorphicTo$}",
				outer sep = -1.5pt]
            & R^p(-,\oplax^0,-)
            \dar \\
            & R^p
        \end{tikzcd}
    \]
    commute whenever they make sense;
    \item[Parallel compositions] 
    	\[
    	    \begin{tikzcd}[column sep = small]
    	        R^p(-,\oplax^q,-,R^r)
    	        \rar
    	        \dar 
    	        & 
    	        \dar  R^{p+q}(-,R^r)\\
    	        R^{p+r}(-,\oplax^q,-)
    	        \rar 
				& R^{p+q+r}
    	    \end{tikzcd}
    	    \quad
    	    \begin{tikzcd}[column sep = small]
    	        R^p(-,\oplax^q,-,\oplax^r,-)
    	        \rar
    	        \dar
    	        &
    	        \dar R^{p+q}(-,\oplax^r,-)\\
    	        R^{p+r}(-,\oplax^q,-)
    	        \rar 
    	        & R^{p+q+r}
    	    \end{tikzcd}
    	\]
    	commute whenever they make sense;
	\item[Sequential compositions]
		\[
			\hspace{-1.1cm}
			\begin{tikzcd}[column sep=small]
				R^p(-,\oplax^q(-,\oplax^r,-),-)
				\rar
				\dar
				& R^{p+q}(-,\oplax^r,-)
				\dar \\
				R^p(-,\oplax^{q+r},-)
				\rar
				& R^{p+q+r}
			\end{tikzcd}
			\quad
			\begin{tikzcd}[column sep=small]
				R^p(-,R^q(-,\oplax^r,-))
				\rar
				\dar
				& 
				\dar R^{p+q}(-,\oplax^r,-)\\
				R^p(-,R^{q+r})
				\rar
				&
				R^{p+q+r}
			\end{tikzcd}
		\]
		and
	\[
		\begin{tikzcd}
			R^p \circ (-,R^q(-,R^r))
			\dar
			\rar["\IsIsomorphicTo", dash]
			&[-3ex]
			R^p(-,R^q) \circ (-,R^r)
			\rar
			&[3ex] R^{p+q}(-,R^r)
			\ar[d] \\
			R^p(-,R^{q+r})
			\ar[rr]
			&& R^{p+q+r}
		\end{tikzcd}
	\]
    commute whenever they make sense.
    \end{description}

    If the structure maps \( R^n \) are left representable
    distributors, then \( \Category M^R \) will
    be called a left representable lax \( \Vdelta \)module.
\end{definition}

\begin{remark}
	Via \( \Cat^{2\text{-}\mathrm{op}} \to \Cat_\deltaup \), every
	oplax monoid is sent to a (representable) lax monoid.
	The definition of \( \Vdelta \)module then corresponds
	to the notion of module over the lax monoid
	induced by \( \Category V^\oplax \) in \( \Cat_\deltaup \).
\end{remark}

\begin{lemma}\label{lemma:left-rep}
    Let \( \Category M \) be a category and
    \( R^n \From \Category V^n \times \Category M
    \nrightarrow \Category M \) be left representable
    distributors for \( n \geq 0 \). For each representations
    \( \rho^n \From \Category V^n \times \Category M
    \to \Category M \), for \( n \geq 0 \), of the previous
    collection of distributors, there is a bijection
    between lax \( \Vdelta \)module structures extending
	\( \Category M^R \) and \( \Category V \)\=/module
    structures extending \( \Category M^\rho \).
\end{lemma}
\begin{proof}
	Straightforward thus omitted.
\end{proof}

Considering \( \Vdelta \)modules as modules over a lax monoid
in \( \Cat_\deltaup \), one can naturally define
a 2\=/category of \( \Vdelta \)modules.
We shall however restrict
ourselves to a sub-bicategory of
\( \Vdelta\mathsf{mod}_\mathrm{oplax} \),
that we shall denote
\( \Vdelta\mathsf{mod}_\mathrm{oplax}^\mathrm{fun} \),
wherein the oplax morphisms of \( \Vdelta \)modules are taken to be
functors (instead of distributors)
equipped with some natural transformations defined
precisely below.

\begin{definition}[(Oplax functor between \( \Vdelta \)modules)]
\label{def:morphism_V-delta_mod}
	An oplax functor between two \( \Vdelta \)modules
	\( \Category M^R \) and \( \Category N^S \)
	is the data of:
	\begin{itemize}
		\item
			a functor
			\( F \From \Category M \longrightarrow \Category N \);
		\item
			natural transformations
			\[
				\begin{tikzcd}
					R^p \inout{\Collection a,x}{y}
					\ar[r, "f^p"]
					& S^p \inout{\Collection a, F(x)}{F(y)}
				\end{tikzcd}
			\]
			with components
			for any collection \( \Collection a \) of objects
			in \( \Category V \), pair of objects \( x, y \)
			in \( \Category M \) and integers \( p \geq 0 \),
	\end{itemize}
	such that the diagrams
	\[
		\begin{tikzcd}
			\Category M \inout{x}{y}
			\rar
			\ar[d, "F"]
			& R^0 \inout{x}{y}
			\ar[d, "f^0"] \\
			\Category N \inout{F(x)}{F(y)}
			\rar
			& S^0 \inout{F(x)}{F(y)}
		\end{tikzcd}
		\quad
		\begin{tikzcd}[sep=large]
			R^p \inout{-, \oplax^q, -, x}{y}
			\rar
			\ar[d, "f^p"]
			& R^{p+q} \inout{-, x}{y}
			\ar[d, "f^{p+q}"] \\
			S^p \inout{-, \oplax^q, -,F(x)}{F(y)}
			\rar
			& S^{p+q} \inout{-, F(x)}{F(y)}
		\end{tikzcd}
	\]
	and
	\[
		\begin{tikzcd}[sep=large]
			R^p(-, R^q) \inout {-,x} y
			\rar
			\dar["f^p \times f^q"]
			& R^{p+q} \inout{-, x}{y}
			\ar[dd, "f^{p+q}"] \\
			\Integral{m} S^p \inout{-, F(m)}{F(y)}
			\times S^q \inout{-, F(x)}{F(m)}
			\dar
			&
			\\
			S^p(-, S^q) \inout {-, F(x)} {F(y)}
			\rar
			& S^{p+q} \inout{-, F(x)}{F(y)}
		\end{tikzcd}
	\]
	commute whenever they make sense.
\end{definition}

\begin{definition}[(Natural transformation between oplax functors)]
	A natural transformation
	\( \alpha \From F \Rightarrow G \)
	between a couple of
	oplax functors
	\( F,G \From \Category M \to \Category N \),
	with structure \( f \) and \( g \) respectively,
	is a natural transformation between the underlying 
	functors such that the diagram
	\[
		\begin{tikzcd}[sep=large]
			R^p \inout{\Collection a,x}{y}
			\ar[r, "f^p"]
			\ar[d, "g^p"]
			& S^p \inout{\Collection a,F(x)}{F(y)}
			\ar[d, "{S^p\inout{-}{\alpha_y}}"] \\
			S^p \inout{\Collection a,G(x)}{G(y)}
			\ar[r, "S^p\inout{-,\alpha_x}{-}"]
			& S^p \inout{\Collection a,F(x)}{G(y)}
		\end{tikzcd}
	\]
	commutes for any \( p \geq 0 \).
\end{definition}

There is an evident 2-functor
\[
	\VMods
	\longrightarrow
	\Vdelta\mathsf{mod}^\mathrm{fun}_\mathrm{oplax}
\]
sending every \( \Category V \)-module \( \Category M^\rho \)
to the lax \( \Vdelta \)module with the same category \( \Category M \)
and structure maps \( \Category M \inout{\rho}{-} \).
Lax morphisms between \( \Category V \)-modules induce
oplax functors between the associated \( \Vdelta \)modules in the
obvious way and the same goes for the natural transformations.

\begin{proposition}
	The 2-functor above induces
	a 2-equivalence
	\[
		\VMods
		=
		\mathsf{left~rep.~}
		\Vdelta\mathsf{mod}^\mathrm{fun}_\mathrm{oplax}
	\]
	between the 2-category of \( \Category V \)-modules, lax
	morphisms and natural transformations, and the 2-category
	of left representable \( \Vdelta \)modules, oplax functors
	and natural transformations.
\end{proposition}

\begin{proof}
	The previous lemma relating left representable
	\( \Vdelta \)modules to \( \Category V \)\=/modules
	\UnskipRef{lemma:left-rep} implies
	essential surjectivity on objects. 
	
	Given a functor
	\( F \From \Category M^\rho \to \Category N^\sigma \)
	between two \( \Category V \)-modules, one has a bijection
	between
	natural transformations
	\( \sigma^p(-,F) \to F\,\rho^p \)
	and natural transformations
	\( \Category M \inout {\rho^p} -
	\to \Category N \inout {\sigma^p(-, F)} {F(-)} \).
	In one direction, this is given by the composition
	\[
		\begin{tikzcd}
			\Category M\inout{\rho^p}{-}
			\rar
			&
			\Category N \inout {F \rho^p} {F(-)}
			\rar
			& \Category N\inout{\sigma^p(-,F)}{F(-)}
		\end{tikzcd}
	\]
	and by the embedding lemma of Yoneda in the other direction.
	The three diagrams defining lax morphisms of
	\( \Category V \)-modules
	on one hand, and oplax functors between \( \Vdelta \)modules
	on the other hand, are in one-to-one correspondence under this bijection.
	Bijectivity on natural transformations is obtained in a similar way.
\end{proof}

\subsection{Categories enriched over \( \Category V \)}
\label{sec:Cat_V}

In this subsection, we introduce the notion of a category
enriched over an oplax monoidal category \( \Category V^\oplax \),
as well as the associated notions of enriched functors
and natural transformations.
An oplax
monoidal category is endowed with structure maps of arbitrarily
high arity which naturally leads one to consider composition
maps of arbitrary arity as well when defining an enrichment
structure.

There is also an abridged version of the definition of an enriched
category which is closer to the familiar definition of an enriched
category over a tensor category.
As we shall see, in the case of enrichment over an oplax monoidal
category, the two definitions are equivalent.

\begin{notation}
	Given three categories \( \Category C \),
	\( \Category V \) and \( \Category W \)
	and a bifunctor
	\( [-, -] \From \Category C\Op \times \Category C \to \Category V \),
	for every \( n \geq 0 \)
	and every functor \( F \From \Category V^n \to \Category W \)
	we shall use the notation
	\[
		{[-,\dots,-]}_F \coloneqq
		F(\EnrichedHom{-}{-}, \dots, \EnrichedHom{-}{-})
	\]
	for the functor
	\( (\Category C\Op \times \Category C)^n \to \Category W \)
	obtained by composing \( [-,-]^n \) with
	\( F \).
\end{notation}

\begin{definition}[(Categories enriched over \( \Category V \), extended version)]
\label{def:cat_V-extended}
    A category enriched over \( \Category V \) is
    the data of a category \( \Category C \)
    endowed with
    \begin{itemize}
    \item a bifunctor
    \[
		\begin{tikzcd}[sep=large]
			\EnrichedHom{-}{-} \From
			\Category C\Op \times \Category C
			\rar
			& \Category V
		\end{tikzcd}
	\]
	taking values in \( \Category V \);
	\item for every \( n \geq -1 \) and each
		\( (n+2) \)-tuple of objects \( x_0, \dots, x_{n+1} \)
		of \( \Category C \), transformations
		\[
			\begin{tikzcd}[column sep = large]
				\oplax^n(\EnrichedHom{x_n}{x_{n+1}},
				\dots,
				\EnrichedHom{x_0}{x_1})
				\rar["m^n"]
				& \EnrichedHom{x_{0}}{x_{n+1}}
			\end{tikzcd}
		\]
		which are natural in \( x_0 \) and \( x_{n+1} \),
		extranatural in \( x_1, \dots, x_n \), and
		which we shall abbreviate as
		\[
			[x_0, \dots, x_{n+1}]_{\oplax^n}
			\longrightarrow
			\EnrichedHom{x_{0}}{x_{n+1}}
		\]
		and collectively abbreviate as
		\[
			[-, \dots, -]_{\omega^n}
			\longrightarrow
			\EnrichedHom{-}{-}
		\]
	\end{itemize}

	It is required that:
	\begin{description}
		\item[Unitality]
			the morphism \( m^0 \) be the counit
			natural transformation of \( \Category V \);
		\item[Additivity]
			the diagram
			\[
				\begin{tikzcd}[sep = large]
					[-, \dots, -]_{\oplax^{p+q}}
					\ar[rd]
					\ar[ddd, "m^{p+q}" swap]
					\\
					[-3ex]
					&
					{[-, \dots, -]_{\oplax^p(-,\oplax^q,-)}}
					\dar["\oplax^p(-{,} m^q{,} -)"]
					\\
					[6ex]
					&
					{[-, \dots, -]_{\oplax^p}}
					\ar[dl, "m^p"]
					\\
					\EnrichedHom{-}{-}
				\end{tikzcd}
			\]
			commute whenever it makes sense.
	\end{description}
\end{definition}

Now let us introduce an abridged version of
the above definition, wherein one only requires
the existence of a binary composition and a unit
map.

\begin{definition}[(Category enriched over \( \Category V \), abridged version)]
	\label{def:oplax_enriched_abridged}
	A category \( \Category C \) enriched over \( \Category V \)
	is a category \( \Category C \) equipped with
	\begin{itemize}
    \item a bifunctor
    \[
		\begin{tikzcd}[sep=large]
			\EnrichedHom{-}{-} \From
			\Category C\Op \times \Category C
			\rar
			& \Category V
		\end{tikzcd}
	\]
	taking values in \( \Category V \);
	\item an extranatural transformation with components
	\[
		\begin{tikzcd}
			1_\oplax
			\rar["\eta_x"]
				& \EnrichedHom{x}{x}
		\end{tikzcd}
	\]
	for each \( x \in \Objects(\Category C) \),
	called the unit map;
	\item a transformation called the composition map, with components
	\[
		\begin{tikzcd}[column sep = large]
			\EnrichedHom{y}{z} \otimes_\oplax
			\EnrichedHom{x}{y}
			\rar["\mu_{x, y, z}"]
			& \EnrichedHom{x}{z}
		\end{tikzcd}
	\]
	for each triplet of objects \( x,y,z \in \Objects(\Category C) \)
	being natural in \( x \) and \( z \) while being
	extranatural in \( y \);
	\end{itemize}
	such that:
	\begin{description}
	\item[Associativity]
	the associativity diagram in \( \Category V \)
	\[
		\begin{tikzcd}[row sep = large]
			&[-15ex]
			\EnrichedHom y z \otimes_\oplax
			\EnrichedHom x y \otimes_\oplax
			\EnrichedHom w x
			\arrow[dl]
			\arrow[dr] &[-15ex]
			\\
			\big(\EnrichedHom y z \otimes_\oplax \EnrichedHom x y\big)
			\otimes_\oplax \EnrichedHom{w}{x}
			\arrow[d, "\mu_{x,y,z} \otimes_\oplax -"{left}]
			&&
			\EnrichedHom{y}{z} \otimes_\oplax
			\big(\EnrichedHom x y \otimes_\oplax \EnrichedHom w x\big)
			\arrow[d, "- \otimes_\oplax \mu_{w,x,y}"]
			\\
			\EnrichedHom x z \otimes_\oplax \EnrichedHom w x
			\arrow[dr, "\mu_{w,x,z}"{below left}]
			&& \EnrichedHom y z \otimes_\oplax \EnrichedHom w y
			\arrow[dl, "\mu_{w,y,z}"]
			\\
			& \EnrichedHom{w}{z} &
		\end{tikzcd}
	\]
	commutes for every \( w, x, y, z \in \Objects(\Category C) \);
	\item[Unitality]
		the two unitality diagrams in \( \Category V \)
		\[
			\begin{tikzcd}[column sep = small]
				& \OplaxHom{x,y}
				\ar[ddd]
				\ar[dl]
				\\[-3ex]
				1_\oplax \otimes_\oplax
				\EnrichedHom{x}{y}
				\dar["\eta_y \otimes_\oplax -" swap]
				\\[6ex]
				\EnrichedHom y y \otimes_\oplax \EnrichedHom x y
				\ar[dr, "\mu_{x, y, y}" swap]
				\\
				& \EnrichedHom{x}{y}
			\end{tikzcd}
			\qand
			\begin{tikzcd}[column sep = small]
				\OplaxHom{x,y}
				\ar[ddd]
				\ar[dr]
				\\[-3ex]
				& \EnrichedHom{x}{y}
				\otimes_\oplax 1_\oplax
				\dar["- \otimes_\oplax \eta_x"]
				\\[6ex]
				& \EnrichedHom x y \otimes_\oplax \EnrichedHom x x
				\ar[dl,"\mu_{x,x,y}"]
				\\
				\EnrichedHom{x}{y}
			\end{tikzcd}
		\]
		commute for every \( x, y \in \Objects(\Category C) \).
	\end{description}
\end{definition}

This definition is close to that of a category
enriched over a strong monoidal category, the main
differences being that we still require the composition
and unit maps to be (extra)natural, and that the associativity
and unitality condition be slightly modified due to the oplax
nature of \( \Category V^\oplax \).

\begin{proposition}
	\label{prop:truncation}
	The two above definitions of a \( \Category V \)-category
	are canonically isomorphic.
\end{proposition}
\begin{proof}
	The proof of the equivalence between the extended and the
	abridged version can be done by induction and is very
	similar to the one of the reconstruction of
	the lax structure of a functor
	\UnskipRef{prop:truncation_lax_functors},
	apart from the fact
	that the diagrams involved are a bit smaller.
\end{proof}

\begin{definition}[(\( \Category V \)\=/enriched functor, extended version)]
	\label{def: V-functor, extended}
	Let \( (\Category B, [-,-], \{l^n\}_{n \geq -1}) \)
	and \( (\Category C, \langle -, -\rangle, \{m^n\}_{n \geq -1}) \)
	be two categories enriched over \( \Category V \).
	A \( \Category V \)\=/enriched functor between 
	\( \Category B \) and \( \Category C \) is
	the data of a functor
	\( F \From \Category B \to \Category C \) equipped
	with natural transformations with components
	\[
		\begin{tikzcd}[column sep = huge]
			\EnrichedHom{x}{y}
			\rar["F_{x,y}"]
				&
				\langle F(x), F(y) \rangle
		\end{tikzcd}
	\]
	for each pair of objects \( x, y \) of \( \Category B \),
	such that the diagrams
	\[
		\begin{tikzcd}[column sep = 120]
			\OplaxHom{x_0, \dots, x_{n+1}}
			\rar["{\oplax^n\left(F_{x_n, x_{n+1}}, \dots,
			F_{x_0, x_1}\right)}"]
			\dar["l^n" swap]
				& \big\langle
					F(x_0), \dots, F(x_{n+1})
				\big\rangle_\oplax
				\dar["m^n"]
			\\
			\EnrichedHom{x_0}{x_{n+1}}
			\rar["F_{x_0, x_{n+1}}"]
				& \big\langle F(x_0), F(x_{n+1})
				\big\rangle
		\end{tikzcd}
	\]
	commute for all integers \( n \geq -1 \)
	and all sequences of objects \( x_0, \dots, x_{n+1} \)
	of \( \Category B \).
\end{definition}

Similarly as for categories enriched over \( \Category V \),
the notion of \( \Category V \)\=/enriched
functor admits an abridged version,
involving only two compatibility diagrams. 

\begin{definition}[(\( \Category V \)\=/enriched functor, abridged version)]
	\label{definition: V-enriched functor abridged}
	A \( \Category V \)\=/enriched functor
	between two categories  \( (\Category B, [-,-], \eta, \mu) \)
	and \( (\Category C, \langle -, -\rangle, \theta, \nu) \)
	enriched over \( \Category V \) is a functor
	\( F \From \Category B \to \Category C \)
	equipped with natural transformations with components
	\[
		\begin{tikzcd}[column sep = huge]
			\EnrichedHom{x}{y}
			\rar["F_{x,y}"]
				&
				\langle F(x), F(y) \rangle
		\end{tikzcd}
	\]
	for each pair of objects \( x, y \) of \( \Category B \),
	such that:
	\begin{description}
		\item[unit]
			the unit diagram in \( \Category V \)
			\[
				\begin{tikzcd}[column sep=huge, row sep=large]
						1_\oplax
						\ar[d, "{\eta_x}",swap]
					\ar[rd,"{\theta_{F(x)}}"] &
					\\
					\EnrichedHom{x}{x}
					\ar[r, "F_{x,x}"]
						& \langle F(x), F(x) \rangle
				\end{tikzcd}
			\]
			commutes for every \( x \in \Objects(\Category B) \);
		\item[composition]
			the composition diagram in \( \Category V \)
			\[
				\begin{tikzcd}[column sep = 70, row sep = large]
					\EnrichedHom{y}{z}
					\otimes_\oplax
					\EnrichedHom{x}{y}
					\dar["{\mu_{x,y,z}}" swap]
					\rar["F_{y,z} \otimes_\oplax F_{x, y}"]
						& \langle F(y), F(z) \rangle
						\otimes_\oplax
						\langle F(x), F(y) \rangle
						\dar["{\nu_{F(x), F(y), F(z)}}"]
					\\
					\EnrichedHom{x}{z}
					\rar["F_{x,z}"]
						& \langle F(x), F(z) \rangle
				\end{tikzcd}
			\]
			commutes for every \( x, y, z \in \Objects(\Category B) \).
	\end{description}
\end{definition}

\begin{proposition}
	The two above definitions of \( \Category V \)\=/functors
	are canonically equivalent.
\end{proposition}
\begin{proof}
	Let us denote \( \Induction(n) \) the hypothesis that the
	\( n \)-th diagram commutes with \( n \geq -1 \).
	The hypotheses \( \Induction(-1) \) and \( \Induction(1) \) are satisfied 
	due to the unit and composition diagrams, respectively, while 
	\( \Induction(0) \) holds by naturality of the counit of \(\oplax\).
	From the additivity of the structure maps \( m^n \), a simple
	diagram shows that, by writing \( m^{n+1} \) as \( m^1 \)
	followed by \( m^n \otimes_\oplax - \),
	and symmetrically for \( l^{n+1} \), one has
	\( \Induction(n) \wedge \Induction(1) \implies \Induction(n+1) \).
\end{proof}

\begin{definition}[(Natural \( \Category V \)\=/enriched transformations,
	extended version)]
	\label{def:enriched_nat_transfo}
	Let \( (\Category B, [-,-], \{l^m\}_{n \geq -1}) \)
	and 
	\( (\Category C, \langle -, - \rangle, \{m^n\}_{n \geq -1}) \)
	be two categories enriched over \( \Category V \).
	Let furthermore \( F,G \From \Category B \to\Category C \)
	be two \( \Category V \)\=/enriched functors.
	A natural \( \Category V \)\=/enriched transformation
	between \( F \) and \( G \) is a natural transformation
	\( \alpha \From F \Rightarrow G \) such that the diagram
	\[
		\begin{tikzcd}[row sep = large]
			& \OplaxHom{x_0, \dots, x_{n+1}}
			\ar[dl,"{\oplax^n\left(F_{x_n,x_{n+1}},
				\dots, F_{x_0,x_1}\right)}" swap]
			\ar[dr,"{\oplax^n\left(G_{x_n,x_{n+1}},
				\dots, G_{x_0,x_1}\right)}"]
			&
			\\
				\big\langle F(x_0), \dots, F(x_{n+1}) \big\rangle_\oplax
				\dar["{\left\langle -,
					\dots, \alpha_{x_{n+1}} \right\rangle_\oplax}" swap]
					& & \big\langle G(x_0), \dots, G(x_{n+1}) \big\rangle_\oplax
					\dar["{\left\langle \alpha_{x_0},
						\dots, - \right\rangle_\oplax}"]
			\\
				\big\langle F(x_0), \dots, G(x_{n+1}) \big\rangle_\oplax
				\ar[dr, "{m}^n" swap]
					& & \big\langle F(x_0), \dots, G(x_{n+1}) \big\rangle_\oplax
					\ar[dl, "{m}^n"]
				\\
				  & \big\langle F(x_0), G(x_{n+1}) \big\rangle_\oplax &
		\end{tikzcd}
	\]
	commutes in \( \Category V \), for every \( n \geq 0 \)
	and every \( x_0, \dots, x_{n+1} \in \Objects(\Category B) \).
\end{definition}

Once again, we can consider an abridged version
of the previous definition allowing to reexpress 
the above compatibility diagram in the usual square form.

\begin{definition}[(Natural \( \Category V \)\=/enriched transformations,
	abridged version)]
	Let \( F, G \From \Category B \to \Category C \)
	be two \( \Category V \)\=/enriched functors between
	two categories \( (\Category B, [-,-],\eta,\mu) \)
	and \( (\Category C, \langle-,-\rangle, \theta, \nu) \)
	enriched over \( \Category V \).
	A natural \( \Category V \)\=/enriched transformation
	between \( F \) and \( G \) is a natural transformation
	\( \alpha \From F \Rightarrow G \) such that the diagram
	\[
		\begin{tikzcd}[sep=large]
			\EnrichedHom{x}{y}
			\arrow[r, "F_{x,y}"]
			\arrow[d, "G_{x,y}" swap]
			& \langle F(x), F(y) \rangle
			\arrow[d, "{\langle -, \alpha_y \rangle}"] \\
			\langle G(x), G(y) \rangle
			\arrow[r, "{\langle \alpha_x, - \rangle}"]
			& \langle F(x), G(y) \rangle
		\end{tikzcd}
	\]
	commutes for every pair \( x, y \in\Objects(\Category B) \).
\end{definition}

\begin{proposition}
	The two above definitions of natural
	\( \Category V \)\=/enriched transformations are canonically equivalent.
\end{proposition}
\begin{proof}
	The fact that the extended version implies the abridged version follows
	straightforwardly from the specialisation of the
	defining diagram to \( n=0 \). In order to prove the converse implication,
	we use the naturality of \( m^n \) and \( \alpha \) to rewrite
	the defining diagram \UnskipRef{def:enriched_nat_transfo} as
	\[
		\begin{tikzcd}[row sep=large]
			& \OplaxHom{x_0, \dots, x_{n+1}}
			\arrow[dd, "{l}^n"]
			\arrow[dl, swap,"{\oplax^n\left(F_{x_n,x_{n+1}}, \dots, F_{x_0,x_1}\right)}"]
			\arrow[dr, "{\oplax^n\left(G_{x_n,x_{n+1}}, \dots, G_{x_0,x_1}\right)}"]&
			\\
			\langle F(x_0), \dots, F(x_{n+1}) \rangle_\oplax
			\arrow[dd, "{m}^n" swap]
			 &&
			\langle G(x_0), \dots, G(x_{n+1}) \rangle_\oplax
			\arrow[dd, "{m}^n"] \\
			& \OplaxHom{x_0, x_{n+1}}
			\ar[dr, "\oplax^0\left(G_{x_0,x_{n+1}}\right)" swap]
			\ar[dl,swap, "\oplax^0\left(F_{x_0,x_{n+1}}\right)" swap]
			&
			\\
			\langle F(x_0), F(x_{n+1}) \rangle_\oplax
			\arrow[dr,swap, "{\left\langle -, \alpha_{x_{n+1}} \right\rangle_\oplax}"]
			&& \langle G(x_0), G(x_{n+1}) \rangle_\oplax
			\arrow[dl, "{\left\langle \alpha_{x_0}, -\right\rangle_\oplax}"]
			\\
			& \langle F(x_0), G(x_{n+1}) \rangle_\oplax
			&
		\end{tikzcd}
	\]
	so that the top squares commute by virtue of
	the property of \( \Category V \)\=/enriched
	functors while the bottom square is the abridged
	\( \Category V \)\=/enriched naturality condition.
\end{proof}

\subsection{Categories enriched over \( \Category V \)
vs \( \Vdelta \)modules}
\label{sec:right-rep}

Having previously compared \( \Category V \)\=/modules
with \( \Vdelta \)modules, we shall now compare categories
enriched over \( \Category V \) with \( \Vdelta \)modules,
and shall show that the former identify with a subclass of
the latter (obeying a certain representability condition).
First, let us start by showing that any category
enriched over \( \Category V \) induces a \( \Vdelta \)module.

\begin{proposition}\label{prop:Cat_V-to-Vdelta}
	Any category \( \Category M \) enriched 
	over \( \Category V \) defines a \( \Vdelta \)module
	with structure maps
	\[
		R^{n+1} \inout {\Collection a, x}{y}\, \coloneqq\,
		\Category V \inout{\oplax^n(\Collection a)}{\EnrichedHom{x}{y}}
	\]
	for any sequence \( \Collection a \) of \( n+1 \) objects
	of \( \Category V \), pair \( x, y \) of objects of \( \Category M \)
	and integer \( n \geq -1 \).
\end{proposition}
\begin{proof}
	First, we need to construct the associators, 
	\ie the natural transformations
	\[
		R^p (-,\oplax^q,-) \to R^{p+q},
		\quad
		R^p \circ (-,R^q) \to R^{p+q}
		\qand
		\IdentityFunctor_{\Category M} \to R^0
	\]
	The first type of associators is
	induced from those of the oplax structure \( \oplax \).
	The second type of associators is obtained from
	the composite
	\[
		\Category V \inout {\oplax^{p-1}(\Collection a)}{\EnrichedHom{m}{y}}
		\times \Category V \inout {\oplax^{q-1}(\Collection b)}{\EnrichedHom{x}{m}}
		\to
		\Category V \inout {\oplax^{p-1}(\Collection a)\,  \otimes_\oplax \oplax^{q-1}(\Collection b)}{\EnrichedHom{m}{y}\,  \otimes_\oplax \EnrichedHom{x}{m}}
		\to
		\Category V \inout {\oplax^{p+q-1}(\Collection a \oplus \Collection b)}{\EnrichedHom{m}{y}\,  \otimes_\oplax \EnrichedHom{x}{m}}
		\to
		\Category V \inout {\oplax^{p+q-1}(\Collection a \oplus \Collection b)}{\EnrichedHom{x}{y}}
	\]
	where \( \Collection a \) and \( \Collection b \)
	are sequences of respectively \( p \) and \( q \)
	objects of \( \Category V \) (with \(p,q \geq 0 \)),
	and \( m, x, y \) are objects of \( \Category M \).
	The first map is simply given by the tensor product
	of morphisms in \( \Category V \), the second arrow
	is induced by either of the compositions
	\[
		\begin{tikzcd}[column sep = 70]
			\oplax^{p+q-1}
			\arrow[r]
			\arrow[d]
				&\oplax^{p}(-, \oplax^{q-1})
				\arrow[d] \\
			\oplax^{q}(\oplax^{p-1}, -)
			\arrow[r]
				& \oplax^{p-1}\otimes_\oplax\oplax^{q-1}
		\end{tikzcd}
	\]
	which are equivalent by parallel decomposition, 
	and the last arrow from the composition map
	of \( \Category M \).
	The total composite map
	is extranatural in \( m \), due to the fact that
	\( \EnrichedHom{-}{-} \) is a bifunctor and that \( \mu \)
	is extranatural in \( m \). Hence
	the universal property of integrals ensures
	the existence of natural transformations
	\[
		\Integral{m}
		\Category V \inout {\oplax^{p-1}(\ba)}{\EnrichedHom{m}{y}}
		\times \Category V \inout {\oplax^{q-1}(\bb)}{\EnrichedHom{x}{m}}
		\to \Category V \inout {\oplax^{p+q-1}(\ba \oplus \bb)}{\EnrichedHom{x}{y}}
	\]
	for all \( p, q \geq 0 \) which play the role of associators 
	\( R^p \circ (-,R^q) \to R^{p+q} \).

	The last associator,
	\ie the unit of the \( \Vdelta \)module structure,
	\[
		\Category M \inout{x}{y}
		\to \Category V \inout{1_\oplax}{\EnrichedHom{x}{y}}\,,
	\]
	is simply given by
	\[
		\begin{tikzcd}
			1_\oplax
			\ar[d, "\eta_y"]
			\ar[r, "\eta_x"]
			& \EnrichedHom{x}{x}
			\ar[d, "\EnrichedHom{-}{f}"] \\
			\EnrichedHom{y}{y}
			\ar[r, "\EnrichedHom{f}{-}"]
			& \EnrichedHom{x}{y}
		\end{tikzcd}
	\]
	for any morphism \( f \From x \to y \), 
	and which commutes by extranaturality 
	of the unit \( \eta \) of \( \Category M \).

	Now we also need to verify that these natural 
	transformations do define a \( \Vdelta \)module 
	structure. One can check that the unitality condition
	containing \( \omega^0 \) is verified as
	a consequence of the counitality condition
	of \( \oplax \), while the other two unitality
	conditions (both involving \( R^0 \)) are
	satisfied by virtue of the unitality condition
	for the enrichment of \( \Category M \) over
	\( \Category V \) as well as the counitality
	condition of \( \oplax \). 
	Moreover, the parallel
	composition condition involving two copies of
	\( \oplax \) is satisfied
	as a consequence 
	of the parallel decomposition
	of \( \oplax \) while the one involving one copy of
	\( \oplax \) holds by virtue
	of both the parallel and sequential decompositions
	of \( \oplax \).
	Finally, the two sequential composition
	conditions involving at least one
	\( \omega \) are verified as a consequence of
	the sequential composition conditions of \( \oplax \),
	while the third one (only involving the structure
	maps \( R^n \)) is verified due to the associativity
	of the composition maps \( \mu \), as well as 
	the parallel and sequential decomposition conditions
	of \( \oplax \).
\end{proof}

Notice that the structure of \( \Vdelta \)module
induced by a category enriched over \( \Category V \)
is such that the diagrams
\[
	\begin{tikzcd}[sep=large]
		\Category V \inout{\oplax^n(\Collection a)}{\EnrichedHom{x}{y}}
		\rar
		\ar[dr, equal]
		& \Category V \inout{\oplax^0\oplax^n(\Collection a)}{\EnrichedHom{x}{y}}
		= R^1 \inout {\oplax^n(\Collection a) , x}{y}
		\dar \\
		& R^{n+1} \inout {\Collection a , x}{y} &
	\end{tikzcd}
\]
commute for any sequence of \( n+1 \) objects \( \Collection a \)
of \( \Category V \), any pair of objects \( x, y \) of \( \Category M \)
and any integer \( n \geq 0 \). This is simply a consequence
of the counitality condition that the oplax monoidal
structure satisfies.
Abstracting away the previous property leads to consider the following definition of a right
representable \( \Vdelta \)module.

\begin{definition}[(Right representable
\( \Vdelta \)modules)]
	A \( \Vdelta \)module \( \Category M^R \) will be called
	right representable if there exists
	  \begin{itemize}
  		  \item a bifunctor
        \[
    		\begin{tikzcd}[sep=large]
    			\EnrichedHom{-}{-} \From
    			\Category M\Op \times \Category M
    			\rar
    			& \Category V
    		\end{tikzcd}
    	\]
	taking values in \( \Category V \);	
	\item  a natural	isomorphism
		\[
		\phi \From \Category V \inout{\oplax^0(-)}{\EnrichedHom{-}{-}}
		\cong
		R^1 \inout {-,-}{-}\,,
		\]
	\end{itemize}
	such that the maps
	\[
		\begin{tikzcd}[sep=large]
			\Category V \inout{\oplax^{p-1}(-)}{\EnrichedHom{-}{-}}
			\ar[r, "\phi^p"]
			& R^p \inout {- , -}{-}
		\end{tikzcd}
	\]
	defined by the diagrams
	\[
		\begin{tikzcd}[column sep=large]
			\Category V \inout{\oplax^{p-1}(\Collection a)}{\EnrichedHom{x}{y}}
			\ar[dr] \ar[ddd, "\phi^p"] & \\
			& \Category V \inout{\oplax^0\oplax^{p-1}(\Collection a)}{\EnrichedHom{x}{y}}
			\ar[d, "\phi"] \\
			& R^1 \inout {\oplax^{p-1}(\Collection a) , x}{y}
			\ar[dl] \\
			R^p \inout {\Collection a , x}{y} &
		\end{tikzcd}
	\]
	are natural isomorphisms, for \( p \geq 0 \).
\end{definition}

\begin{remark}
	By counitality of \( \oplax^0 \),
	one recovers \( \phi \) as
	\( \phi^1 \).
\end{remark}

\begin{remark}
	\label{remark: Vdelta truncation}
	In a right representable \( \Vdelta \)module, the composition map
	\( R^1(\oplax^p, -) \to R^{p+1} \) admits a canonical section
	\[
		R^{p+1} \longrightarrow R^1(\oplax^p, -)
	\]
	which, using sequential composition, allows us
	to rewrite the associators \( R^p(-, R^q) \to R^{p+q} \)
	as the composite
	\[
		R^p(-, R^q)
		\longrightarrow
		R^1(\oplax^{p-1}, R^1(\oplax^{q-1}, -))
		\longrightarrow
		R^2(\oplax^{p-1}, \oplax^{q-1}, -)
		\longrightarrow
		R^{p+q}
	\]
	for \( p, q \geq 0 \).
\end{remark}

\begin{lemma}
	\label{lemma: rep Vdelta truncation}
	In a right representable \( \Vdelta \)module,
	the square
	\[
		\begin{tikzcd}[ampersand replacement=\&]
			\Category V \inout {\oplax^{p-1}(-, \oplax^q, -)} {[-,-]}
			\arrow[r, "\phi^p"]
			\arrow[d, "", swap]
			\& 
			R^p \inout{-, \oplax^q, -}{-}
			\arrow[d, ""] \\
			\Category V \inout {\oplax^{p + q -1}} {[-,-]}
			\arrow[r, "\phi^{p+q}"]
			\& 
			R^{p+q} \inout - -
		\end{tikzcd}
	\]
	commutes whenever it makes sense.
\end{lemma}

\begin{proof}
	For this it is enough to consider the diagram
	\[
		\begin{tikzcd}
			\Category V \inout {\oplax^{p-1}(-, \oplax^q, -)} {[-,-]}
			\arrow[r]
			\arrow[d]
			& 
			\Category V
			\inout {\oplax^0\oplax^{p-1}(-, \oplax^q, -)} {[-,-]}
			\dar \rar
			&
			R^1 \inout {\oplax^{p-1}(-, \oplax^q,-)} -
			\dar \rar
			&
			R^p \inout{-, \oplax^q, -}{-}
			\arrow[d, ""] \\
			\Category V \inout {\oplax^{p + q -1}} {[-,-]}
			\arrow[r]
			& 
			\Category V \inout {\oplax^0\oplax^{p + q -1}} {[-,-]}
			\rar
			& 
			R^1 \inout {\oplax^{p+q-1}} -
			\rar
			&
			R^{p+q} \inout - -
		\end{tikzcd}
	\]
	whose first two squares commute by naturality and whose
	last square commutes by sequential composition.
\end{proof}

The notion of right representable \( \Vdelta \)module
identifies with the one of category enriched
over \( \Category V \), as shown in the following
proposition. 

\begin{proposition}
	\label{prop:right_rep=Cat_V}
	Let \( \Category M \) be a category and
	\( R^n \From \Category V^n \times \Category M
	\nrightarrow \Category M \) be distributors
	for \( n \geq 0 \). Given a bifunctor
	\( [-,-] \From \Category M\Op \times \Category M
	\to \Category V \), and natural isomorphisms
	\( \psi^p \From \Category V\inout{\oplax^{p-1}}{[-,-]} \cong R^p \)
	for \( p \geq 0 \), there is a bijection
	between right representable \( \Vdelta \)module
	structures extending \( (\Category M, R) \) and 
	structures of category enriched over \( \Category V \)
	extending \( (\Category M, [-,-]) \).
\end{proposition}
\begin{proof}
	Starting from an extension of structure \( (\Category M, [-,-]) \)
	into a category enriched over \( \Category V \),
	we have already explained how to build a canonical
	\( \Vdelta \)module structure on \( \Category M \)
	\UnskipRef{prop:Cat_V-to-Vdelta}.
	Using this construction and
	the isomorphisms \( \psi^p \), one can
	then extend \( (\Category M, R) \) into a \( \Vdelta \)module
	structure.
	Finally, letting \( \phi \coloneqq \psi^1 \),
	one can show that for this \( \Vdelta \)module structure
	\( (\Category M, R) \), we have
	\( \phi^p = \psi^p \) 
	as in the definition of right representable modules;
	it is thus right representable.

	We conversely now prove that any right
	representable \( \Vdelta \)module yields
	a category enriched over \( \Category V \).
	Let \( \Category M \) be a right representable \( \Vdelta \)module 
	with bifunctor \( \EnrichedHom{-}{-} \). Let us start by showing that
	the \( \Vdelta \)module associators induce an enrichment
	over \( \Category V \) on the category \( \Category M \) endowed
	with \( \EnrichedHom{-}{-} \).
	Thanks to the associator
	\[
		R^1\,(-,R^1) \longrightarrow R^2
	\]
	of the right representable \( \Vdelta \)module
	structure, one can obtain the composite map
	\[
		(\Category V \times \Category V)
		\inout{a \times b}{\EnrichedHom{m}{y} \times \EnrichedHom{x}{m}}
		\to
		(\Category V \times \Category V)
		\inout{\oplax^0(a) \times \oplax^0(b)}{\EnrichedHom{m}{y}
		\times \EnrichedHom{x}{m}}
		\to
		\Integral{m}
		\Category V \inout{\oplax^0(a)}{\EnrichedHom{m}{y}}
		\times 
		\Category V \inout{\oplax^0(b)}{\EnrichedHom{x}{m}} 
	\to \Category V \inout{a \otimes_\oplax b}{\EnrichedHom{x}{y}}
	\]
	natural in \( a, b \).
	Such maps are in bijection
	with transformations
	\[
		\begin{tikzcd}
			\EnrichedHom{m}{y} \otimes_\oplax
			\EnrichedHom{x}{m}
			\ar[r, "\mu_{x,m,y}"]
			& \EnrichedHom{x}{y}
		\end{tikzcd}
	\]
	which are natural in \( x, y \) and
	extranatural in \( m \).
	In a simpler fashion, the unit
	\( \IdentityFunctor_{\Category M} \to R^0 \)
	yields
	\[
		\begin{tikzcd}
			1_\oplax 
			\ar[r, "\eta_x"]
			& \EnrichedHom{x}{x}
		\end{tikzcd}
	\]
	which is extranatural in \( x \).
	
	The associativity condition for \( [-,-] \)
	follows from this
	\[
		\begin{tikzcd}
			R^1 \circ (-,R^1(-,R^1))
			\dar
			\rar["\IsIsomorphicTo", dash]
			&[-3ex]
			R^1(-,R^1) \circ (-,R^1)
			\rar
			&[3ex] R^2(-,R^1)
			\ar[d] \\
			R^1(-,R^2)
			\ar[rr]
			&& R^3
		\end{tikzcd}
	\]
	sequential composition diagram, while the two unitality conditions
	follow from these two
	\[
		\begin{tikzcd}
		    \IdentityFunctor_{\Category M} \circ R^p
		    \rar
		    \ar[dr,
			"\rotatebox{-25}{$\IsIsomorphicTo$}", outer sep = -1.5pt]
		    & R^0 \circ R^1
		    \dar \\
		    & R^1
		\end{tikzcd}
		\qand
		\begin{tikzcd}
		    R^1(-, \IdentityFunctor_{\Category M})
		    \rar
		    \ar[dr, "\rotatebox{-25}{$\IsIsomorphicTo$}",
				outer sep = -1.5pt]
		    & R^1(-,R^0)
		    \dar \\
		    & R^1
		\end{tikzcd}
	\]
	unitality diagrams.
	This concludes the proof that any right
	representable \( \Vdelta \)modules induces
	a category enriched over \( \Category V \).

	The fact that these two procedures are inverse to one another
	follows from integral calculus, the counitality of \( \oplax^0 \)
	and from the fact that
	the structure of a right representable \( \Vdelta \)module
	can be reconstructed from the two transformations
	\( \IdentityFunctor_{\Category M} \to R^0 \)
	and \( R^1(-,R^1) \to R^2 \) (together with the oplax structure
	of \( \Category V^\oplax \)) as explained earlier
	\DoubleUnskipRef{remark: Vdelta truncation}
	{lemma: rep Vdelta truncation}.
\end{proof}

We have previously described how to create a right representable
\( \Vdelta \)module from a category enriched over \( \Category V \).
This construction can be extended into a 2-functor
\[
	\Cat_{\Category V} \longrightarrow 
		\Vdelta\mathsf{mod}^\mathrm{fun}_\mathrm{oplax}
\]
in the obvious way.

\begin{remark}
	\label{remark: trunction of oplax functors between deltamods}
	Let \( (F, \{f^p\}) \From \Category M^R \to \Category N^S \)
	be an oplax functor between two \( \Vdelta \)modules.
	If \( \Category M^R \) is right representable, then
	each \( f^p \) can be rewritten as
	\[
		\begin{tikzcd}
			R^p \inout {-, x} y
			\ar[dr]
			\ar[ddd,"f^p"] &
			\\
			[-2ex]
			& R^1 \inout {\oplax^{p-1}(-), x} y
			\ar[d, "f^1"]
			\\
			[4ex]
			& S^1 \inout {\oplax^{p-1}(-), F(x)} {F(y)}
			\ar[dl]
			\\
			[-2ex]
			S^p \inout {-, F(x)} {F(y)} 
		\end{tikzcd}
	\]
	thanks to the second diagram defining
	oplax functors
	\UnskipRef{def:morphism_V-delta_mod}
	and thanks to the canonical section
	\( R^p \to R^1(\oplax^{p-1}, -) \)
	\UnskipRef{remark: Vdelta truncation}.
\end{remark}

\begin{theorem}
	Let \( \Category V^\oplax \) be a normal oplax monoidal
	category, then
	the above map induces a 2-equivalence
	\[
		\Cat_{\Category V}
		=
		\mathsf{right~rep.~}
		\Vdelta\mathsf{mod}^\mathrm{fun}_\mathrm{oplax}
	\]
	between the 2-category of categories enriched over \( \Category V \)
	and the one of right representable \( \Vdelta \)modules.
\end{theorem}

\begin{proof}
	The previous proposition implies essential surjectivity on objects.
	We now turn to 1-fully faithfulness.
	Let  \( (\Category B, [-,-]) \)
	and \( (\Category C, \langle -, - \rangle) \) be two
	categories enriched over \( \Category V \) and let us denote by
	\( \Category B^R \) and \( \Category C^S \)
	the two respective
	associated \( \Vdelta \)modules.
	We need to build from every oplax functor
	\( (F, \{f^p\}) \From \Category B^R \to \Category C^S \)
	a functor \( F \From \Category B \to \Category C \)
	enriched over \( \Category V \).
	The underlying functor will be the same. For its additional
	structure,
	we start by
	noting that since \( \Category V^\oplax \) is normal, one
	gets natural transformations
	\[
		\begin{tikzcd}[sep=large]
			R^1 \inout {-,x} y \coloneqq \Category V\inout{-}{\EnrichedHom{x}{y}}
			\rar["f^1"]
			& \Category V\inout{-}{\langle F(x), F(y) \rangle}
			\eqqcolon S^1 \inout {-, x} y
		\end{tikzcd}
	\]
	from which one extracts the maps \( F_{x, y} \From [x, y] \to
	\langle F(x), F(y) \rangle \).
	These maps satisfy the composition and unit axioms for
	\( \Category V \)-enriched functors
	\UnskipRef{definition: V-enriched functor abridged}
	because the diagrams
	\[
		\begin{tikzcd}[ampersand replacement=\&]
			\Category B \inout x y
			\arrow[r, ""]
			\arrow[d, "", swap]
			\& R^0 \inout x y
			\arrow[d, "f^0"] \\
			\Category C \inout {F(x)} {F(y)}
			\arrow[r, "", swap]
			\& S^0 \inout {F(x)} {F(y)}
		\end{tikzcd}
		\qand
		\begin{tikzcd}[ampersand replacement=\&]
			R^1(-, R^1) \inout {-,x} y
			\arrow[d, "f^1 \times f^1"]
			\arrow[r, "", swap]
			\& R^2 \inout {-,x} y
			\arrow[d, "f^2"] \\
			S^1(-, S^1) \inout {-, F(x)} {F(y)}
			\arrow[r, "", swap]
			\& S^2 \inout {-, F(x)} {F(y)}
		\end{tikzcd}
	\]
	commute and because both \( f^0 \) and \( f^2 \)
	can be rewritten using
	the components \( F_{x,y} \)
	\UnskipRef{remark: trunction of oplax functors between deltamods}.
	From this process one receives a \( \Category V \)-enriched functor.
	One can check that this process is inverse to the one
	sending \( \Category V \)-enriched functors to oplax functors using
	the usual arguments in addition to the remark above that
	all maps \( f^p \) can be recovered from \( f^1 \)
	\UnskipRef{remark: trunction of oplax functors between deltamods}.

	The proof of the 2-fully faithfulness can be obtained similarly.
\end{proof}

\subsection{Categories enriched over \( \Category V \)
vs \( \Category V \)\=/categories}
\label{sec:V-Cat&Cat_V}

As mentioned at the beginning of this section,
the definition of a category \( \Category C \)
enriched over an oplax monoidal category
\( \Category V^\oplax \) 
differs from the standard definition of
a \( \Category V \)\=/category in that the former
consists in a category together with additional
structure, whereas the latter is a collection of
objects and hom-objects together with composition
maps, from which an underlying category is recovered.

\begin{definition}[(\( \Category V \)\=/category)]
	\label{def:V-cat}
	A \( \Category V \)\=/category
	\( \Category C \)
	is the data of
	\begin{itemize}
		\item
			a set of objects \( \Objects(\Category C) \);
		\item
			an object \( \EnrichedHom{x}{y} \) of
			\( \Category V \)
			for each pair of objects
			\( x, y \in \Objects(\Category C) \);
		\item
			for each \( x \in \Objects(\Category C) \),
			\[
				\begin{tikzcd}
					1_\oplax
					\rar["\etaVCat_x"]
						& \EnrichedHom{x}{x}
				\end{tikzcd}
			\]
			a unit map;
		\item
			for all triplets of objects
			\( x, y, z \in \Objects(\Category C) \),
			morphisms
			\[
				\begin{tikzcd}[column sep = large]
					\EnrichedHom{y}{z}\otimes_\oplax
					\EnrichedHom{x}{y}
					\rar["\muVCat_{x, y, z}"]
						& \EnrichedHom{x}{z}
				\end{tikzcd}
			\]
	\end{itemize}
	such that:
	\begin{description}
		\item[Associativity]
		the associativity diagram in \( \Category V \)
	\[
		\begin{tikzcd}[row sep = large]
			&[-15ex]
			\EnrichedHom y z \otimes_\oplax
			\EnrichedHom x y \otimes_\oplax
			\EnrichedHom w x
			\arrow[dl]
			\arrow[dr] &[-15ex]
			\\
			\big(\EnrichedHom y z \otimes_\oplax \EnrichedHom x y\big)
			\otimes_\oplax \EnrichedHom{w}{x}
			\arrow[d, "\muVCat_{x,y,z} \otimes_\oplax -"{left}]
			&&
			\EnrichedHom{y}{z} \otimes_\oplax
			\big(\EnrichedHom x y \otimes_\oplax \EnrichedHom w x\big)
			\arrow[d, "- \otimes_\oplax \muVCat_{w,x,y}"]
			\\
			\EnrichedHom x z \otimes_\oplax \EnrichedHom w x
			\arrow[dr, "\muVCat_{w,x,z}"{below left}]
			&& \EnrichedHom y z \otimes_\oplax \EnrichedHom w y
			\arrow[dl, "\muVCat_{w,y,z}"]
			\\
			& \EnrichedHom{w}{z} &
		\end{tikzcd}
	\]
	commutes for every \( w, x, y, z \in \Objects(\Category C) \);
	\item[Unitality]
		the two unitality diagrams in \( \Category V \)
		\[
			\begin{tikzcd}[column sep = small]
				& \OplaxHom{x, y}
				\ar[ddd]
				\ar[dl]
				\\[-3ex]
				1_\oplax \otimes_\oplax
				\EnrichedHom{x}{y}
				\dar["\etaVCat_y \otimes_\oplax -" swap]
				\\[6ex]
				\EnrichedHom y y \otimes_\oplax \EnrichedHom x y
				\ar[dr, "\muVCat_{x, y, y}" swap]
				\\
				& \EnrichedHom{x}{y}
			\end{tikzcd}
			\qand
			\begin{tikzcd}[column sep = small]
				\OplaxHom{x,y}
				\ar[ddd]
				\ar[dr]
				\\[-3ex]
				& \EnrichedHom{x}{y}
				\otimes_\oplax 1_\oplax
				\dar["- \otimes_\oplax \etaVCat_x"]
				\\[6ex]
				& \EnrichedHom x y \otimes_\oplax \EnrichedHom x x
				\ar[dl,"\muVCat_{x,x,y}"]
				\\
				\EnrichedHom{x}{y}
			\end{tikzcd}
		\]
		commute for every \( x, y \in \Objects(\Category C) \).
	\end{description}
\end{definition}

In other words, the previous definition of
a \( \Category V \)\=/category \( \Category C \)
is almost identical to that of a category enriched
over \( \Category V \) given previously
\UnskipRef{def:oplax_enriched_abridged}
except  that we do not require \( [-,-] \) to be
a bifunctor but simply a map between pairs of objects
in \( \Category C \) to an object in \( \Category V \)
---the hom-objects---nor do we require
the composition  and unit maps \( \muVCat \) and \( \etaVCat \)
to be extra/natural transformations.

\begin{remark}
	Note that one could adapt the extended definition
	of a category enriched over \( \Category V \)
	given above \UnskipRef{def:cat_V-extended}
	so as to obtain a \( \Category V \)\=/category with
	composition maps \( m^n \) for all \( n \geq -1 \)
	which are not required to be extra/natural transformations.
	As in the case of categories enriched over
	\( \Category V \), these two possible definitions
	(abridged \vs extended)
	turn out to be equivalent (due to the same
	mechanism).
	This has been previously remarked by Leinster
	\cite[2.2]{arXiv:9901139}.
\end{remark}

To each \( \Category V \)\=/category \( \Category C \) one can associate
the corresponding underlying category \( \uC \)
endowed with the same objects \( x, y, \dots \) and maps of
the form \( \fVCat \From 1_\oplax \to \EnrichedHom{x}{y} \)
as morphisms. Note that the composition of these
underlying morphisms---and more importantly
its associativity---is due to the comonoid structure
of the oplax unit \( 1_\oplax \) \UnskipRef{lemma: 1 is a comonoid}.
Similarly, one can define \( \Category V \)\=/functors as follows.

\begin{definition}[(\( \Category V \)\=/functor)]
	A \( \Category V \)-functor
	\( F \From \Category B \to \Category C \)
	between two \( \Category V \)-categories
	\( (\Category B, [-,-], \etaVCat, \muVCat) \)
	and \( (\Category C, \langle -, -\rangle, \thetaVCat, \nuVCat) \)
	consists in the following data:
	\begin{itemize}
		\item
			a map which assigns to each object \( x \)
			of \( \Category B \) an object
			\( F(x) \) of \( \Category C \);
		\item
			a map which assigns to each pair of objects
			\( x, y \) of \( \Category B \) a morphism
			\[
				\begin{tikzcd}[column sep = huge]
					\EnrichedHom{x}{y}
					\rar["F_{x,y}"]
						&
						\langle F(x), F(y) \rangle
				\end{tikzcd}
			\]
			in \( \Category V \),
	\end{itemize}
	such that:
	\begin{description}
		\item[unit]
			the unit diagram in \( \Category V \)
			\[
				\begin{tikzcd}[column sep=huge, row sep=large]
						1_\oplax
						\ar[d, "{\etaVCat_x}",swap]
					\ar[rd,"{\thetaVCat_{F(x)}}"] &
					\\
					\EnrichedHom{x}{x}
					\ar[r, "F_{x,x}"]
						& \langle F(x), F(x) \rangle
				\end{tikzcd}
			\]
			commutes for every \( x \in \Objects(\Category B) \);
		\item[composition]
			the composition diagram in \( \Category V \)
			\[
				\begin{tikzcd}[column sep = 70, row sep = large]
					\EnrichedHom{y}{z}
					\otimes_\oplax
					\EnrichedHom{x}{y}
					\dar["{\muVCat_{x,y,z}}" swap]
					\rar["F_{y,z} \otimes_\oplax F_{x, y}"]
						& \langle F(y), F(z) \rangle
						\otimes_\oplax
						\langle F(x), F(y) \rangle
						\dar["{\nuVCat_{F(x), F(y), F(z)}}"]
					\\
					\EnrichedHom{x}{z}
					\rar["F_{x,z}"]
						& \langle F(x), F(z) \rangle
				\end{tikzcd}
			\]
			commutes for every \( x, y, z \in \Objects(\Category B) \).
	\end{description}
\end{definition}

Again, the definition of a \( \Category V \)\=/functor differs from the one of
a \( \Category V \)\=/enriched functor by the fact that \( F \) is not functorial
and the components \( F_{x,y} \) are not assumed to be natural. 

Any \( \Category V \)\=/functor \( F \From \Category B \to \Category C \) induces an 
underlying functor \( \uF\From\uB \to \uC \)  
between the corresponding underlying categories.
The action of \( \uF \) is given by the map \( x \mapsto F(x) \)
on objects and associates to any underlying morphism
\( \fVCat \From 1_\oplax \to \EnrichedHom x y \), the composite
\[
	\begin{tikzcd}
		1_\oplax
		\ar[r, "\fVCat"]
		& \EnrichedHom x y
		\ar[r, "F_{x,y}"]
		& \langle F(x), F(y) \rangle
	\end{tikzcd}
\]
interpreted as a morphism in \( \uC \).

In order to define natural \( \Category V \)\=/transformations, we shall focus on 
the case where the category \( \Category V \) is normal
so that the  structure map
	\(
		\oplax^0(x) \to x
	\)
	is invertible for every object \( x \in \Objects(\Category V) \).
In the normal case, we recover the standard result that the bracket
\( \EnrichedHom{-}{-} \) is bifunctorial on the underlying category.

\begin{proposition}[(\( \EnrichedHom{-}{-} \) is a bifunctor)]
\label{prop: EnrichedHom is a bifunctor}
	Let \( \Category V \) be a normal oplax monoidal category.
	Let furthermore \( (\Category C, \EnrichedHom{-}{-}, \eta, \mu) \) be a 
	\( \Category V \)\=/category and let
	\( \uC \) denote its underlying
	category. The map
	\[
		\begin{tikzcd}[sep=large]
			\EnrichedHom{-}{-}\From
			\uC\Op \times
			\uC
			\rar
			& \Category V
	   \end{tikzcd}
	\]
	which assigns
	\begin{itemize}
	\item
		to any pair of objects
		$x,y \in \Objects(\Category C)$,
		the object $\EnrichedHom{x}{y}$
		of \( \Category V \);
	\item
		to any pair of morphisms
		$\hVCat_1 \From 1_\oplax
		\to \EnrichedHom{y_1}{x_1}$
		and $\hVCat_2 \From 1_\oplax
		\to \EnrichedHom{x_2}{y_2}$,
		the morphism
		\[
			\begin{tikzcd}[column sep = large]
				\EnrichedHom{x_1}{x_2}
				\ar[r, "\EnrichedHom{\hVCat_1}{\hVCat_2}"]
				& \EnrichedHom{y_1}{y_2}
			\end{tikzcd}
		\]
		defined as
		\[
			\begin{tikzcd}[row sep=large, column sep=small]
				\EnrichedHom{x_1}{x_2}
				\ar[dr]
				\ar[ddd, "\EnrichedHom{\hVCat_1}{\hVCat_2}"] & \\
				& 1_\oplax \otimes_\oplax \EnrichedHom{x_1}{x_2} \otimes_\oplax 1_\oplax
				\ar[d, "{\hVCat_2 \otimes_\oplax -\, \otimes_\oplax \hVCat_1}"] \\
				& \EnrichedHom{x_2}{y_2} \otimes_\oplax \EnrichedHom{x_1}{x_2}
								\otimes_\oplax \EnrichedHom{y_1}{x_1}
				\ar[dl, "m^2"] \\
				\EnrichedHom{y_1}{y_2}
			\end{tikzcd}
		\]
		where the first diagonal arrow is either one of the composite
		\[
			\begin{tikzcd}[sep=small]
				& 1_\oplax \otimes_\oplax \EnrichedHom{x_1}{x_2}
				\ar[dr] & \\
				\EnrichedHom{x_1}{x_2}
				\ar[ur]
				\ar[dr]
						&& 1_\oplax \otimes_\oplax \EnrichedHom{x_1}{x_2} \otimes_\oplax 1_\oplax\\
				& \EnrichedHom{x_1}{x_2} \otimes_\oplax 1_\oplax 
				\ar[ur] &
			\end{tikzcd}
		\]
		which are identical due to the sequential decomposition conditions;
	\end{itemize}
	is a bifunctor.
\end{proposition}

\begin{proof}
	Similar to the strong monoidal case \cite{MR2177301}. 
\end{proof}

\begin{remark}
	Note that, while the bracket \( \EnrichedHom{-}{-} \)
	is not generically bifunctorial in the non-normal case, 
	the composite \( \OplaxHom{-,-} \) is always a bifunctor. 
	This fact arguably constitutes the most
	important departure from the standard case (wherein
	one enriches over a strong monoidal category). 
	Indeed,
	when enriching over an oplax monoidal category,
	the assignment of a pair of objects \( (x, y) \)
	of a \( \Category V \)\=/category \( \Category C \) to
	the corresponding hom-object \( [x,y] \) of \( \Category V \)
	cannot be upgraded to a bifunctor unless the enrichment
	is normal.
\end{remark}


\begin{definition}[(Natural \( \Category V \)\=/transformations)]
	Let \( \Category V^\oplax \) be a normal oplax monoidal category.
	Let furthermore \( F, G \From \Category B \to \Category C \)
	be two \( \Category V \)\=/functors between two
	\( \Category V \)\=/categories
	\( (\Category B, [-,-],\eta,\mu) \)
	and \( (\Category C, \langle-,-\rangle, \theta, \nu) \).
	A natural \( \Category V \)\=/transformation
	\( \alpha \From F \Rightarrow G \) is the data of a map
	assigning to each object \( x \) of \( \Category B \)
	a morphism
	\[
		\begin{tikzcd}
			1_\oplax
			\rar["\alphaVCat_x"]
				& \langle{F(x)},{G(x)}\rangle
		\end{tikzcd}
	\]
	of \( \Category V \) such
	that the diagram
	\[
		\begin{tikzcd}[sep=large]
			\EnrichedHom{x}{y}
			\arrow[r, "F_{x,y}"]
			\arrow[d, "G_{x,y}" swap]
			& \langle F(x), F(y) \rangle
			\arrow[d, "{\langle -, \alpha_y \rangle}"] \\
			\langle G(x), G(y) \rangle
			\arrow[r, "{\langle \alpha_x, - \rangle}"]
			& \langle F(x), G(y) \rangle
		\end{tikzcd}
	\]
	commutes for every pair \( x, y \in \Objects(\Category B) \).
\end{definition}

\( \Category V \)\=/categories, \( \Category V \)\=/functors
and natural \( \Category V \)\=/transformations can be assembled into
a \( 2 \)\=/category that we shall denote
\( \VCats \).
The construction associating to each \( \Category V \)-category
\( \Category C \) its underlying category \( \uC \) naturally
extends into a 2-functor \( \VCats \to \Cat \).
In fact, the underlying category of a \( \Category V \)-category
is also naturally endowed with an enrichment structure over
\( \Category V \) due to the bifunctoriality of \( [-,-] \),
hence one gets a 2-functor
\[
	\VCats \longrightarrow \Cat_{\Category V}
\]
whose image is made of \emph{normal} categories.

\begin{definition}[(Normal category enriched over \( \Category V \))]
	We shall say that a category \( (\Category C, [-,-]) \) enriched over \( \Category V^\oplax \)
	is normal if the canonical map
	\[
		\Category C \inout x y \longrightarrow \Category V \inout
		{1_\oplax} {[x,y]}
	\]
	achieved by sending each \( f \From x \to y \) to either of the
	equal composite maps
	\[
		\begin{tikzcd}[row sep=0, column sep = huge]
			& {[x,x]}
			\ar[dr, "{[-, f]}"] &
			\\
			1_\oplax
			\ar[ur, "\eta_x"]
			\ar[dr, "\eta_y" swap]
					&& {[x,y]}
			\\
			& {[y,y]}
			\ar[ur, "{[f, -]}" swap] &
		\end{tikzcd}
	\]
	is an isomorphism for every \( x, y \in \Category C \).
\end{definition}

\begin{proposition}
	\label{theorem: V-Cat vs Cat-V}
	Let \( \Category V^\oplax \) be a normal oplax monoidal category.
	The \( 2 \)\=/category \( \VCats \)
	of \( \Category V \)\=/categories
	is a full reflective \( 2 \)\=/subcategory
	\[
		\begin{tikzcd}[ampersand replacement=\&, column sep = huge]
			\VCats
			\arrow[r, hook, shift right=2, swap, ""]
			\&
			\arrow[l, shift right=2, swap, ""]
			\Cat_{\Category V}
		\end{tikzcd}
	\]
	of the \( 2 \)\=/category 
	\( \Cat_{\Category V} \) of categories enriched over
	\( \Category V \).
\end{proposition}

\begin{proof}
	The reflector is simply given by forgetting the underlying
	category structure of a category enriched over \( \Category V \)
	and only remembering its underlying \( \Category V \)-categorical
	structure.

	This results in an endofunctor
	\( (\Category C, [-,-]) \mapsto (\uC, [-, -]) \)
	of \( \Cat_{\Category V} \).
	Using the construction described in the definition above,
	one obtains a unit functor
	\[
		\begin{tikzcd}
			(\Category C, [-,-])
			\ar[r, "u_{\Category C}"]
			& (\uC, [-,-])
		\end{tikzcd}
	\]
	which easily extends into a 2-natural map \( u \).
	Finally it is enough to observe that by construction,
	both \( u_{\uC} \) and \( \underline{u_{\Category C}} \) are
	identities for every \( \Category C \).
\end{proof}

\begin{remark}
	When \( \Category V^\oplax \) is normal,
	one has a 2-isomorphism
	\[
		\VCats = \mathsf{normal~}\Cat_{\Category V}
	\]
	between \( \Category V \)-categories and normal categories
	enriched in \( \Category V \).
	When \( \Category V^\oplax \) is not normal, one can no longer use
	the 2-category of \( \Category V \)-categories.
	In this case, it becomes natural to replace the familiar notion of
	\( \Category V \)-category with the notion of normal
	category enriched over \( \Category V \).
\end{remark}

\subsection{An example: the theory of operads}
\label{sec:operads}

Any category enriched over a monoidal
category constitutes an example of an oplax enriched monoidal category.
In the present subsection, we shall present a genuine example of enrichment for which
the oplax monoidal category is not strong.

Let \( \Category V^\otimes \) be a symmetric monoidal category
with countable coproducts.
Let \( \ReducedSequences(\Category V) \)
denote the category of reduced sequences
in \( \Category V \) \ie the category whose objects are
sequences \( M(1), M(2), \dots,  \) of objects of \( \Category V \)
indexed from \( 1 \).

The category of sequences \( \ReducedSequences(\Category V) \)
admits a (normal) oplax monoidal structure
\cite[2.17]{doi:10.1007/s40062-012-0007-2}
with unit
\[
	1_{\Insertion}
	\coloneqq (1_{\Category V}, \emptyset_{\Category V},
	\emptyset_{\Category V}, \dots)
\]
and whose maps \( \Insertion^1 \) and \( \Insertion^2 \)
are given by
\[
	(M \Insertion N)(n) \coloneqq
	\coprod_p
	\coprod_{n_1 + \dots + n_p = n}
	M(p) \otimes_{\Category V} N(n_1)
	\otimes_{\Category V} \dots \otimes_{\Category V}
	N(n_p)
\]
\begin{multline*}
	(M \Insertion N \Insertion O)(n)
	\coloneqq
	\coprod_{p, q}
	\coprod_{p_1 + \dots + p_q = p}
	\coprod_{n_1 + \dots + n_p = n}
	M(q) \otimes_{\Category V} N(p_1)
	\otimes_{\Category V} \dots
	\otimes_{\Category V} N(p_q)
	\\
	\otimes_{\Category V} O(n_1)
	\otimes_{\Category V} \dots \otimes_{\Category V}
	O(n_p)
\end{multline*}
where all integers in the formulæ are positive.
The higher maps \( \Insertion^n \) can be obtained via
the combinatorics of the wreath product of planar trees
\cite{Cooperads}.
The structural decomposition natural transformations are obtained
by permutation of factors using the distributivity of coproducts over bifunctors.

A monoid in \( \ReducedSequences(\Category V)^\Insertion \)
is a reduced planar monochromatic operad in \( \Category V \).
In the case where the tensor structure of \( \Category V \)
commutes with countable coproducts, the oplax monoidal structure
on reduced sequences becomes strong.
It is a fundamental tool in the theory of operads.

Let \( \Category C^\otimes \) be a
monoidal \( \Category V \)\=/category with
enrichment bifunctor
\( \langle - , - \rangle \).
Then \( \Category C \) is naturally a
\( \ReducedSequences(\Category V) \)-category
with enrichment bifunctor
\[
	\EnrichedHom{x}{y}(n) \coloneqq
	\left\langle x^{\otimes_{\mathcal C} n}, y \right\rangle
\]
with unit maps \( 1_{\Insertion} \to \EnrichedHom{x}{x} \)
given by the unit map of the \( \Category V \)\=/enrichment
of \( \Category C \) and whose composition maps
\[
	\EnrichedHom{y}{z} \Insertion \EnrichedHom{x}{y}
	\longrightarrow \EnrichedHom{x}{z}
\]
arise simply from the composition maps
\[
	\left\langle y^{\otimes_{\Category C} p}, z \right\rangle
	\otimes_{\Category V}
	\left\langle x^{\otimes_{\Category C} n_1}, y \right\rangle
	\otimes_{\Category V}
	\dots
	\otimes_{\Category V}
	\left\langle x^{\otimes_{\Category C} n_p}, y \right\rangle
	\longrightarrow
	\left\langle x^{\otimes_{\Category C} n} , z \right\rangle
\]
for the tensor structure of \( \Category C \).

For \( P \) a reduced planar monochromatic operad in \( \Category V \)
viewed as a monoid in \( \ReducedSequences(\Category V)^\Insertion \),
a morphism of monoids
\[
	P \longrightarrow \EnrichedHom{\Lambda}{\Lambda}
\]
is equivalent to endowing the object
\( \Lambda \in \Objects(\Category C) \) with a \( P \)-algebra
structure.
The theory of oplax enriched categories thus allows one to
extend the theorem `algebras over operads can be viewed
as representations of monoids’
\cite[B.1.2]{doi:10.1007/s40062-019-00252-1}
to the case of operads in
symmetric monoidal categories that are not closed.

We have chosen to describe the case of reduced planar monochromatic
operads for its simplicity.
The same arguments extend directly to the general case of
symmetric coloured operads with operations in arity \( 0 \)
\cite{Cooperads}, though the formulæ for the definition
of \( \Insertion \) become more involved.

\section{From \( \Category V \) to \( \Cat_{\Category V} \)}
\label{sec:V to Cat_V}

Now that we have defined the notion of enrichment
over a given oplax monoidal category \( \Category V \), 
we shall discuss how the enrichment structure varies
with the enrichment base.
A standard result for
categories enriched over a monoidal category states
that a lax monoidal functor
\( f \From \Category U \to \Category V \) between two 
monoidal categories induces a \( 2 \)\=/functor
\( f_\ast \From \UCats \to \VCats \) 
between the associated categories of (normal) enriched categories
\cite{MR2177301, doi:10.1007/978-3-642-99902-4_22}.
This result can be enhanced
to the case where
\( \Category U \) and \( \Category V \) are oplax monoidal
(and the enrichments may not be normal).
In this section, we shall construct a \( 2 \)\=/functor
\( \Cat_{(-)} \From \Oplax \to \TwoCat \).
The action of this \( 2 \)\=/functor
can be summarised in the following
\[
	\begin{tikzcd}[sep=large]
		\Category U
		\ar[r, bend left=40, "f"{name="f"}]
		\ar[r, bend right=40, "g"{below, name="g"}]
		& \Category V
		\arrow[Rightarrow, shorten <= 2pt,
		shorten >= 2pt, "\alpha", from="f", to="g"]
	\end{tikzcd}
	\qquad \longmapsto \qquad
	\begin{tikzcd}[sep=large]
		\Cat_{\Category U}
		\ar[r, bend left=40, "f_*"{name="f"}]
		\ar[r, bend right=40, "g_*"{below, name="g"}]
		& \Cat_{\Category V}
		\arrow[Rightarrow, shorten <= 2pt,
		shorten >= 2pt, "\alpha_*", from="f", to="g"]
	\end{tikzcd}
\]
diagram.

After reviewing its definition
\SectionRef{sec:pushforward}, we shall
show that it is equipped with a lax monoidal structure,
corresponding to the external product of enriched categories
(\ie the operation that, given two categories enriched over
two different oplax monoidal categories, constructs a category
enriched over the product of the latter)
\SectionRef{sec:external_product}.
Finally we shall deduce that when \( \Category D^{\lax, \oplax} \)
is lax-oplax duoidal, the 2-category of \( \Category D^\oplax \)-enriched 
categories naturally inherits a lax monoidal structure
\SectionRef{sec: D-enriched}.

\subsection{Pushforward}
\label{sec:pushforward}

\paragraph{Pushforward of lax functors}
Let \( \Category U^\psi, \Category V^\oplax \) be
two oplax monoidal categories and let
\( \ChangeEnrichment \From \Category U \to \Category V \)
be a functor endowed with a lax structure
\( \{l^n\}_{n \geq -1} \).
We shall describe a \( 2 \)-functor
\[
	f_\ast \From \Cat_{\Category U} \to \Cat_{\Category V}
\]
between the corresponding 2-categories of enriched categories.
Let \( (\Category C, \EnrichedHom{-}{-}, \{m^n\}) \)
be a category enriched over \( \Category U \). From it,
we can define a category \( (\ChangeEnrichment_*\Category C,
\InducedHom{-,-}, \{m_\ChangeEnrichment^n\}) \) enriched
over \( \Category V \) as follows:
\begin{itemize}
	\item its underlying category
		\(\ChangeEnrichment_\ast\Category C \) is
		the original category \( \Category C \);
	\item the enrichment bifunctor is given by
		\[
			\InducedHom{-,-} \coloneqq \ChangeEnrichment\circ[-,-]
		\]
		\ie it is the composition of the enrichment bifunctor
		of \( \Category C \) and the lax functor
		\( \ChangeEnrichment \);
	\item
		the composition maps \( m^n_\ChangeEnrichment \)
		are defined, using the functor \( \ChangeEnrichment \)
		and its lax monoidal structure, as
		\[
			\begin{tikzcd}
				m^n_\ChangeEnrichment \From\,
				{[-, \dots, -]}_{\oplax^n\,\ChangeEnrichment}
				\rar["l^n"]
				& {[-, \dots, -]}_{\ChangeEnrichment \,\psi^n}
				\rar["\ChangeEnrichment(m^n)"]
				& \InducedHom{-,-}
			\end{tikzcd}
		\]
		for any integer \( n \geq -1 \).
\end{itemize}
The (extra)naturality conditions of $m^n_\ChangeEnrichment$ 
follow straightforwardly from the (extra)naturality of $m^n$
and the naturality of $l^n$. 
The morphism \( m_\ChangeEnrichment^0 \) verifies the counitality
condition of the structure maps of a category enriched over
\( \Category V \), \ie it identifies with the counit structure
map \( \oplax^0 \to \IdentityFunctor \), as a direct consequence
of the unitality condition of the lax monoidal structure of
\( \ChangeEnrichment \). Moreover, this family of morphisms also
verifies the additivity conditions of the structure maps of
a category enriched over \( \Category V \),
as a consequence of the additivity conditions of
the lax monoidal structure \( \{l^n\}_{n \geq -1} \)
of \( \ChangeEnrichment \).

Next, let us define the action of \( \ChangeEnrichment_* \)
on \( 1 \)\=/morphisms, \ie enriched functors.
Given a \( \Category U \)\=/enriched functor 
\( F \From \Category B \to \Category C \)
between two categories \( (\Category B, [-,-], \{m^p\}) \)
and \( (\Category C, \langle-,-\rangle, \{n^p\}) \)
enriched over \( \Category U \), we define
the \( \Category V \)\=/enriched functor
\[
	\ChangeEnrichment_*F \From \ChangeEnrichment_*\Category B
	\to \ChangeEnrichment_*\Category C
\]
as the same underlying functor
\( F \From \Category B \to \Category C \),
equipped with the natural transformation
with components
\[
	(\ChangeEnrichment_* F)_{x,y} \coloneqq f(F_{x,y}) \From \InducedHom{x,y}
	\to \langle \ChangeEnrichment_*F(x),
	\ChangeEnrichment_*F(y) \rangle_\ChangeEnrichment
\]
for any pair of objects \( x, y \) in \( \Category B \).
The naturality of the lax monoidal structure
\( l^n \) on the functor \( \ChangeEnrichment \),
together with the fact that \( F \) is
a \( \Category U \)\=/enriched functor ensures that
the above does define a \( \Category V \)\=/enriched functor.
On top of that, it is clear from the functoriality
of \( \ChangeEnrichment\) that \( \ChangeEnrichment_\ast \)
intertwines the composition of \( \Category U \)\=/enriched
and \( \Category V \)\=/enriched functors.
Finally, a natural \( \Category U \)\=/enriched transformation
\( \alpha \From F \Rightarrow G \) between two natural
\( \Category U \)\=/enriched functors 
\( F,G \From \Category B \to \Category C \) is
automatically a natural \( \Category V \)\=/enriched
transformation between the \( \Category V \)\=/enriched
functors 
\( \ChangeEnrichment_\ast F,\ChangeEnrichment_\ast G
\From \ChangeEnrichment_\ast \Category B
\to \ChangeEnrichment_\ast \Category C \).
Indeed, using the above definition, one finds
that the \( \Category V \)\=/naturality condition
between these \( \Category V \)\=/enriched functors
is nothing but the image of the \( \Category U \)\=/naturality
condition (by functoriality of \( \ChangeEnrichment \)),
so that we can simply define
\( \ChangeEnrichment_\ast\alpha = \alpha \). 

It is also clear that this definition ensures
that \( \ChangeEnrichment_\ast \) preserves
the vertical and horizontal compositions of natural
\( \Category U \)\=/enriched and \( \Category V \)\=/enriched 
transformations. As a consequence, we proved the claim
made at the beginning of the present section that any lax monoidal
functor \( f \From \Category U \to \Category V \)
between two oplax monoidal categories induces
a \( 2 \)\=/functor
\( f_* \From \Cat_{\Category U} \to \Cat_{\Category V} \),
referred to as the pushforward of $f$.

\paragraph{Pushforward of a monoidal transformation}

Now let \( t \From f \Rightarrow g \) be a monoidal natural
transformation between two functors
\( f,g \From \Category U^\psi \to \Category V^\oplax \)
with lax monoidal structure \( \{k^n\}_{n \geq -1} \)
and \( \{l^n\}_{n \geq -1} \) respectively.
We can define a \( 2 \)\=/natural transformation
\[
	t_* \From f_* \Rightarrow g_*
\]
between the previously defined pushforward
of the lax monoidal functors \( f \) and \( g \)
as follows: it assigns to any category
\( (\Category C, [-,-], \{m^n\}) \) enriched 
over \( \Category U \), a \( \Category V \)\=/enriched
functor
\[
    t_{*\,\Category C} \From f_*\Category C
    \to g_*\Category C
\]
whose underlying functor is the identity functor,
and which is equipped with natural transformations
with components
\[
	(t_{*\,\Category C})_{x,y} \coloneqq
	t_{\EnrichedHom{x}{y}} \From
	\EnrichedHom{x}{y}_f \to \EnrichedHom{x}{y}_g
\]
for any pairs of objects \( x, y \) of \( \Category C \).
Indeed, the above does define a \( \Category V \)\=/enriched
functor due to
the monoidality (and naturality) of \( t \).
Moreover, upon using the previous definition,
one can check that \( t_\ast \) satisfies both
the \( 1 \)- and \( 2 \)- naturality conditions.

\begin{proposition}
	The operation which sends any oplax monoidal category
	\( \Category V \) to the \( 2 \)\=/category \( \Cat_{\Category V} \),
	and any lax monoidal functor and monoidal natural transformation
	to their pushforwards defines a \( 2 \)\=/functor
	\[
		\begin{tikzcd}[sep=large]
			\Oplax
			\ar[r, "\Category V \mapsto \Cat_{\Category V}"]
			& \TwoCat
		\end{tikzcd}
	\]
	from the \( 2 \)\=/category of oplax monoidal categories
	to the one of \( 2 \)\=/categories.
\end{proposition}
\begin{proof}
	As we have previously proven that
	the pushforward of lax monoidal functors
	and monoidal natural transformations are respectively
	\( 2 \)\=/functors and \( 2 \)\=/natural transformations,
	what is left to show is that the pushforward preserves
	the composition of \( 1 \)\=/cells, preserves the vertical
	and horizontal compositions of \( 2 \)\=/cells and maps 
	identities to identities. Doing so simply amounts to
	re-writing the various definitions detailed above and
	checking that they are compatible with the compositions
	in \( \Oplax \) and in \( \TwoCat \), and does not
	require the use of any particular property other than
	the associativity of the composition of morphisms in
	a category.
\end{proof}

\subsection{External product of enriched categories}
\label{sec:external_product}

The category \( \Oplax \) admits products.
If \( \Category U^\psi \) and \( \Category V^\oplax \) are two
oplax monoidal categories, their product is given by the
category \( \Category U \times \Category V \) with
structural maps
\[
	\begin{tikzcd}
		(\Category U \times \Category V)^{n+1}
		\rar["\IsCanonicallyIsomorphicTo"]
			& \Category U^{n+1} \times \Category V^{n+1}
			\rar["\psi^n \times \oplax^n"]
				&[5ex] \Category U \times \Category V
	\end{tikzcd}
\]
for every \( n \geq -1 \) and decomposition natural transformations
defined similarly.

\begin{definition}
	The external product is the 2-functor
	\[
		\begin{tikzcd}[sep=large]
			\Cat_{\Category U} \times \Cat_{\Category V}
			\ar[r, "\boxtimes"]
			& \Cat_{\Category U \times \Category V}
		\end{tikzcd}
	\]
	taking as input a \( \Category U \)-enriched category
	\( (\Category B, [-,-], \{l^n\}) \) and
	a \( \Category V \)-enriched category
	\( (\Category C, \langle-,-\rangle, \{m^n\}) \) and returning
	the \( \Category U \times \Category V \)-enriched category
	\( \Category B \boxtimes \Category C \) defined as
	the product \( \Category B \times \Category C \) with
	enrichment bifunctor
	\[
		\begin{tikzcd}
			(\Category B\Op \times \Category C\Op)
			\times (\Category B \times \Category C)
			\ar[r, "\IsCanonicallyIsomorphicTo"]
			& (\Category B\Op \times \Category B)
			\times (\Category C\Op \times \Category C)
			\ar[r, "{[-,-] \times \langle-,-\rangle}"]
			&[7ex] \Category B \times \Category C
		\end{tikzcd}
	\]
	and composition maps \( \{l^n \times m^n\} \).

	The external product of a \( \Category U \)\=/enriched
	functor (resp. natural transformation)
	with a \( \Category V \)\=/enriched functor (resp. natural
	transformation) is defined similarly.
\end{definition}

\begin{proposition}\label{prop:pushforward_monoidal}
	The external product of enriched categories endows
	\[
		\begin{tikzcd}[sep=large]
			\Oplax
			\ar[r, "\Category V \mapsto \Cat_{\Category V}"]
			& \TwoCat
		\end{tikzcd}
	\]
	with the structure of a lax monoidal \( 2 \)\=/functor.
\end{proposition}

\begin{proof}
	The above proposition can be seen as a direct generalisation 
	in the oplax setting
	of the standard result
	according to which the pushforward \( 2 \)\=/functor
	\[
		\begin{tikzcd}[sep=large]
			\Mon
			\ar[r, "\Category V \mapsto \VCats"]
			&[2ex] \TwoCat
		\end{tikzcd}
	\]
	from the \( 2 \)\=/category
	\( \Mon \) of monoidal categories and lax monoidal functors
	to the \( 2 \)\=/category
	\( \TwoCat \) of \( 2 \)\=/categories is lax monoidal. 
	The proof in the oplax case is essentially
	the same as in the monoidal case for which
	we refer the interested reader 
	to the detailed proof in
	Cruttwell's PhD thesis
	\cite{phd:Cruttwell}.
\end{proof}

\subsection{Lax monoidal structure on \( \Cat_{\Category D} \),
when \( \Category D \) is lax-oplax duoidal}
\label{sec: D-enriched}

When \( \Category D^{\lax, \oplax} \) is a lax-oplax duoidal category,
the 2-category of \( \Category D^\oplax \)-enriched categories
\( \Cat_{\Category D} \) can be endowed with a lax monoidal
structure whose structural 2-functors
\[
	\begin{tikzcd}
		\times_\lax^n \From\
		\Cat_{\Category D} \times \dots \times \Cat_{\Category D}
		\ar[r, "\boxtimes^n"]
		& \Cat_{\Category D \times \dots \times \Category D}
		\ar[r, "(\lax^n)_\ast"]
		& \Cat_{\Category D}
	\end{tikzcd}
\]
are obtained by using the external tensor product of categories and
the pushforward of the lax structure of \( \Category D \).

\begin{theorem}
	\label{corollary:lax oplax duoidal}
	Let \( \Category D^{\lax,\oplax} \) be a lax-oplax 
	duoidal category. Then the lax structure \( \lambda \)
	induces a lax monoidal 2-category structure \( \times_\lambda \)
	on the 2-category
	\( \Cat_{\Category D} \) of categories enriched
	over \( \Category D^\oplax \)
	(\ie it becomes a lax monoid in \( \TwoCat \)).
\end{theorem}
\begin{proof}
	Lax-oplax duoidal categories can be identified with
	the lax monoids of the 2-category \( \Oplax \), meanwhile lax monoidal
	categories are the lax monoids of \( \TwoCat \).
	In the previous section we explained how the external
	tensor product endows the functor
	\( \Category V \mapsto \Cat_{\Category V} \) with a lax
	structure.

	Day and Street have shown that lax monoidal \( 2 \)\=/functors
	map lax monoids to lax monoids
	\cite{zbMATH01963527}, so that applying the pushforward
	\( 2 \)\=/functor to \( \Category D \) endows
	\( \Cat_{\Category D} \)
	with a structure of lax monoid in \( \TwoCat \)
	\ie \( \Cat_{\Category D} \) is a lax monoidal \( 2 \)\=/category.
\end{proof}

\begin{remark}
	In the case of categories enriched over \( \Category D^\oplax \)
	with only one object (\ie monoids in \( \Category D^\oplax \)),
	one recovers the fact that the category of monoids in
	\( \Category D^\oplax \) inherits a lax monoidal structure
	from \( \Category D^\lax \)
	\UnskipRef{prop:monoids=lax+comonoid=oplax}.
\end{remark}

\begin{remark}
	Note that the notion of lax-strong duoidal category
	identifies with the one of lax monoid in \( \Mon \). 
	Hence, the fact that the \( 2 \)\=/category
	of categories enriched
	over  a lax-strong duoidal category is a lax monoidal
	\( 2 \)\=/category can be readily obtained from the fact 
	that the pushforward \( 2 \)\=/functor
	\[
		\begin{tikzcd}[sep=large]
			\Mon
			\ar[r, "\Category V \mapsto \Cat_{\Category V}"]
			& \TwoCat
		\end{tikzcd}
	\]
	maps lax monoids to lax monoids. The alternative
	notion of strong-oplax duoidal categories, \ie 
	a lax-oplax duoidal category for which the lax
	monoidal structure is strong monoidal, can be 
	equivalently characterised
	as a pseudo-monoid in  \( \Oplax \). Hence, 
	from the fact that the pushforward \( 2 \)\=/functor
	\[
		\begin{tikzcd}[sep=large]
			\Oplax
			\ar[r, "\Category V \mapsto \Cat_{\Category V}"]
			& \TwoCat
		\end{tikzcd}
	\]
	maps pseudo-monoids to pseudo-monoids follow that
	the  \( 2 \)\=/category
	of categories enriched
	over  a strong-oplax duoidal category is strong monoidal.
\end{remark}

\section{Motivating example:
	\( (\Ring\Enveloping, \Ring\Enveloping) \)\=/bimodules}
\label{sec:Re-bimod}

Let \( \Ring \) be a ring and let
\( \Ring\Enveloping \coloneqq \Ring \otimes_\Integers \Ring\Op \) be
its enveloping ring.
In this section we shall endow the category of
\( (\Ring\Enveloping, \Ring\Enveloping) \)-bimodules with
a lax-strong duoidal structure whose bimonoids are the bialgebroids.

\subsection{Strong monoidal structure on
	\( (\Ring\Enveloping, \Ring\Enveloping) \)-bimodules}
The category of bimodules over the enveloping ring \( \Ring\Enveloping \)
can naturally be endowed with a tensor structure
\( \otimes_{\Ring\Enveloping} \) which can be lifted into a
strong monoidal structure
\SectionRef{sec:enrichment over monoidal categories}.

Let us recall that
(given a choice of strong monoidal structure
on the category of abelian groups)
one can describe the strong
monoidal structure on the category of \( \Ring\Enveloping \)-bimodules
using integrals
\cite{doi:10.1007/BF02685882, doi:10.2969/jmsj/02930459}
\cite[IX.6]{doi:10.1007/978-1-4757-4721-8}.
One can write
\[
	A \otimes_{\Ring\Enveloping} B \otimes_{\Ring\Enveloping} \dots
	\otimes_{\Ring\Enveloping} Z \coloneqq \int_{(r_1,s_1),\dots,(r_n,s_n)}
	A_{(r_1,s_1)} \otimes {_{(r_1, s_1)}}B_{(r_2,s_2)} \otimes
	\dots \otimes {_{(r_n, s_n)}}Z
\]
where 
\( (A, B, \dots, Z) \) represents any given sequence
of \( (\Ring\Enveloping,\Ring\Enveloping) \)\=/bimodules.
Here, the integral is taken over elements
\( (r,s) \in \Ring\Enveloping \) whose position as a subscript of
a bimodule indicates whether the left or the right
\( \Ring\Enveloping \)-action is involved in the computation of the
integral.

\begin{remark}
	\label{rmk: enrichment over Re-bimodules}
	Given an \( \Ring\Enveloping \)-ring \( A \),
	\ie a ring \( A \) together with a morphism
	of rings \( \Ring\Enveloping \to A \), every \( A \)-module
	admits a structure of \( (\Ring, \Ring) \)-bimodule.
	Thus, the category of \( A \)\=/modules can be naturally
	enriched over the category of
	\( (\Ring\Enveloping,\Ring\Enveloping) \)\=/bimodules.
	For each pair of \( A \)-modules \( (M, N) \),
	the enrichment bifunctor is given by
	\( \Hom{\Integers}{M}{N} \), with
	\( (\Ring\Enveloping,\Ring\Enveloping) \)\=/bimodule
	structure
	\[
		\big[(r, r') \cdot f \cdot (s, s')\big](m)
		\coloneqq r \cdot f(s \cdot m \cdot s') \cdot r'
	\]
	for every pairs \( (r,r'),(s,s') \in \Ring\Enveloping \),
	every \( f \From M \to N \) 
	and every element \( m \in M \).

	Notice that in general this enrichment is not normal:
	the set of weak morphisms between two \( A \)-modules
	\( M \) and \( N \) is in bijection with
	the set \( \Hom{\Ring\Enveloping} M N \).
\end{remark}

\subsection{Lax monoidal structure on
	\( (\Ring\Enveloping, \Ring\Enveloping) \)-bimodules}
On top of the (strong) monoidal structure \( \otimes_{\Ring\Enveloping} \),
the category of \( \Ring\Enveloping \)\=/bimodules
can also be endowed with
a strictly normal lax monoidal structure, the restricted tensor product
\( \times_\Ring \subset \otimes_R \)
of Sweedler
\cite{doi:10.1007/BF02685882}
and Takeuchi \cite{doi:10.2969/jmsj/02930459}.

The restricted tensor product is given by the formula
\[
	M \times_\Ring N
	\coloneqq \int^s M_{\bar s} \otimes_\Ring N_s
	\coloneqq \int^s\int_r
	{_{\bar r}}M_{\bar s} \otimes {_r}N_s
\]
where \( \bar r, \bar s \) denotes elements
of the opposite ring \( \Ring\Op \),
and \( \int^s M_{\bar s} \otimes_\Ring N_s \) is the cointegral
whose elements
are the combinations \( \sum_i m_i \otimes n_i \in M \otimes_R N \)
for which
\( \sum_i  m_i \cdot {\bar s} \otimes n_i
= \sum_i m_i \otimes n_i \cdot s \) for every \( s \in R \).
We shall extend this bifunctor into a fully-fledged lax monoidal
structure.

\begin{description}
	\item[\( n \)-ary tensors]
		Given
		a number of
		\( (\Ring\Enveloping, \Ring\Enveloping) \)\=/bimodules
		\( (A,B, \dots, Z) \),
		their restricted product shall be defined as
		\[
			A \times_\Ring B \times_\Ring \dots \times_\Ring Z
			\coloneqq \int\limits^{s_1,\dots,s_n}
			\int\limits_{r_1,\dots,r_n}
			\left({_{\bar r_1}}A_{\bar s_1}\right)
			\otimes \left({_{r_1,\bar r_2}}B_{s_1,\bar s_2}\right)
			\otimes \dots
			\otimes \left({_{r_n}}Z_{s_n}\right)
		\]
		which is a new
		\( (\Ring\Enveloping, \Ring\Enveloping) \)\=/bimodule,
		with left and right actions given by
		\[
			(r, s) \cdot
			(a \otimes b \otimes \dots \otimes y \otimes z)
			\cdot (r', s')
			= (r \cdot a \cdot r') \otimes b \otimes \dots
			\otimes y \otimes (\bar s \cdot z \cdot \bar s')\,,
		\]
		for any \( (r,s),(r',s') \in \Ring\Enveloping \), \( a \in A \),
		\( b \in B \), \( y \in Y \) and \( z \in Z \).
		Here as usual,
		\( A \otimes \cdots \otimes Z \) denotes a chosen
		strong monoidal lift of the tensor product of abelian groups;
	\item[Unit]
		The unit object is:
		\[
			\times_\Ring^{-1}
			\coloneqq \End(\Ring)
			\coloneqq \Hom \Integers R R
		\]
		wherein the \( (\Ring\Enveloping,\Ring\Enveloping) \)\=/bimodule
		structure of \( \mathrm{End}(\Ring) \)
		is the one described in the previous subsection
		\UnskipRef{rmk: enrichment over Re-bimodules};
	\item[Associators]
		When \( q \neq -1 \), the natural transformations
		\( \times^p_\Ring(-, \times^q_\Ring, -) \implies
		\times^{p+q}_\Ring \) are given by
		\[
			\begin{tikzcd}
				\displaystyle
				\int\limits^{u_1,\dots,u_p}
				\!\!\!
				\!\!\! 
				\!\!\!
				\int\limits_{t_1,\dots,t_p}
				{_{\bar t_1}}A_{\bar u_1} \otimes \dots \otimes\,
				{_{t_i, \bar t_{i+1}}}\big(\int\limits^{s_1,\dots,s_q}
				\!\!\!
				\!\!\! 
				\!\!\!
				\int\limits_{r_1,\dots,r_q} {_{\bar r_1}}M_{\bar s_1}
				\otimes \dots \otimes
				{_{r_q}}N_{s_q}\big)_{u_i, \bar u_{i+1}}
				\otimes \dots \otimes {_{t_p}}Z_{u_p}
				\dar["\text{dist.}"]
				\\
				\displaystyle
				\int\limits^{u_1,\dots,u_p} \int\limits_{t_1,\dots,t_p}
				\int\limits^{s_1,\dots,s_q}
				\int\limits_{r_1,\dots,r_q} 
				{_{\bar t_1}}A_{\bar u_1}
				\otimes \dots
				\otimes {_{t_p}}Z_{u_p}
				\dar["\text{dist.}"]
				\\
				\displaystyle
				\int\limits^{u_1,\dots,u_p}
				\int\limits^{s_1,\dots,s_q}
				\int\limits_{t_1,\dots,t_p}
				\int\limits_{r_1,\dots,r_q} 
				{_{\bar t_1}}A_{\bar u_1}
				\otimes \dots
				\otimes {_{t_p}}Z_{u_p}
			\end{tikzcd}
		\]
		where \( (A, \dots, Z) \) is any sequence of
		\( (p+q+1) \)
		\( (\Ring\Enveloping, \Ring\Enveloping) \)-bimodules
		and where \( i \) denotes the position of the insertion of
		\( \times^q_\Ring \) in \( \times^p_\Ring \).
		The first map amounts to the commutativity of integrals with
		the tensor product of abelian groups together with the
		distributivity of cointegrals with the same tensor.
		Lastly the second arrow is the usual distributivity
		\( \int^\ast \int_\ast \implies \int_\ast \int^\ast \)
		of cointegrals over integrals;

	\item[Unitors]
		If \( q=-1 \), we define:
		\begin{align*}
			\End(\Ring) \times_\Ring A \times_\Ring \dots
			\times_\Ring Z
			&\longrightarrow
			A \times_\Ring \dots \times_\Ring Z
			\\
			\textstyle
			\sum_i \phi_i
			\otimes a_i \otimes \cdots
			\otimes z_i
			&\longmapsto \textstyle\sum_i
			\phi_i(1_\Ring) \cdot a_i \otimes \cdots \otimes z_i
		\end{align*}
		\begin{align*}
			A \times_\Ring \dots
			\times_\Ring Z \times_\Ring \End(\Ring)
			&\longrightarrow
			A \times_\Ring \dots \times_\Ring Z
			\\
			\textstyle
			\sum_i
			a_i \otimes \cdots \otimes z_i \otimes \phi_i
			&\longmapsto \textstyle\sum_i
			a_i \otimes \cdots \otimes \overline{\phi_i(1_\Ring)}
			\cdot z_i
		\end{align*}
		and
		\[
			A \times_\Ring \dots \times_\Ring M \times_\Ring
			\End(\Ring) \times_\Ring N \times_\Ring \dots
			\times_\Ring Z
			\longrightarrow A \times_\Ring \dots \times_\Ring M \times_\Ring N
			\times_\Ring \dots \times_\Ring Z
		\]
		by
		\begin{align*}
			\textstyle
			\sum_i \dots \otimes m_i \otimes \phi_i
			\otimes n_i \otimes \cdots &
			\longmapsto \textstyle\sum_i \cdots \otimes m_i
			\otimes \big(\phi_i(1_\Ring) \cdot n_i\big) \otimes \cdots
			\\
			& \qquad = \textstyle\sum_i
			\cdots
			\otimes \big(\overline{\phi_i(1_\Ring)} \cdot m_i \big)
			\otimes n_i \otimes \cdots
		\end{align*}
		for any \( (\Ring\Enveloping,\Ring\Enveloping) \)\=/bimodules
		\( A, \dots, M, N, \dots, Z \).
		The proof that these maps are
		\( (\Ring\Enveloping,\Ring\Enveloping) \)-linear is
		similar to the low arity case
		\cite[2.2]{doi:10.2969/jmsj/02930459}.
\end{description}

\begin{proposition}
	The structure described above---which
	extends the restricted tensor product
	\( \times_R \)---is
	a (strictly normal) lax monoidal structure on the category
	of \( (\Ring\Enveloping,\Ring\Enveloping) \)\=/bimodules.
\end{proposition}

\begin{proof}
	Sequential and parallel composition axioms follow straightforwardly
	from integral calculus.
\end{proof}

\subsection{Lax-strong duoidal structure on
	\( (\Ring\Enveloping, \Ring\Enveloping) \)-bimodules}

We are now going to describe how the strong monoidal structure
\( \otimes_{\Ring\Enveloping} \) and the lax monoidal structure
\( \times_\Ring \) form a lax-strong duoidal structure on the
category of \( (\Ring\Enveloping, \Ring\Enveloping) \)-bimodules.
For this, we know that it is enough to describe two sequences
of maps \( \{\chi^p\}_{p \geq -1} \) and \( \{w^p\}_{p \geq -1} \)
\UnskipRef{prop: lax-strong}.
The difficulty here is in checking that the maps are well defined,
the additivity axioms are then straightforward to check.

\begin{description}
	\item[\( (w^p) \)]
		Since \( \Ring\Enveloping \) is generated by \( (1,1) \) as a
		\( (\Ring\Enveloping, \Ring\Enveloping) \)-bimodule,
		in order to describe the coproducts
		\( w^p \From \Ring\Enveloping
		\to \times_\Ring^{p} \Ring\Enveloping \),
		it is enough to give the image of \( (1, 1) \).
		The only constraint being that the image of \( (1, 1) \) must
		be an element on which the left and the right action of
		\( \Ring\Enveloping \) agree.
		
		The map \( w^{-1} \From \Ring\Enveloping \to \End(\Ring) \)
		sends \( (1, 1) \) to the identity endomorphism
		and \( w^p \) sends \( (1, 1) \in \Ring\Enveloping \) to
		\( (1, 1) \otimes \dots \otimes (1,1)
		\in \Ring\Enveloping \times_\Ring \dots
		\times_\Ring \Ring\Enveloping \)
		for every \( n \geq 0 \).
		
	\item[\( (\chi^p) \)]
		We shall follow the example of
		Takeuchi who described the map \( \chi^1 \)
		\cite[1.12]{doi:10.2969/jmsj/02930459}.
		For the case \( p = -1 \), the map
		\[
			\End(R) \otimes_{\Ring\Enveloping} \End(R)
			\longrightarrow \End(R)
		\]
		is given by composition of endomorphisms.
		The case \( p = 0 \) is the identity transformation, so let
		\( p \geq 1 \).
		Let \( \bm M \coloneqq (M^{(0)}, \dots, M^{(p)}) \)
		and \( \bm N \coloneqq (N^{(0)}, \dots, N^{(p)}) \)
		be two families of
		\( (\Ring\Enveloping, \Ring\Enveloping) \)-bimodules.
		Consider the composite
		\( (\Ring\Enveloping, \Ring\Enveloping) \)-bimodule map
		\[
			\begin{tikzcd}
				\left(\times_\Ring^p \bm M\right)
				\otimes_{\Ring\Enveloping} \left(\otimes^p \bm N\right)
				\dar[hook]
				\\
				\displaystyle\int_{a, b, x_1, \dots, x_p}
					{_{\bar x_1}}M^{(0)}_a \otimes
					{_{x_1, \bar x_2}}M^{(1)} \otimes \dots \otimes
					{_{x_p}}M^{(p)}_{\bar b} \otimes
					{_a}N^{(0)}
					\otimes
					N^{(1)} \otimes \dots \otimes {_{\bar b}}N^{(p)}
				\dar["\text{twist}"]
				\\
					\displaystyle\int_{x_1, \dots, x_p}
					\left({_{\bar x_1}}M^{(0)} \otimes_\Ring
					N^{(0)}\right) \otimes
					\left({_{x_1, \bar x_2}}M^{(1)}
						\otimes
					N^{(1)}\right)
					\otimes \dots \otimes
					\left({_{x_p}}M^{(p)}
					\otimes_\Ring N^{(p)}\right)
				\dar["\text{proj}"]
				\\
					\displaystyle\int_{x_1, \dots, x_p}
					\left({_{\bar x_1}}M^{(0)} \otimes_{\Ring\Enveloping}
					N^{(0)}\right) \otimes
					\left({_{x_1, \bar x_2}}M^{(1)}
						\otimes_{\Ring\Enveloping}
					N^{(1)}\right)
					\otimes \dots \otimes
					\left({_{x_p}}M^{(p)}
					\otimes_{\Ring\Enveloping} N^{(p)}\right)
			\end{tikzcd}
		\]
		and let us denote it by \( \phi^p \). Note that
		the middle step of this definition (twist) implicitly
		uses the fact that the symmetric structure on
		a monoidal category can be lifted to an action of 
		the symmetric group on a strong monoidal category
		\cite{Lift}.
\end{description}

\begin{lemma}
	The map \( \phi^p \) induces a map
	\[
		\begin{tikzcd}
			\left( \times_\Ring^p \bm M\right)
			\otimes_{\Ring\Enveloping}
			\left(\times_\Ring^p \bm N\right)
			\rar["\chi^p"]
				& \times^p_\Ring \left(\bm M \otimes_{\Ring\Enveloping}
				\bm N\right)
		\end{tikzcd}
	\]
	of \( (\Ring\Enveloping,\Ring\Enveloping) \)\=/bimodules.
\end{lemma}

\begin{proof}
	Since tensor products commute with integrals,
	one can swap the integral out of the formula
	for the left hand side of the map \( \chi^p \)
	that we are going to define.
	Let \( \sum_i m_i^{(0)} \otimes \dots \otimes m_i^{(p)} \)
	be an element of \( \times^p_\Ring \bm M \),
	and let
	\( n^{(0)} \otimes \dots \otimes n^{(p)} \in \otimes^p \bm N \).
	Then for every \( 0 \leq j \leq p-1 \) and every \( c \in \Ring \),
	\begin{align*}
		\dots \otimes m_i^{(j)} \otimes \bar c n^{(j)}
		\otimes m_i^{(j+1)} \otimes n^{(j+1)}
		\otimes \cdots
		&=
		\dots \otimes m_i^{(j)}\bar c \otimes n^{(j)}
		\otimes m_i^{(j+1)} \otimes n^{(j+1)}
		\otimes \cdots
		\\
		&=
		\dots \otimes
		m_i^{(j)} \otimes n^{(j)}
		\otimes m_i^{(j+1)} c \otimes n^{(j+1)} \otimes \cdots
		\\
		&=
		\dots \otimes
		m_i^{(j)} \otimes n^{(j)}
		\otimes m_i^{(j+1)} \otimes c n^{(j+1)} \otimes \cdots
	\end{align*}
	in the target of \( \phi^p \),
	where summation symbols are omitted.
	This shows that \( \phi^p \) factors through the quotient
	given by integration over \( c_1, \dots, c_p \) and that
	the map
	\[
		\begin{tikzcd}
			\left( \times_\Ring^p \bm M\right)
			\otimes_{\Ring\Enveloping}
			\left(\otimes_\Ring^p \bm N\right)
			\rar
				& \otimes_\Ring^p\left(\bm M \otimes_{\Ring\Enveloping}
				\bm N\right)
		\end{tikzcd}\,.
	\]
	induced by \( \phi^p \) is well defined. Finally,
	the map \( \chi^p \) can be obtained as the composition
	\[
		\begin{tikzcd}
			\
			\left( \times_\Ring^p \bm M\right)
			\otimes_{\Ring\Enveloping}
			\left(\int^{\bd}\otimes_\Ring^p \bm N\right)
			\rar["\text{distr.}"]
				& \displaystyle\int^{\bd}
				\left(\times^p_\Ring \bm M\right)
				\otimes_{\Ring\Enveloping}
				\left(\otimes_\Ring^p \bm N\right)
				\rar
				& \displaystyle\int^{\bd}
				\otimes^p_\Ring \left(\bm M \otimes_{\Ring\Enveloping}
				\bm N\right)
		\end{tikzcd}
	\]
	involving the canonical
	distributivity of cointegrals over functors.
\end{proof}

\begin{theorem}
	\label{theorem: lax-strong duoidal structure on Re-bimodules}
	The category of \( (\Ring\Enveloping,\Ring\Enveloping) \)\=/bimodules
	endowed with the normal lax structure \( \times_\Ring \),
	the strong structure
	\( \otimes_{\Ring\Enveloping} \)
	and the maps \( \{w^p\}_{p \geq -1} \)
	and \( \{\chi^p\}_{p \geq -1} \)
	is a lax-strong duoidal category.
\end{theorem}
\begin{proof}
	The fact that \( \{w^p\}_{p \geq -1} \) defines
	a comonoid structure follows straightforwardly
	from its definition (as repeated insertion of
	the unit \( (1,1) \) of \( \Ring\Enveloping \)).
	Similarly, the unitality condition for the pair
	\( (\chi^p, w^p) \) to define a lax monoidal structure
	on the functor \( \times_\Ring^p \) with respect
	to \( \otimes_{\Ring\Enveloping} \) follows simply
	from their definition. The associativity condition
	follows from the property of distributivity
	and commutativity of co/integrals, as well as
	the fact that \( \Ab^\otimes \) is symmetric
	monoidal and hence its lift as a strong monoidal
	category is equipped with an action of the symmetric
	group \cite{Lift}. Finally, the additivity condition
	for the collection \( \{\chi^p\,\tau^{1,p}\}_{p \geq -1} \)
	to define a lax monoidal structure on the functor
	\( \otimes_{\Ring\Enveloping} \) with respect to
	\( \{\times_\Ring^p\}_{n \geq -1} \) is verified
	for the same reasons, while the unitality condition
	is trivial, as both \( \times_\Ring^0 \) and
	\( \chi^0 \) are identities.
\end{proof}

\begin{corollary}
	The tensor product \( \otimes_{\Ring\Enveloping} \) endows
	the category of \( \times_R \)-comonoids with a strong monoidal
	structure.
\end{corollary}

\begin{corollary}
	The restricted tensor product \( \times_\Ring \) endows
	the category of \( \Ring\Enveloping \)-rings
	(the \( \otimes_{\Ring\Enveloping} \)-monoids) with
	a strictly normal lax monoidal structure.
\end{corollary}

\begin{remark}
	This last corollary was already proven by Day and Street
	\cite[4.1]{doi:10.1090/fic/043/08}.
\end{remark}

\subsection{Bialgebroids}

\begin{definition}[(Bialgebroid)
\cite{doi:10.2969/jmsj/02930459,doi:10.1023/A:1008608028634}]
	A bialgebroid over a ring \( \Ring \) consists in
	a \( (\Ring\Enveloping,\Ring\Enveloping) \)\=/bimodule
	\( A \) endowed with:
	\begin{itemize}
		\item Two morphisms
			of \( (\Ring\Enveloping,\Ring\Enveloping) \)\=/bimodules 
			\[
				\mu_A\From A \otimes_{\Ring\Enveloping}A
				\longrightarrow A 
				\qand \eta_A\From \Ring\Enveloping
				\longrightarrow A
			\]
			making the following associativity and unitality diagrams
			\[
				\begin{tikzcd}[column sep = tiny]
					(A\otimes_{\Ring\Enveloping} A)
					\otimes_{\Ring\Enveloping}A
					\arrow[rr, "\IsIsomorphicTo"]
					\arrow[d,
					"\mu_A\otimes_{\Ring\Enveloping}-"]
					&
					&
					A\otimes_{\Ring\Enveloping}( A
					\otimes_{\Ring\Enveloping}A)
					\arrow[d,"-\otimes_{\Ring\Enveloping}\mu_A"]
					\\
					A\otimes_{\Ring\Enveloping} A
					\arrow[rd,"\mu_A",swap]
					&
					&
					A\otimes_{\Ring\Enveloping} A
					\arrow[ld,"\mu_A"] \\
					& A &
				\end{tikzcd}
				\begin{tikzcd}
					\Ring\Enveloping \otimes_{\Ring\Enveloping} A
					\ar[dr,
					"\eta_A \otimes_{\Ring\Enveloping} -" description]
					\ar[rdd, bend right, "\IsIsomorphicTo" description]
					&[-5ex] &[-5ex]
					A \otimes_{\Ring\Enveloping} \Ring\Enveloping
					\ar[dl,
					"- \otimes_{\Ring\Enveloping} \eta_A" description]
					\ar[ldd, bend left, "\IsIsomorphicTo" description]
					\\
					& A \otimes_{\Ring\Enveloping} A
					\dar["\mu_A"] &
					\\
					& A &
				\end{tikzcd}
			\]
			commute;
		\item Two morphisms of
			\( (\Ring\Enveloping,\Ring\Enveloping) \)\=/bimodules
			\[
				\delta_A
				\From  A
				\longrightarrow A \times_\Ring A
				\qand
				\epsilon_A\From A
				\longrightarrow \End(\Ring)
			\]
			making the following coassociativity and counitality
			diagrams
			
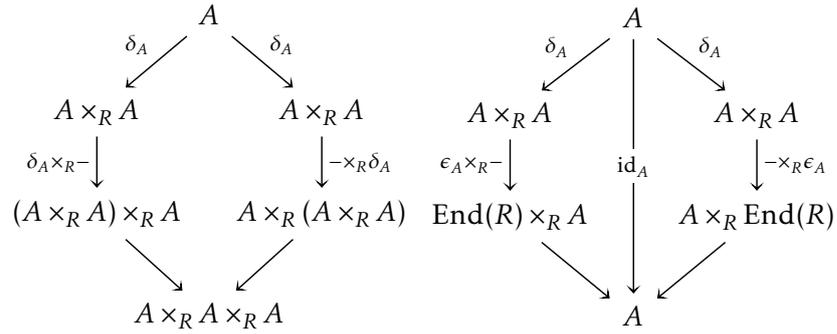
\[
				\begin{tikzcd}[column sep = tiny]
					&[-7ex]
					A
					\arrow[dl,"\delta_A",swap]
					\arrow[dr,"\delta_A"]
						&[-7ex]
					\\
					A \times_\Ring A
					\arrow[d,
					"{\delta_A \times_\Ring -}"{left}]
					& &
					A\times_\Ring A
					\arrow[d, "{-\times_\Ring\delta_A}"]
					\\
					(A\times_\Ring A)
					\times_R A
					\arrow[dr]
					&&A\times_\Ring
					(A\times_R A)
					\arrow[dl]
					\\
					& A\times_\Ring A
					\times_R A &
				\end{tikzcd}
				\begin{tikzcd}[column sep = small]
					&[-2ex] A
					\ar[ddd, "\id_A" description]
					\ar[dl,"\delta_A",swap]
					\ar[dr, "\delta_A"]
						&[-2ex]
					\\
					A \times_\Ring A
					\dar["\epsilon_A \times_\Ring -" swap]
					& & A \times_\Ring A
					\dar["- \times_\Ring \epsilon_A"]
					\\
					\End(R)\times_\Ring A
					\ar[dr]
					& & A \times_\Ring \End(\Ring)
					\ar[dl]
					\\
					& A
				\end{tikzcd}
			\]
			commute.
	\end{itemize}
	In addition, the following compatibility diagrams
	
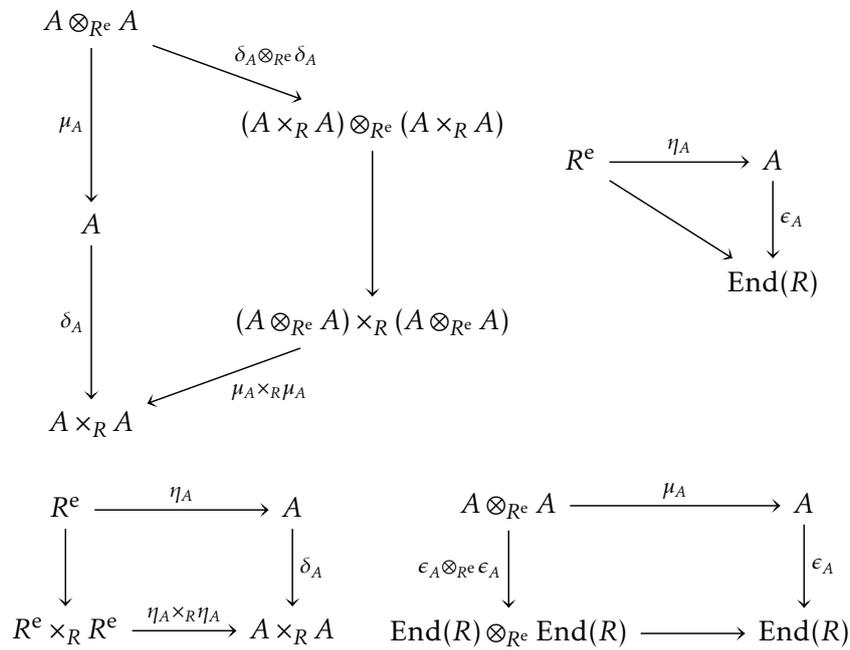
\[
		\begin{tikzcd}
			A \otimes_{\Ring\Enveloping} A
			\ar[dd, "\mu_A" swap]
			\ar[dr, "\delta_A \otimes_{\Ring\Enveloping} \delta_A"]
			\\
			& (A \times_\Ring A)
			\otimes_{\Ring\Enveloping}
			(A \times_\Ring A)
			\ar[dd]
			\\
			A
			\ar[dd, "\delta_A" swap]
			\\
			& (A \otimes_{\Ring\Enveloping} A)
			\times_\Ring
			(A \otimes_{\Ring\Enveloping} A)
			\ar[dl, "\mu_A \times_R \mu_A"]
			\\
			A \times_\Ring A
		\end{tikzcd}
		\quad
		\begin{tikzcd}[sep=large]
			\Ring\Enveloping 
			\arrow[dr]
			\arrow[r, "\eta_A"]
			& A 
			\arrow[d, "\epsilon_A"{right}] \\
			& \End(\Ring)
		\end{tikzcd}
	\]
	\[
		\begin{tikzcd}[sep=large]
			\Ring\Enveloping
			\arrow[d]
			\arrow[r, "\eta_A"]
			& A
			\arrow[d, "\delta_A"{right}]\\
			\Ring\Enveloping \times_\Ring \Ring\Enveloping
			\arrow[r, "\eta_A \times_\Ring \eta_A"]
			& A \times_\Ring A
		\end{tikzcd}
		\quad
		\begin{tikzcd}[sep=large]
			A \otimes_{\Ring\Enveloping} A
			\arrow[d, "\epsilon_A
			\otimes_{\Ring\Enveloping} \epsilon_A"{left}]
			\arrow[r, "\mu_A"]
			& A
			\arrow[d, "\epsilon_A"{right}] \\
			\End(\Ring) \otimes_{\Ring\Enveloping} \End(\Ring)
			\arrow[r]
			& \End(\Ring)
		\end{tikzcd}
	\]
	are also required to commute.
\end{definition}

\begin{theorem}
	Bialgebroids are the bimonoids of the lax-strong
	duoidal category
	\( \BimodulesOn
	{\Ring\Enveloping}^{\times_R, \otimes_{\Ring\Enveloping}} \).
\end{theorem}

\begin{proof}
	This comes from the fact that both \( \times_R \) and
	\( \otimes_{\Ring\Enveloping} \) are strictly normal,
	one then only needs to apply the reconstruction
	theorem for bimonoids
	\UnskipRef{Truncation of bimonoids}, modulo the
	lift of \( \otimes_{\Ring\Enveloping} \) from a monoidal
	structure to a strong one.
\end{proof}

\begin{remark}
	The fact that bialgebroids can be described as
	bimonoids in a lax-oplax duoidal category
	was considered in the book of Aguiar and Mahajan
	\cite[6.45]{doi:10.1090/crmm/029}
	and its corollary---that they can be described
	as \( \times_R \)-comonoids
	in the category
	of \( \Ring\Enveloping \)-rings---was
	proven by Day and Street
	\cite[4.3]{doi:10.1090/fic/043/08}.
\end{remark}

\section{Comparison with other enrichment theories}
\label{Sec: comparison of enrichment theories}

We conclude by displaying a few comments
on the relations between the previously introduced theory
of enrichment over an oplax monoidal category  
and the theories of enrichment over strong monoidal
categories, multicategories, skew-monoidal categories
and lax monoidal categories. The corresponding 
relations are schematised in the following figure
\UnskipRef{figure:Relations between enrichment theories}.

\begin{figure}
	\[
		\begin{tikzcd}
			\Mon \dar[hook]
				& \Strong \lar["\IsEquivalentTo" swap]
				\dar[hook]
				\rar[hook]
					& \Oplax
					\rar[hook]
						& \Multi
			\\
			\Skew
				& \Lax
		\end{tikzcd}
	\]
	\caption{Relations between enrichment theories}
	\label{figure:Relations between enrichment theories}
\end{figure}
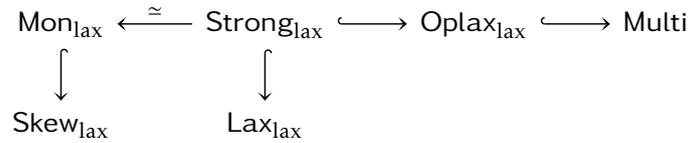

\subsection{Enrichment over monoidal categories}
\label{sec:enrichment over monoidal categories}

We claimed in the introduction that enriching over an oplax
monoidal base extends the notion of enrichment over
a monoidal base.
This is backed by the fact that the definitions
of a \( \Category V \)-category are similar
when \( \Category V \) is either endowed with
a monoidal or an oplax monoidal structure.

However, it is not straightforward to relate the 2-category
of monoidal categories and lax functors to the one
of oplax monoidal categories and lax functors.
Indeed, even though there is an obvious forgetful functor
\[
	\Strong \longrightarrow \Mon
\]
one needs a structural result to show
the equivalence.

The equivalence was shown by Leinster
\cite[3.2.2,3.2.3,3.2.4]{doi:10.1017/cbo9780511525896}
using coherence theorems for generalised monoidal
categories.
Another way of obtaining this result is to use, on the one hand,
the classical coherence theorem for monoidal categories
to show that each monoidal category admits an essentially
unique strong oplax lift given by any choice
of parenthesising and, on the other hand, the reconstruction
result for lax functors
\UnskipRef{prop:truncation_lax_functors}
given in the first
section.
This way, one can promote the equivalence result 
to a \( 2 \)\=/equivalence
by adding monoidal transformations as the \( 2 \)\=/cells
on each side.

This makes the theory of categories enriched over oplax monoidal
categories a direct extension of the classical
theory of categories enriched over a monoidal category, as claimed.

\subsection{Enrichment over multicategories}
We discussed
how the definition
of multicategories can be slightly generalised in order
for the embedding of oplax monoidal categories into
multicategories to become fully faithful
\SectionRef{sec:multi}. 

One can check that given a normal oplax monoidal category
\( \Category V^\oplax \), with associated multicategory
\( \Category M \),
the classical notion of category enriched in \( \Category M \)
\cite{arXiv:9901139}
coincides with the notion of \( \Category V \)-category
we outlined in the present paper.
Note
the fact that higher arity composition maps also
arise in a first definition of an enrichment over
a multicategory. However, this extended definition can also be
shown to be equivalent to an abridged one containing only a binary
composition map and an identity map \cite[2.2]{arXiv:9901139},
in complete parallel with what we observed for enrichment
over oplax monoidal categories.

\subsection{Enrichment over skew monoidal categories}
Enriching over skew-monoidal categories is in many
respects similar to enriching over oplax monoidal
categories. Recall that a skew monoidal category
\cite{doi:10.1016/j.aim.2012.06.027} is a category
endowed with a kind of monoidal structure in
which the associator, left and right unitors
may not be invertible.
As for oplax
monoidal categories, one can consider left normal
skew-monoidal categories which are skew-monoidal
categories wherein the left unitor is invertible.

Initially, the notion of enrichment over a skew-monoidal
category \( \Category V \),
called a \( \Category V \)\=/category,
was defined by Street
\cite{doi:10.1016/j.jpaa.2012.09.020}
in a way similar to that of enrichment over
a strong monoidal category, namely as a collection of objects
\( x, y, \dots \) together with hom-objects \( [x,y] \)
which belong to \( \Category V \) and endowed with
composition maps \( [y,z] \otimes [x,y] \to [x,z] \)
and identity maps \( 1_\otimes \to [x,x] \) obeying
some associativity and unitality conditions. However, 
one is faced with the difficulty that the enrichment
structure \( [-,-] \) defines a bifunctor only when
the enrichment base \( \Category V \) is left normal.
To bypass this problem, another notion of category
enriched over a skew monoidal base, referred to
as a skew \( \Category V \)\=/category, was given
by Campbell
\cite{doi:10.1007/s10485-017-9504-0}, consisting in
a category together with a bifunctor \( [-,-] \)
and requiring that the composition and identity maps
are extra/natural.

Although the notions of skew monoidal categories and oplax monoidal
categories both generalise the notion
of monoidal category, they are not comparable in general.
However when
a skew monoidal category is left normal,
we expect that its skew structure
can be lifted to a normal oplax structure,
as in the monoidal case.
Whenever \( \Category V^\oplax \) is an oplax monoidal
category admitting an underlying
skew monoidal category \( \Category V^{\otimes_\oplax} \),
the \( \Category V^{\otimes_\oplax} \)-categories of Street
are isomorphic to the \( \Category V^\oplax \)-categories and
the skew \( \Category V^{\otimes_\oplax} \)-categories of Campbell
are the categories enriched over \( \Category V^\oplax \).

\subsection{Enrichment over lax monoidal categories}
Using chiral definitions to the ones given above,
one can describe a theory of categories enriched over
lax monoidal categories.
This has been used for example by Batanin and Weber
\cite{doi:10.1007/s10485-008-9179-7} in order to study
higher operads.

Despite their apparent symmetry, there exist several important
differences between the theories of lax
and oplax enriched categories:
\begin{itemize}
	\item
		lax monoidal categories do not generate 
		multicategories and thus categories enriched over
		lax monoidal categories are not part of the general
		scheme of `categories enriched over multicategories';
	\item
		the \( 0 \)-part of the enrichment, that is the
		maps \( [x, y]_\lambda \to [x, y] \), actually encodes
		some information. Indeed, recall that a lax monoidal category
		possesses a unit natural transformation
		\( \IdentityFunctor \to \lax^0 \) which sends
		any object into its image by the endofunctor
		\( \lax^0 \), so that the previous map cannot
		be identified with the aforementioned unit
		and hence requiring its existence is non-trivial
		(unlike when the enrichment base is an oplax
		monoidal category, in which case 
		\( [x,y]_\oplax \to [x,y] \) can, and is,
		identified with the counit);
	\item
		it is not possible to reconstruct
		the composition maps
		\( \mu^n \From [x_0, \dots, x_n]_\lax \to [x_0, x_n] \)
		for \( n \geq 2 \) from the \( n \in \{-1,0,1\} \)
		ones. This stems from the fact that the additivity
		conditions \UnskipRef{def:cat_V-extended} are replaced by
		\[
			\begin{tikzcd}[column sep = huge]
				{[-,\dots,-]_{\lax^p(-,\lax^q,-)}}
				\dar
				\ar[r, "{\lax^p(-,\mu^q,-)}"]
				& {[-,\dots,-]_{\lax^p}} 
				\ar[d, "\mu^p"] \\
				{[-,\dots,-]_{\lax^{p+q}}}
				\ar[r, "\mu^{p+q}"]
				& {[-,-]}
			\end{tikzcd}
		\]
		which cannot be used as a definition of the higher
		arity composition maps in terms of the lower arity
		ones, due to the fact that the associator for lax
		monoidal categories (appearing as the left vertical
		arrow in the above diagram) is in the opposite direction
		compared to the associators of an oplax monoidal category.
\end{itemize}

\subsection*{Acknowledgements}
The authors are grateful to Gabriella B\"ohm for useful correspondence.

\input{oplaxenriched.bbl}
\end{document}

%% file: macros.tex
\newcommand{\Category}[1]{\mathcal #1}

\newcommand{\Hom}[3]{\mathrm{Hom}_{#1} \mleft(#2, #3\mright)}

\newcommand{\Sets}{\mathsf{Sets}}
\newcommand{\Ab}{\mathsf{Ab}}
\newcommand{\Cat}{\mathsf{Cat}}
\newcommand{\TwoCat}{\mathsf{2}\text{-}\mathsf{Cat}}

\newcommand{\IdentityFunctor}{\mathrm{Id}}

\newcommand{\Op}{^\mathrm{op}}
\newcommand{\CanonicalTwist}{\tauup}

\newcommand{\lax}{\lambda}
\newcommand{\oplax}{\omega}
\newcommand{\Oplax}{\mathsf{Oplax}_\mathrm{lax}}
\newcommand{\Lax}{\mathsf{Lax}_\mathrm{lax}}

\newcommand{\Skew}{\mathsf{Skew}_\mathrm{lax}}
\newcommand{\Strong}{\mathsf{Strong}_\mathrm{lax}}
\newcommand{\Mon}{\mathsf{Mon}_\mathrm{lax}}
\newcommand{\Multi}{\mathsf{Multi}}
\newcommand{\Ring}{R}

\newcommand{\EnrichedHom}[2]{\mleft[#1,#2\mright]}
\newcommand{\OplaxHom}[1]{\mleft[#1\mright]_\oplax}

\newcommand{\Objects}{\mathrm{Ob}}
\newcommand{\ReducedSequences}{\mathsf{Seq}^\ast}
\newcommand{\Insertion}{\vartriangleleft}
\newcommand{\ChangeEnrichment}{f}
\newcommand{\InducedHom}[1]{[#1]_\ChangeEnrichment}
\newcommand{\BimodulesOn}[1]{(#1, #1)\text{-}\mathsf{bimod}}
\newcommand{\Enveloping}{^\mathrm{e}}
\newcommand{\Underlying}[1]{\underline{#1}}
\newcommand{\Induction}{\mathfrak H}
\newcommand{\Collection}[1]{\boldsymbol{#1}}

\newcommand{\End}{\mathrm{End}}
\newcommand{\Vdelta}{\Category V\text{-}\deltaup}
\newcommand{\etaVCat}{\eta}
\newcommand{\muVCat}{\mu}
\newcommand{\thetaVCat}{\theta}
\newcommand{\nuVCat}{\nu}
\newcommand{\fVCat}{f}
\newcommand{\alphaVCat}{\alpha}
\newcommand{\hVCat}{h}

\newcommand{\ba}{\boldsymbol{a}}
\newcommand{\bb}{\boldsymbol{b}}

\newcommand{\bd}{\boldsymbol{d}}
\newcommand{\bx}{\boldsymbol{x}}

\newcommand{\ie}{\textit{i.e.}~}

\newcommand{\vs}{\textit{vs.}~}

\newcommand{\inout}[2]{
	\mathchoice
		{\raisebox{0.5pt}{\(\tbinom{#1}{#2}\)}}
		{\raisebox{0.5pt}{\(\tbinom{#1}{#2}\)}}
		{\raisebox{0.4pt}{\scalebox{0.73}{\(\tbinom{#1}{#2}\)}}}
		{}
	}

\newcommand{\UnskipRef}[1]{\unskip~\textnormal{[\ref{#1}]}}
\newcommand{\DoubleUnskipRef}[2]{\unskip~\textnormal{[\ref{#1}, \ref{#2}]}}
\newcommand{\SectionRef}[1]{\unskip~[\hyperref[#1]{\S\thinspace\ref{#1}}]}

\theoremstyle{plain}
\newtheorem{obvious lemma}[theorem]{Obvious lemma}

\usepackage{accents}

\newcommand{\uB}{\underaccent{\bar}{\Category B}}
\newcommand{\uC}{\underaccent{\bar}{\Category C}}

\newcommand{\uF}{\underaccent{\bar}{F}}

\newcommand{\Integral}[1]{\int_{#1}}

\newcommand{\VCats}{\mathcal V\text{-}\mathsf{cat}}
\newcommand{\UCats}{\mathcal U\text{-}\mathsf{cat}}
\newcommand{\VMods}{\mathcal V\text{-}\mathsf{mod}_\mathrm{lax}}

%% file: oplaxenriched.bbl
\providecommand{\href}[2]{#2}\begingroup\raggedright\endgroup

%% file: oplaxenriched.bbl
\begin{thebibliography}{10}

\bibitem{doi:10.4153/cjm-1965-076-0}
J.-M. Maranda, `Formal Categories',
  \href{http://dx.doi.org/10.4153/cjm-1965-076-0}{{\em Canadian Journal of
  Mathematics} {\bfseries 17} (1965) 758--801}.

\bibitem{zbMATH03220364}
J.~{Bénabou}, `{Catégories relatives}', {\em {Comptes Rendus Hebdomadaires
  des Séances de l'Académie des Sciences, Paris}} {\bfseries 260} (1965)
  3824--3827.

\bibitem{arXiv:9901139}
Tom Leinster, `Generalized Enrichment for Categories and Multicategories',
  1999.
\newblock \href{http://arxiv.org/abs/math/9901139}{{\color{arXiv}\ttfamily
  math/9901139 [math.CT]}}.

\bibitem{doi:10.1007/s10485-017-9504-0}
Alexander Campbell, `Skew-Enriched Categories',
  \href{http://dx.doi.org/10.1007/s10485-017-9504-0}{{\em Applied Categorical
  Structures} {\bfseries 26} (2018) 597--615},
  \href{http://arxiv.org/abs/1709.01222}{{\color{arXiv}\ttfamily
  arXiv:1709.01222 [math.CT]}}.

\bibitem{zbMATH01687308}
G.~{Janelidze} and G.~M. {Kelly}, `A note on actions of a monoidal category',
  {\em Theory and Applications of Categories} {\bfseries 9} (2001) 61--91.

\bibitem{zbMATH01963527}
Brian Day and Ross Street, \href{http://dx.doi.org/10.1090/conm/318/05545}{`Lax
  monoids, pseudo-operads, and convolution.',} in {\em Diagrammatic morphisms
  and applications. AMS special session on diagrammatic morphisms in algebra,
  category theory, and topology, San Francisco, CA, USA, October 21--22, 2000},
  pp.~75--96.
\newblock Providence, RI: American Mathematical Society (AMS), 2003.

\bibitem{MR2177746}
G.~M. Kelly, `On the operads of J. P. May', {\em Reprints in Theory and
  Applications of Categories} no.~13, (2005) 1--13.

\bibitem{doi:10.1007/s40062-012-0007-2}
Michael Ching, `A note on the composition product of symmetric sequences',
  \href{http://dx.doi.org/10.1007/s40062-012-0007-2}{{\em Journal of Homotopy
  and Related Structures} {\bfseries 7} no.~2, (Apr., 2012) 237--254},
  \href{http://arxiv.org/abs/math/0510490}{{\color{arXiv}\ttfamily
  arXiv:math/0510490 [math.CT]}}.

\bibitem{doi:10.2969/jmsj/02930459}
Mitsuhiro Takeuchi, `{Groups of algebras over $A\otimes\overline{A}$}',
  \href{http://dx.doi.org/10.2969/jmsj/02930459}{{\em J. Math. Soc. Japan}
  {\bfseries 29} no.~3, (1977) 459--492}.

\bibitem{doi:10.1007/BF02685882}
Moss~E. Sweedler, `{Groups of simple algebras}',
  \href{http://dx.doi.org/10.1007/BF02685882}{{\em Publications Math\'ematiques
  de l'IH\'ES} {\bfseries 44} (1974) 79--189}.

\bibitem{doi:10.1090/conm/771}
Gabriella Böhm and Joost Vercruysse, `BiHom Hopf algebras viewed as Hopf
  monoids', \href{http://dx.doi.org/10.1090/conm/771}{{\em Contemporary
  Mathematics} {\bfseries 771} (2021) 1--42},
  \href{http://arxiv.org/abs/2003.08819}{{\color{arXiv}\ttfamily
  arXiv:2003.08819 [math.QA]}}.

\bibitem{On_Ext_and_exact_sequences}
Nobuo Yoneda, `{On Ext and exact sequences}', {\em Journal of the Faculty of
  Science, Imperial University of Tokyo} {\bfseries 8} (1960) 507--576.

\bibitem{doi:10.1023/A:1008608028634}
Peter Schauenburg, `{Bialgebras over noncommutative rings and a structure
  theorem for Hopf bimodules}',
  \href{http://dx.doi.org/10.1023/A:1008608028634}{{\em Applied Categorical
  Structures} {\bfseries 6} no.~2, (1998) 193--222}.

\bibitem{doi:10.1017/cbo9780511525896}
Tom Leinster, \href{http://dx.doi.org/10.1017/cbo9780511525896}{{\em Higher
  Operads, Higher Categories}}.
\newblock Cambridge University Press, 2004.
\newblock \href{http://arxiv.org/abs/math/0305049}{{\color{arXiv}\ttfamily
  arXiv:math/0305049 [math.CT]}}.

\bibitem{doi:10.1007/s10485-008-9179-7}
Michael Batanin and Mark Weber, `Algebras of Higher Operads as Enriched
  Categories', \href{http://dx.doi.org/10.1007/s10485-008-9179-7}{{\em Applied
  Categorical Structures} {\bfseries 19} no.~1, (Dec., 2008) 93--135},
  \href{http://arxiv.org/abs/0803.3594}{{\color{arXiv}\ttfamily arXiv:0803.3594
  [math.CT]}}.

\bibitem{doi:10.1007/s10485-017-9497-8}
Marcelo Aguiar, Mariana Haim, and Ignacio~L{\'{o}}pez Franco, `Monads on Higher
  Monoidal Categories', \href{http://dx.doi.org/10.1007/s10485-017-9497-8}{{\em
  Applied Categorical Structures} {\bfseries 26} no.~3, (June, 2017) 413--458},
  \href{http://arxiv.org/abs/1701.03028}{{\color{arXiv}\ttfamily
  arXiv:1701.03028 [math.CT]}}.

\bibitem{doi:10.1090/crmm/029}
Marcelo Aguiar and Swapneel Mahajan,
  \href{http://dx.doi.org/10.1090/crmm/029}{{\em Monoidal Functors, Species and
  {H}opf Algebras}}, vol.~29 of {\em CRM Monograph Series}.
\newblock American Mathematical Society, Providence, RI, 2010.

\bibitem{Benabou}
Jean Bénabou, `Les distributeurs', Jan., 1973.
\newblock Teaching notes taken by Jean-Roger Roisin at the Université
  catholique de Louvain.

\bibitem{doi:10.1007/bf02924844}
F.~William Lawvere, `Metric spaces, generalized logic, and closed categories',
  \href{http://dx.doi.org/10.1007/bf02924844}{{\em Rendiconti del Seminario
  Matematico e Fisico di Milano} {\bfseries 43} no.~1, (Dec., 1973) 135--166}.

\bibitem{doi:10.1017/9781108778657}
Fosco Loregian, \href{http://dx.doi.org/10.1017/9781108778657}{{\em Coend
  calculus}}.
\newblock Cambridge University Press, 2021.
\newblock \href{http://arxiv.org/abs/1501.02503}{{\color{arXiv}\ttfamily
  arXiv:1501.02503 [math.CT]}}.

\bibitem{MR2177301}
{G. M.} Kelly, `Basic concepts of enriched category theory', {\em Reprints in
  Theory and Applications of Categories} no.~10, (2005) vi+137.

\bibitem{Cooperads}
Brice {Le Grignou} and Damien Lejay, `What is a cooperad?'.
\newblock In preparation.

\bibitem{doi:10.1007/s40062-019-00252-1}
Gabriel~C. Drummond-Cole, Joseph Hirsh, and Damien Lejay, `Representations are
  adjoint to endomorphisms',
  \href{http://dx.doi.org/10.1007/s40062-019-00252-1}{{\em Journal of Homotopy
  and Related Structures} {\bfseries 15} no.~2, (Dec., 2019) 377–393},
  \href{http://arxiv.org/abs/1904.06987}{{\color{arXiv}\ttfamily
  arXiv:1904.06987 [math.CT]}}.

\bibitem{doi:10.1007/978-3-642-99902-4_22}
Samuel Eilenberg and G.~Max Kelly,
  \href{http://dx.doi.org/10.1007/978-3-642-99902-4_22}{`Closed Categories',}
  in {\em Proceedings of the Conference on Categorical Algebra}, pp.~421--562.
\newblock Springer Berlin Heidelberg, 1966.

\bibitem{phd:Cruttwell}
Geoff~SH Cruttwell, {\em Normed spaces and the change of base for enriched
  categories}.
\newblock PhD thesis, Dalhousie University, 2008.

\bibitem{doi:10.1007/978-1-4757-4721-8}
Saunders Mac~Lane, \href{http://dx.doi.org/10.1007/978-1-4757-4721-8}{{\em
  Categories for the Working Mathematician}}, vol.~5 of {\em Graduate Texts in
  mathematics}.
\newblock Springer New York, New York, NY, 1978.

\bibitem{Lift}
Damien Lejay, `Equivariant unbiased monoidal categories'.
\newblock In preparation.

\bibitem{doi:10.1090/fic/043/08}
Brian Day and Ross Street, \href{http://dx.doi.org/10.1090/fic/043/08}{`Quantum
  categories, star autonomy, and quantum groupoids',} in {\em Galois theory,
  Hopf algebras, and semiabelian categories}, Bodo~Pareigis George~Janelidze
  and Walter Tholen, eds., pp.~187--225.
\newblock American Mathematical Society, July, 2004.
\newblock \href{http://arxiv.org/abs/math/0301209}{{\color{arXiv}\ttfamily
  arXiv:math/0301209 [math.CT]}}.

\bibitem{doi:10.1016/j.aim.2012.06.027}
K.~Szlach{\'a}nyi, `{Skew-monoidal categories and bialgebroids}',
  \href{http://dx.doi.org/10.1016/j.aim.2012.06.027}{{\em Advances in
  Mathematics} {\bfseries 231} (2012) 1694--1730},
  \href{http://arxiv.org/abs/1201.4981}{{\color{arXiv}\ttfamily arXiv:1201.4981
  [math.QA]}}.

\bibitem{doi:10.1016/j.jpaa.2012.09.020}
Ross Street, `Skew-closed categories',
  \href{http://dx.doi.org/10.1016/j.jpaa.2012.09.020}{{\em Journal of Pure and
  Applied Algebra} {\bfseries 217} (2013) 973--988},
  \href{http://arxiv.org/abs/1205.6522}{{\color{arXiv}\ttfamily arXiv:1205.6522
  [math.CT]}}.

\end{thebibliography}
